\documentclass{article}
\usepackage{amssymb}
\usepackage{amsmath}
\usepackage{makeidx}
\usepackage{amsfonts}
\usepackage[a4paper]{geometry}

\newtheorem{thm}{Theorem}

\newtheorem{cor}[thm]{Corollary}

\newtheorem{lem}[thm]{Lemma}

\newtheorem{prop}[thm]{Proposition}
\newtheorem{remark}[thm]{Remark}

\newenvironment{proof}[1][Proof]{\noindent\textbf{#1.} }{\ \rule{0.5em}{0.5em}}
\input{tcilatex}
\newtheorem{datAsm}{Assumption}

\newtheorem{densAsm}{Assumption}

\newtheorem{modAsm}{Assumption}

\newtheorem{rhoAsm}{Assumption}

\begin{document}

\title{Non-Parametric Maximum Likelihood Density Estimation and
Simulation-Based Minimum Distance Estimators\thanks{%
This paper is based on the doctoral thesis of the first author written under
the supervision of the second author. The authors are grateful to Richard
Nickl for many discussions and for helpful comments on the paper. }}
\author{Florian Gach and Benedikt M. P\"{o}tscher \\
Department of Statistics, University of Vienna}
\date{December 16, 2010\\
Revision: September 2011}
\maketitle

\begin{abstract}
Indirect inference estimators (i.e., simulation-based minimum distance
estimators) in a parametric model that are based on auxiliary non-parametric
maximum likelihood density estimators are shown to be asymptotically normal.
If the parametric model is correctly specified, it is furthermore shown that
the asymptotic variance-covariance matrix equals the inverse of the
Fisher-information matrix. These results are based on uniform-in-parameters
convergence rates and a uniform-in-parameters Donsker-type theorem for
non-parametric maximum likelihood density estimators.
\end{abstract}

\section{Introduction}

Suppose $X_{1},\ldots ,X_{n}$ are independent and identically distributed
(i.i.d.) random variables with law $\mathbb{P}$. Furthermore, we are given a 
\emph{parametric model} $\mathcal{P}_{\Theta }=\left\{ p_{\theta }:\theta
\in \Theta \right\} $ of probability density functions $p_{\theta }$ and $%
\Theta \subseteq \mathbb{R}^{m}$. Assume for the moment that $\mathcal{P}%
_{\Theta }$ is correctly specified and identifiable in the sense that there
is a unique $\theta _{0}\in \Theta $ such that $p_{\theta _{0}}$ is a
density of $\mathbb{P}$. A standard method of estimation of $\theta $ is
then the maximum likelihood method, which under appropriate regularity
conditions is known to lead to \emph{asymptotically efficient} estimators.
However, in a number of models, e.g., in econometrics and biostatistics, the
maximum likelihood method may not be feasible as no closed form expressions
for the densities $p_{\theta }$, and thus for the likelihood, are available.
For example, the data may be modeled by an equation of the form $%
X_{i}=g(\varepsilon _{i},\theta _{0})$ where $\varepsilon _{i}$ are
i.i.d.~with a known distribution but the implied parametric densities are
not analytically tractable because $g$ is complicated or $\varepsilon _{i}$
is high-dimensional. A similar problem naturally also occurs in the
estimation of dynamic nonlinear models; see Smith (1993), Gouri\'{e}roux,
Monfort and Renault (1993), Gallant and Tauchen (1996), Gouri\'{e}roux and
Monfort (1996), and Gallant and Long (1997) for several concrete examples.
This has led to the development of alternative estimation methods like the
so-called \emph{indirect inference method}, see the just mentioned
references as well as Jiang and Turnbull (2004). Ideally, these estimation
methods should also be asymptotically efficient. In our context these
methods can be described in a nutshell as follows:

\begin{enumerate}
\item Simulate a random sample $X_{1}(\theta ),...,X_{k}(\theta )$ of size $%
k $ from the density $p_{\theta }$ for $\theta \in \Theta $. [This is often
possible in the examples alluded to above, e.g., by perusing the equations
defining the model. Note that then only the disturbances $\varepsilon
_{1},\ldots ,\varepsilon _{k}$ have to be simulated once and $X_{i}(\theta )$
can be computed from $g(\varepsilon _{i},\theta )$ for any given $\theta $.]

\item Based on the simulated sample \textit{as well as} on the true data,
compute auxiliary estimators $\tilde{p}_{k}(\theta )$ and $\hat{p}_{n}$,
respectively, in a not necessarily correctly-specified but \emph{numerically
tractable auxiliary} model $\mathcal{M}^{aux}$. [For example, by maximum
likelihood if $\mathcal{M}^{aux}$ is finite-dimensional.]

\item With a suitable choice of a distance $\chi $ then estimate $\theta
_{0} $ by minimizing over $\Theta $ the objective function 
\begin{equation}
\mathbb{Q}_{n,k}(\theta ):=\chi (\hat{p}_{n},\tilde{p}_{k}(\theta )).
\label{mindist1}
\end{equation}
\end{enumerate}

In most of the indirect inference literature, the auxiliary model $\mathcal{M%
}^{aux}$ is assumed to be finite-dimensional indexed by a vector $\beta \in
B\subseteq \mathbb{R}^{l}$, say, and one then in fact minimizes a distance
between $\hat{\beta}_{n}$, the maximum likelihood estimator in the auxiliary
model computed from the original data, and $\tilde{\beta}_{k}(\theta )$, the
maximum likelihood estimator in the auxiliary model computed from the
simulated sample $X_{1}(\theta ),...,X_{k}(\theta )$. The resulting indirect
inference estimator can be shown to be consistent and asymptotically normal
(under standard regularity conditions, see Gouri\'{e}roux and Monfort
(1996)). However, the indirect inference estimator is asymptotically
efficient (in the sense of having the inverse of the Fisher-information
matrix as its asymptotic variance-covariance matrix) only if $\mathcal{M}%
^{aux}$ happens to be \textit{correctly specified}. This assumption is
certainly restrictive and often unnatural if $\mathcal{M}^{aux}$\ is of
fixed finite dimension. Therefore Gallant and Long (1997) suggested that
choosing $\mathcal{M}^{aux}$ with dimension increasing in sample size should
result in estimators that are asymptotically efficient, the idea being that
this essentially amounts to choosing an infinite-dimensional auxiliary model 
$\mathcal{M}^{aux}$, for which the assumption of correct specification is
much less restrictive. In particular, Gallant and Long (1997) set out to
study the case where the density estimators are based on non-parametric
maximum likelihood estimators over sieves spanned by Hermite-polynomials,
but their limiting result is only informative if the sieve dimension stays
bounded (so that efficiency of the estimator is only established if the true
density is a \textit{finite} linear combination of Hermite-polynomials)
bringing one back into the realm of finite-dimensional auxiliary models.

In the present paper we show in some generality that the suggestion in
Gallant and Long (1997) is indeed correct, namely that the indirect
inference estimator for $\theta $\ is asymptotically normal with the inverse
of the Fisher-information matrix as its asymptotic variance-covariance
matrix if the auxiliary estimators $\tilde{p}_{k}(\theta )$\ and $\hat{p}%
_{n} $\ in Step 2 are chosen to be non-parametric maximum likelihood (NPML)
estimators obtained from optimizing the non-parametric likelihood over
suitable bounded subsets of a Sobolev-space and if the size $k$ of the
simulated sample is of order larger than $n^{2}$. Furthermore, we show that
asymptotic normality persist even if the originally given model $\mathcal{P}%
_{\Theta }$\ is misspecified. [We do not explicitly consider sieved NPMLs,
although analogous results for such estimators are certainly possible. This
would require a uniform-in parameters extension of the results in Nickl
(2009), paralleling the extension of Nickl (2007) provided in the present
paper.]

We now comment on some related literature in the area of indirect inference:
Fermanian and Salani\'{e} (2004) propose a different procedure and establish
asymptotic efficiency of their estimators under several high-level
conditions, which, as they admit themselves, are very stringent. For
example, even in the simplest model they consider, they need to have
simulations of order $k\sim n^{6}$. Nickl and P\"{o}tscher (2010) consider
the case where $\tilde{p}_{k}(\theta )$ and $\hat{p}_{n}$ are not NPML
estimators but are spline projection estimators and they establish
asymptotic normality and asymptotic efficiency if the parametric model $%
\mathcal{P}_{\Theta }$ is correctly specified. In contrast to the present
paper, Nickl and P\"{o}tscher (2010) also analyze the case where $k$, the
size of the simulated sample, is not necessarily of order larger than $n^{2}$%
. We discuss this in more detail in Remark \ref{rem26} in Section \ref%
{Section: Indirect inference estimators}. There are also some other related
recent papers on this topic, Altissimo and Mele (2009) and Carrasco,
Chernov, Florens, and Ghysels (2007), whose proofs, however, we were not
able to follow.

In the present paper we shall use for $\chi $ the Fisher-metric, hence the
objective function defining the indirect inference estimator will be given
by 
\begin{equation*}
\mathbb{Q}_{n,k}(\theta )=\int (\hat{p}_{n}-\tilde{p}_{k}(\theta ))^{2}\hat{p%
}_{n}^{-1}.
\end{equation*}%
It transpires that the indirect inference estimators considered in the
present paper can be viewed as \emph{minimum distance estimators} with the
important (and nontrivial) modification that $p_{\theta }$ has been replaced
by an estimator $\tilde{p}_{k}(\theta )$ based on the simulated data. In
that sense our results can be viewed as an extension of Beran's (1977)
asymptotic efficiency result for classical minimum distance estimators to
the case of \emph{simulation-based minimum distance estimators}, the
simulation step introducing considerable additional complexity into the
proofs.

In order to establish the above mentioned results for the indirect inference
estimator a careful study of several aspects of the NPML-estimators $\tilde{p%
}_{k}(\theta )$ and $\hat{p}_{n}$ is required. In particular, it turns out
to be beneficial to establish the weak convergence of the stochastic process 
\begin{equation}
(\theta ,f)\mapsto \sqrt{k}\int (\tilde{p}_{k}(\theta )-p_{\theta })f
\label{null}
\end{equation}%
to a Gaussian process in $\ell ^{\infty }(\Theta \times \mathcal{F})$ where $%
\mathcal{F}$ is an appropriate class of functions. This result can be seen
to imply a uniform-in-$\theta $ version of a Donsker-type result for
NPML-estimators obtained recently by Nickl (2007). In the course of
establishing this weak convergence result it is also necessary to derive
rates of convergence for%
\begin{equation}
\sup_{\theta \in \Theta }\left\Vert \tilde{p}_{k}(\theta )-p_{\theta
}\right\Vert _{s,2}  \label{eins}
\end{equation}%
where the norm is a suitable Sobolev-norm.

The outline of the paper is as follows: After some preliminaries in Section %
\ref{Section: Notation and basic definitions}, we introduce the model and
assumptions in Section \ref{Section: The framework}. In Section \ref%
{Subsection: Existence-of-AML-estimators} we derive existence and uniqueness
of the NPML-estimator while rates of convergence as indicated in (\ref{eins}%
) are given in Section \ref{Rates}. Donsker-type theorems like (\ref{null})
are the subject of Section \ref{Section: A uniform Donsker theorem for
AML-estimators}. In contrast to Nickl (2007), we avoid an assumption that
requires all densities to be bounded away from zero in our results as far as
possible. Section \ref{Section: Indirect inference estimators} introduces
simulation-based minimum distance estimators (i.e., indirect inference
estimators) based on auxiliary NPML-estimators and establishes asymptotic
normality of these estimators even if the originally given parametric model $%
\mathcal{P}_{\Theta }$ is misspecified. If $\mathcal{P}_{\Theta }$ is
correctly specified, it is furthermore shown that the estimator is
asymptotically efficient in the sense that its asymptotic
variance-covariance matrix equals the inverse of the Fisher-information
matrix. Some proofs and technical results are collected in the appendices.

\section{Preliminaries and Notation \label{Section: Notation and basic
definitions}}

For $\Lambda $ a non-empty set and $f$ a real-valued function on $\Lambda $,
define $\Vert f\Vert _{\Lambda }=\sup_{x\in \Lambda }|f(x)|$ and let $\ell
^{\infty }(\Lambda )$ denote the Banach space of all bounded real-valued
functions on $\Lambda $, equipped with the sup-norm $\Vert \cdot \Vert
_{\Lambda }$. If $\mathcal{D}$ is a (non-empty) subset of $\ell ^{\infty
}(\Lambda )$ we shall write $(\mathcal{D},\Vert \cdot \Vert _{\Lambda })$ to
denote the metric space $\mathcal{D}$ with the induced metric $\Vert
f-g\Vert _{\Lambda }$. For $(\Lambda ,\mathcal{A})$ a (non-empty) measurable
space, let $\mathcal{L}^{0}(\Lambda ,\mathcal{A})$ denote the vector space
of all $\mathcal{A}$-measurable real-valued functions on $\Lambda $ and
define the Banach space $\mathsf{L}^{\infty }(\Lambda ,\mathcal{A})=\mathcal{%
L}^{0}(\Lambda ,\mathcal{A})\cap \ell ^{\infty }(\Lambda )$, again equipped
with the sup-norm. For $f\in \mathcal{L}^{0}(\Lambda ,\mathcal{A})$ and $\mu 
$ a non-negative measure on $(\Lambda ,\mathcal{A})$, define $\Vert f\Vert
_{2,\mu }=\left[ \int_{\Lambda }f^{2}d\mu \right] ^{1/2}$ and set $\mathcal{L%
}^{2}(\Lambda ,\mathcal{A},\mu )=\left\{ f\in \mathcal{L}^{0}(\Lambda ,%
\mathcal{A}):\Vert f\Vert _{2,\mu }<\infty \right\} $. For the measure space 
$(\Omega ,\mathcal{B}(\Omega ),\lambda )$, where $\Omega $ is a (non-empty)
measurable subset of the real line $\mathbb{R}$ with associated Borel $%
\sigma $-field $\mathcal{B}(\Omega )$ and where $\lambda $ is Lebesgue
measure, we shall simplify notation and write $\mathcal{L}^{0}(\Omega )$, $%
\mathcal{L}^{2}(\Omega )$, $\mathsf{L}^{\infty }(\Omega )$, and $\Vert \cdot
\Vert _{2}$ for $\mathcal{L}^{0}(\Omega ,\mathcal{B}(\Omega ))$, $\mathcal{L}%
^{2}(\Omega ,\mathcal{B}(\Omega ),\lambda )$, $\mathsf{L}^{\infty }(\Omega ,%
\mathcal{B}(\Omega ))$, and $\Vert \cdot \Vert _{2,\lambda }$, respectively.
Furthermore, we shall write a.e.~instead of $\lambda $-a.e. For any
(non-empty) metric space $(T,d)$, we denote by $\mathcal{B}(T,d)$, or simply 
$\mathcal{B}(T)$, its Borel $\sigma $-field and by $\mathsf{C}(T,d)$, or
simply $\mathsf{C}(T)$, the Banach space of all bounded, $d$-continuous
real-valued functions on $T$, equipped with the sup-norm.

We shall denote by $\Vert \cdot \Vert $ the $2$-norm on Euclidean space. For
two real-valued functions $f$ and $g$ on $(0,\infty )$, we shall write $%
f(\varepsilon )\lesssim g(\varepsilon )$ if there is a constant $C$, $%
0<C<\infty $, such that $f(\varepsilon )\leq Cg(\varepsilon )$ holds true
for all $\varepsilon >0$. It will also prove useful to define $\log \infty
=\infty $ and $\log 0=-\infty $, thus making the logarithm a continuous
function from $\left[ 0,\infty \right] $ to $\left[ -\infty ,\infty \right] $%
.

Let $(\Lambda _{0},\mathcal{A}_{0},P_{0})$, $(\Lambda _{n},\mathcal{A}%
_{n},P_{n})$, $n\geq 1$, be probability spaces. Suppose $Y_{0}:\Lambda
_{0}\rightarrow T$ is an $\mathcal{A}_{0}$-$\mathcal{B}(T,d)$-measurable
mapping and $Y_{n}:\Lambda _{n}\rightarrow T$ are (not necessarily
measurable) mappings, where $(T,d)$ is a metric space. We say that $Y_{n}$
converges weakly to $Y_{0}$ in $(T,d)$, denoted by $Y_{n}\rightsquigarrow
Y_{0}$, if the outer integrals $\int_{\Lambda _{n}}^{\ast }g(Y_{n})dP_{n}$
converge to $\int_{\Lambda _{0}}g(Y_{0})dP_{0}$ for every $g\in \mathsf{C}%
(T,d)$; furthermore, $Y_{n}$ is said to converge weakly to a Borel
probability measure $L$ on $(T,\mathcal{B}(T,d))$, denoted by $%
Y_{n}\rightsquigarrow L$, if $\int_{\Lambda _{n}}^{\ast }g(Y_{n})dP_{n}$
converges to $\int_{T}gdL$ for every $g\in \mathsf{C}(T,d)$. We say that $%
Y_{n}$ converges to $\tau \in T$ in outer $P_{n}$-probability if $%
P_{n}^{\ast }(d(Y_{n},\tau )>\varepsilon )$ converges to $0$ for all $%
\varepsilon >0$. If $Y_{n}$ are real-valued and $r_{n}$ is a sequence of
positive real numbers, we write $Y_{n}=o_{P_{n}}^{\ast }(r_{n})$ if $%
r_{n}^{-1}Y_{n}$ converges to $0$ in outer $P_{n}$-probability, and $%
Y_{n}=O_{P_{n}}^{\ast }(r_{n})$ if 
\begin{equation*}
\lim_{M\rightarrow \infty }\,\limsup_{n\rightarrow \infty }P_{n}^{\ast
}\left( r_{n}^{-1}Y_{n}>M\right) =0.
\end{equation*}%
In case the probability spaces $(\Lambda _{n},\mathcal{A}_{n},P_{n})$ are
the $n$-fold products of a single probability space $(\Lambda ,\mathcal{A}%
,P) $, that is, $(\Lambda _{n},\mathcal{A}_{n},P_{n})=(\Lambda ^{n},\mathcal{%
A}^{n},P^{n})$, we write $Y_{n}=o_{P}^{\ast }(r_{n})$ instead of $%
Y_{n}=o_{P^{n}}^{\ast }(r_{n})$ and $Y_{n}=O_{P}^{\ast }(r_{n})$ for $%
Y_{n}=O_{P^{n}}^{\ast }(r_{n})$.

\subsection{H\"{o}lder and Sobolev Spaces\label{sub:Sobolev-spaces}}

For $\Omega $ a (non-empty) open subset of $\mathbb{R}$, a function $%
f:\Omega \rightarrow \mathbb{R}$, and $s\geq 0$, define 
\begin{equation*}
\Vert f\Vert _{s,\Omega }=%
\begin{cases}
\sum_{0\leq \alpha \leq \lfloor s\rfloor }\Vert f^{(\alpha )}\Vert _{\Omega
}+\sup_{x\neq y}\frac{\left\vert f^{\lfloor s\rfloor }(x)-f^{\lfloor
s\rfloor }(y)\right\vert }{|x-y|^{s-\lfloor s\rfloor }} & \text{if $s$ is
non-integer,} \\ 
\sum_{0\leq \alpha \leq s}\Vert f^{(\alpha )}\Vert _{\Omega } & \text{%
otherwise}.%
\end{cases}%
\end{equation*}%
Here $f^{(\alpha )}$ denotes the classical derivative of $f$ of order $%
\alpha $, and $\lfloor s\rfloor $ denotes the integer part of $s$. For any
non-integer $s>0$, define the H\"{o}lder space $\mathsf{C}^{s}(\Omega )$ as
the space of all $f:\Omega \rightarrow \mathbb{R}$ such that $\Vert f\Vert
_{s,\Omega }<\infty $; for any integer $s\geq 0$, let $\mathsf{C}^{s}(\Omega
)$ be the space of all $f:\Omega \rightarrow \mathbb{R}$ such that $\Vert
f\Vert _{s,\Omega }<\infty $ and $f^{(s)}$ is uniformly continuous. Note
that $\mathsf{C}^{0}(\Omega )$ thus is the space of bounded and uniformly
continuous functions on $\Omega $.

For $\Omega $ and $s$ as above and functions $f,g\in \mathcal{L}^{2}(\Omega
) $, let 
\begin{equation*}
\langle f|g\rangle _{s,2}=%
\begin{cases}
\sum_{0\leq \alpha \leq \lfloor s\rfloor }\langle f^{(\alpha
)_{w}}|g^{(\alpha )_{w}}\rangle _{2} \\ 
\qquad +\int_{\Omega }\int_{\Omega }\frac{(f^{(\lfloor s\rfloor
)_{w}}(x)-f^{(\lfloor s\rfloor )_{w}}(y))(g^{(\lfloor s\rfloor
)_{w}}(x)-g^{(\lfloor s\rfloor )_{w}}(y))}{|x-y|^{1+2(s-\lfloor s\rfloor )}}%
d\lambda (x)d\lambda (y) \\ 
\hfill \text{if }s\text{ is non-integer,} \\ 
\sum_{0\leq \alpha \leq s}\langle f^{(\alpha )_{w}}|g^{(\alpha )_{w}}\rangle
_{2} \\ 
\hfill \text{otherwise,}%
\end{cases}%
\end{equation*}%
and set $\Vert f\Vert _{s,2}=\sqrt{\langle f|f\rangle _{s,2}}$. Here, $%
f^{(\alpha )_{w}}$ denotes the weak derivative of $f$ of order $\alpha $,
and $\langle \cdot |\cdot \rangle _{2}$ is the usual (semi)inner product on $%
\mathcal{L}^{2}(\Omega )$. Define $\mathcal{W}_{2}^{s}(\Omega )$ as the
space of all $f\in \mathcal{L}^{2}(\Omega )$ such that $\Vert f\Vert _{s,2}$
is finite. As usual, we equip $\mathcal{W}_{2}^{s}(\Omega )$ with the
(semi)norm $\Vert \cdot \Vert _{s,2}$. For $s>1/2$ and $\Omega $ a non-empty
bounded open interval in $\mathbb{R}$, each $f\in \mathcal{W}_{2}^{s}(\Omega
)$ is a.e.~equal to exactly one bounded continuous function on $\Omega $.
For $s>1/2$ and such $\Omega $, we consequently define the Sobolev space $%
\mathsf{W}_{2}^{s}(\Omega )=\mathcal{W}_{2}^{s}(\Omega )\cap \mathsf{C}%
(\Omega )$ and note that it is a Hilbert space. The Sobolev balls $\left\{
f\in \mathsf{W}_{2}^{s}(\Omega ):\Vert f\Vert _{s,2}\leq B\right\} $ of
radius $B$, $0<B<\infty $, will be denoted by $\mathcal{U}_{s,B}$, and its
translates $g+\mathcal{U}_{s,B}$ by $\mathcal{U}_{s,B}(g)$. The next
proposition collects some properties of Sobolev spaces; see Gach and P\"{o}%
tscher (2010) for a proof.

\begin{prop}
\label{prop:SobolevembedsinHoelder} Let $\Omega $ be a non-empty bounded,
open interval in $\mathbb{R}$.

(a) For $s>1/2$, the Sobolev space $\mathsf{W}_{2}^{s}(\Omega )$ is a
multiplication algebra; that is, there is a finite constant $M_{s}>0$ such
that%
\begin{equation*}
\Vert fg\Vert _{s,2}\leq M_{s}\Vert f\Vert _{s,2}\Vert g\Vert _{s,2}
\end{equation*}%
holds true for all $f,g\in \mathsf{W}_{2}^{s}(\Omega )$.

(b) For $s>1/2$, the Sobolev space $\QTR{up}{\mathsf{W}}_{2}^{s}(\Omega )$
is continuously embedded in $\QTR{up}{\mathsf{C}}^{s-1/2}(\Omega )$.
Consequently, $\QTR{up}{\mathsf{W}}_{2}^{s}(\Omega )$ is embedded in $%
\QTR{up}{\mathsf{C}}(\Omega )$ with an embedding constant $C_{s}$, $%
0<C_{s}<\infty $; that is, 
\begin{equation*}
\Vert f\Vert _{\Omega }\leq C_{s}\Vert f\Vert _{s,2}
\end{equation*}%
holds true for all $f\in \mathsf{W}_{2}^{s}(\Omega )$.

(c) If $0\leq r<s$, then $\mathcal{W}_{2}^{s}(\Omega )$ is compactly
embedded in $\mathcal{W}_{2}^{r}(\Omega )$; if $1/2<r<s$, then $\mathsf{W}%
_{2}^{s}(\Omega )$ is compactly embedded in $\mathsf{W}_{2}^{r}(\Omega )$.

(d) If $\mathcal{F}$ is a (non-empty) bounded subset of some Sobolev space $%
\mathsf{W}_{2}^{s}(\Omega )$ of order $s>1/2$ such that $\inf_{x\in \Omega
,f\in \mathcal{F}}\left\vert f(x)\right\vert >0$ holds, then $\left\{
1/f:f\in \mathcal{F}\right\} $ is also a bounded subset of $\mathsf{W}%
_{2}^{s}(\Omega )$.
\end{prop}

\subsection{Covering Numbers and Metric Entropy}

Let $(T,d)$ be a metric space. Let $0<\varepsilon <\infty $ and let $X$ be a
(non-empty) totally bounded subset of $T$. Then we denote by $N(\varepsilon
,X,T,d)$ the covering number of $X$, i.e., the minimal number of closed
balls in $T$ of radius $\varepsilon $ needed to cover $X$; we define the
metric entropy of $X$ as 
\begin{equation*}
H(\varepsilon ,X,T,d)=\log N(\varepsilon ,X,T,d).
\end{equation*}%
If $T$ is a normed space with norm $\left\Vert \cdot \right\Vert $, we shall
write in abuse of notation $N(\varepsilon ,X,T,\left\Vert \cdot \right\Vert
) $ and similarly for the metric entropy.

Let $(\Lambda ,\mathcal{A},\mu )$ be a (non-empty) measure space. For any
two elements $l,u\in \mathcal{L}^{0}(\Lambda ,\mathcal{A})$, the set 
\begin{equation*}
\lbrack l,u]=\{f\in \mathcal{L}^{0}(\Lambda ,\mathcal{A}):l(x)\leq f(x)\leq
u(x)\text{ for all }x\in \Lambda \}
\end{equation*}%
is called a bracket and $\Vert u-l\Vert _{2,\mu }$ its $\mathcal{L}^{2}(\mu
) $-bracketing size. For $0<\varepsilon <\infty $ and $\mathcal{F}$ a
(non-empty) subset of $\mathcal{L}^{0}(\Lambda ,\mathcal{A})$, we define $%
N_{[\hspace{0.75ex}]}(\varepsilon ,\mathcal{F},\Vert \cdot \Vert _{2,\mu })$
to be the minimal number of brackets of $\mathcal{L}^{2}(\mu )$-bracketing
size less than or equal to $\varepsilon $ needed to cover $\mathcal{F}$; if
there is no finite number of such brackets, we set $N_{[\hspace{0.75ex}%
]}(\varepsilon ,\mathcal{F},\Vert \cdot \Vert _{2,\mu })=\infty $ for
convenience. The $\mathcal{L}^{2}(\mu )$-bracketing metric entropy of $%
\mathcal{F}$ is defined as 
\begin{equation*}
H_{[\hspace{0.75ex}]}(\varepsilon ,\mathcal{F},\Vert \cdot \Vert _{2,\mu
})=\log N_{[\hspace{0.75ex}]}(\varepsilon ,\mathcal{F},\Vert \cdot \Vert
_{2,\mu }).
\end{equation*}%
Furthermore, for $0<\eta <\infty $ the $\mathcal{L}^{2}(\mu )$-bracketing
metric integral $I_{[\hspace{0.75ex}]}(\eta ,\mathcal{F},\Vert \cdot \Vert
_{2,\mu })$\ of $\mathcal{F}$ is given by%
\begin{equation*}
I_{[\hspace{0.75ex}]}(\eta ,\mathcal{F},\Vert \cdot \Vert _{2,\mu
})=\int_{(0,\eta ]}\sqrt{1+H_{[\hspace{0.75ex}]}(\varepsilon ,\mathcal{F}%
,\Vert \cdot \Vert _{2,\mu })}\,d\varepsilon .
\end{equation*}

\section{The Framework and Assumptions\label{Section: The framework}}

From now on let $\Omega $ be a non-empty bounded, open interval in $\mathbb{R%
}$. We consider i.i.d.~random variables $(X_{i})_{i\in \mathbb{N}}$ that
take their values in $(\Omega ,\mathcal{B}(\Omega ))$ and have common law $%
\mathbb{P}$, with $X_{1},\ldots ,X_{n}$ representing the data at sample size 
$n$. Furthermore, let $\Theta $ be a (non-empty) compact subset of $\mathbb{R%
}^{m}$ and let $\mathcal{P}_{\Theta }=\left\{ p_{\theta }:\theta \in \Theta
\right\} $ be a parametric family of probability density functions $%
p_{\theta }$ on $\Omega $. The law $\mathbb{P}$ may or may not correspond to
a density in $\mathcal{P}_{\Theta }$. We assume that there is a way of
simulating synthetic data according to the densities in the class $\mathcal{P%
}_{\Theta }$ in the following sense: There is a probability space $(V,%
\mathcal{V},\mu )$ and a function $\rho :V\times \Theta \rightarrow \Omega $%
, which is $\mathcal{V}$-$\mathcal{B}(\Omega )$-measurable in its first
argument, such that for every $\theta \in \Theta $ the law of $\rho (\cdot
,\theta )$ under $\mu $ has density $p_{\theta }$. Consequently, if $%
(V_{i})_{i\in \mathbb{N}}$ is a sequence of i.i.d.~random variables with
values in $(V,\mathcal{V})$ and law $\mu $, then $X_{i}(\theta )=\rho
(V_{i},\theta )$ is an i.i.d.~sequence with law having density $p_{\theta }$%
, simultaneously so for all $\theta \in \Theta $. We shall also always
assume that the process $(V_{i})_{i\in \mathbb{N}}$ is independent of $%
(X_{i})_{i\in \mathbb{N}}$. [As indicated in the Introduction, the
simulation mechanism $\rho $ may derive form an underlying equation model,
but it may also arise in some other way.] In the application to indirect
inference in Section \ref{Section: Indirect inference estimators} we shall
estimate $\theta $ by matching a non-parametric estimator for (the density
of) $\mathbb{P}$ obtained from the data $X_{1},\ldots ,X_{n}$ with a
non-parametric estimator for $p_{\theta }$ obtained from the synthetic data $%
X_{1}(\theta ),\ldots ,X_{k}(\theta )$. We stress that construction of the
synthetic data requires only one simulation, and not a separate simulation
for every $\theta $. For convenience we shall from now on assume that the
random variables $X_{i}$ and $V_{i}$ are the respective coordinate
projections on the measurable space $(\Omega ^{\mathbb{N}}\times V^{\mathbb{N%
}},\mathcal{B}(\Omega )^{\mathbb{N}}\otimes \mathcal{V}^{\mathbb{N}})$
equipped with the product measure $\func{Pr}:=\mathbb{P}^{\mathbb{N}}\otimes
\mu ^{\mathbb{N}}$. We note, however, that all results of the paper hold
also without this assumption; see Remark \ref{acanonical}. Furthermore, the
empirical measures associated with $X_{1},\ldots ,X_{n}$ and $V_{1},\ldots
,V_{k}$ will be denoted by $\mathbb{P}_{n}$ and $\mu _{k}$, respectively.

The density estimators we shall consider will be NPML-estimators over
non-parametric models (called auxiliary models in Section \ref{Section:
Indirect inference estimators}) of the form%
\begin{equation*}
\mathcal{P}(t,\zeta ,D)=\left\{ p\in \mathsf{W}_{2}^{t}(\Omega
):\int_{\Omega }p\,d\lambda =1,\,\inf_{x\in \Omega }p(x)\geq \zeta ,\,\Vert
p\Vert _{t,2}\leq D\right\} ,
\end{equation*}%
where $t>1/2$, $0\leq \zeta <\infty $, and $0<D<\infty $. Some important
properties of $\mathcal{P}(t,\zeta ,D)$ that will be used repeatedly are
summarized in the subsequent propositions, the proofs of which can be found
in Appendix \ref{App A}.

\begin{prop}
\label{Proposition:Propertiesoftheauxiliarymodel} Suppose $t>1/2$, $0\leq
\zeta <\infty $, and $0<D<\infty $.

(a) The following statements are equivalent: (i) $\zeta \leq \lambda (\Omega
)^{-1}\leq D^{2}$; (ii) the constant density $\lambda (\Omega )^{-1}$
belongs to $\mathcal{P}(t,\zeta ,D)$; (iii) $\mathcal{P}(t,\zeta ,D)$ is
non-empty.

(b) Suppose $\zeta \leq \lambda (\Omega )^{-1}\leq D^{2}$. Then the
following statements are equivalent: (i) $\zeta =\lambda (\Omega )^{-1}$ or $%
\lambda (\Omega )^{-1}=D^{2}$; (ii) the constant density $\lambda (\Omega
)^{-1}$ is the only element of $\mathcal{P}(t,\zeta ,D)$; (iii) $\mathcal{P}%
(t,\zeta ,D)$ is a singleton.

(c) Suppose $\zeta \leq \lambda (\Omega )^{-1}\leq D^{2}$. Then $\mathcal{P}%
(t,\zeta ,D)$ is a non-empty convex set, which is compact in $\mathsf{C}%
(\Omega )$ as well as in $\mathsf{W}_{2}^{s}(\Omega )$ for every $s$
satisfying $1/2<s<t$.
\end{prop}

In the following let $\mathsf{H}_{t}$ denote the closed affine hyperplane
given by $\mathsf{H}_{t}=\left\{ f\in \mathsf{W}_{2}^{t}(\Omega
):\int_{\Omega }f\,d\lambda =1\right\} $ endowed with the relative topology
it inherits from $\mathsf{W}_{2}^{t}(\Omega )$. Note that $\mathcal{P}%
(t,\zeta ,D)\subseteq \mathsf{H}_{t}$ holds.

\begin{prop}
\footnote{%
An obvious extension of Theorem V.2.1 in Dunford and Schwartz (1966) to
affine spaces shows that in our setting the notion of an element being
interior relative to $\mathcal{H}$\ coincides with the notion of internality
of that element (relative to $\mathcal{H}$).}\label{prop_interior} Suppose $%
t>1/2$ and $0\leq \zeta \leq \lambda (\Omega )^{-1}\leq D^{2}<\infty $.

(a) An element $p\in \mathcal{P}(t,\zeta ,D)$ is an interior point of $%
\mathcal{P}(t,\zeta ,D)$ relative to $\mathsf{H}_{t}$ if and only if (i) $%
\Vert p\Vert _{t,2}<D$ and (ii) $\inf_{x\in \Omega }p(x)>\zeta $ hold.

(b) A (non-empty) subset $\mathcal{P}^{\prime }$ of $\mathcal{P}(t,\zeta ,D)$
is uniformly interior to $\mathcal{P}(t,\zeta ,D)$ relative to $\mathsf{H}%
_{t}$ (meaning that there exists a $\delta >0$ such that for every $p\in 
\mathcal{P}^{\prime }$ the set $\mathcal{U}_{t,\delta }(p)\cap \mathsf{H}%
_{t}\subseteq \mathcal{P}(t,\zeta ,D)$) if and only if (i) $\sup_{p\in 
\mathcal{P}^{\prime }}\Vert p\Vert _{t,2}<D$ and (ii) $\inf_{x\in \Omega
,p\in \mathcal{P}^{\prime }}p(x)>\zeta $ hold.

(c) Suppose $\zeta <\lambda (\Omega )^{-1}<D^{2}$ holds. Then the constant
density $\lambda (\Omega )^{-1}$ is interior to $\mathcal{P}(t,\zeta ,D)$
relative to $\mathsf{H}_{t}$. Moreover, the interior of $\mathcal{P}(t,\zeta
,D)$ relative to $\mathsf{H}_{t}$ is dense in $\mathcal{P}(t,\zeta ,D)$
(w.r.t.~the $\mathsf{W}_{2}^{t}(\Omega )$-topology).
\end{prop}

\emph{We emphasize that for the rest of the paper }$t$\emph{, }$\zeta $\emph{%
, and }$D$\emph{\ will be treated as fixed (although at arbitrary values)
satisfying the constraints }$t>1/2$\emph{\ and }$0\leq \zeta <\lambda
(\Omega )^{-1}<D^{2}<\infty $\emph{\ }(thus excluding only the trivial cases
where $\mathcal{P}(t,\zeta ,D)$ is empty or the singleton $\{\lambda (\Omega
)^{-1}\}$). Many results will hold under the natural condition $\zeta \geq 0$%
, but for some results we shall have to assume the stronger requirement $%
\zeta >0$. In that context we note that if $D^{2}$ is sufficiently close to $%
\lambda (\Omega )^{-1}$, then $\mathcal{P}(t,0,D)$ coincides with $\mathcal{P%
}(t,\zeta ,D)$ for sufficiently small $\zeta >0$, cf. Remark \ref{pyth} in
Appendix \ref{App A}.

For later use we stress that any $p\in \mathcal{P}(t,\zeta ,D)$ is
continuous on $\Omega $ and satisfies $\Vert p\Vert _{\Omega }\leq C_{t}D$
in view of Part~(b) of Proposition~\ref{prop:SobolevembedsinHoelder}. We
further note the fact that in $\mathcal{P}(t,\zeta ,D)$ pointwise
convergence is equivalent to convergence in all Sobolev norms of order
smaller than $t$, as well as to convergence in the sup-norm, as shown in
Proposition \ref{Proposition: equivalence of pointwise and sup-norm
convergence} in Appendix \ref{App A}.

Apart from the maintained assumptions laid out at the beginning of this
section, we will make frequent use of the assumptions listed below. We start
with assumptions on the probability measure $\mathbb{P}$ governing the data.

\begin{datAsm}
\label{Data}

The probability measure $\mathbb{P}$ has a density $p_{\blacktriangle }$.
\end{datAsm}

In the following we treat the probability density $p_{\blacktriangle }$ as a 
\emph{function} from $\Omega $ to $\mathbb{R}$, that is, we let $%
p_{\blacktriangle }$ denote a fixed representative of the Radon-Nikodym
derivative of $\mathbb{P}$ with respect to $\lambda $. Recall also that $%
\mathbb{P}$ need not correspond to an element of $\mathcal{P}_{\Theta }$,
hence $p_{\blacktriangle }$ need not be a.e.~equal to an element of $%
\mathcal{P}_{\Theta }$.

\begin{densAsm}
\label{dens:Element}

Assumption \ref{Data} holds and the density function $p_{\blacktriangle }$
belongs to $\mathcal{P}(t,\zeta ,D)$.
\end{densAsm}

\begin{densAsm}
\label{dens:StrictInequality}

Assumption \ref{Data} holds and the density function $p_{\blacktriangle }$
satisfies the strict inequality%
\begin{equation*}
\inf_{x\in \Omega }p_{\blacktriangle }(x)>0.
\end{equation*}
\end{densAsm}

Clearly, if $\zeta >0$, then Assumption~\ref{dens:Element} implies
Assumption~\ref{dens:StrictInequality}. In light of Proposition \ref%
{prop_interior}, the next assumption just states that $p_{\blacktriangle }$
is an interior point of $\mathcal{P}(t,\zeta ,D)$ relative to $\mathsf{H}%
_{t} $.

\begin{densAsm}
\label{dens:InternalPoint}

Assumption \ref{dens:Element} holds and the strict inequalities%
\begin{equation*}
\inf_{x\in \Omega }p_{\blacktriangle }(x)>\zeta \text{ \ \ and \ \ }\Vert
p_{\blacktriangle }\Vert _{t,2}<D
\end{equation*}%
are satisfied.
\end{densAsm}

We note here, however, that even under Assumption \ref{dens:InternalPoint}
the NPML-estimator is \emph{never} an interior point of $\mathcal{P}(t,\zeta
,D)$ relative to $\mathsf{H}_{t}$ as shown in Section \ref{Section:
Auxiliary maximum likelihood estimators}; this leads to a number of
complications as discussed prior to Lemma \ref{derivativeorder} in Section %
\ref{Section: A uniform Donsker theorem for AML-estimators}.

Next are assumptions on the class $\mathcal{P}_{\Theta }$. We will often
write $p(x,\theta )$ for $p_{\theta }(x)$, and we stress that $p(x,\theta )$
is a \emph{function} from $\Omega \times \Theta $ to $\mathbb{R}$.

\begin{modAsm}
\label{mod:Inclusion} $\mathcal{P}_{\Theta }\subseteq \mathcal{P}(t,\zeta
,D) $.
\end{modAsm}

\begin{modAsm}
\label{mod: Strict inequality}

The strict inequality%
\begin{equation*}
\inf_{\Omega \times \Theta }p(x,\theta )>0
\end{equation*}%
holds true.
\end{modAsm}

Clearly, if $\zeta >0$ then Assumption~\ref{mod:Inclusion} implies
Assumption~\ref{mod: Strict inequality}.

\begin{modAsm}
\label{mod:InclusionWithStrictInequalities} Assumption \ref{mod:Inclusion}
holds and the strict inequalities%
\begin{equation*}
\inf_{\Omega \times \Theta }p(x,\theta )>\zeta \text{ \ \ and \ \ }%
\sup_{\theta \in \Theta }\Vert p_{\theta }\Vert _{t,2}<D
\end{equation*}%
are satisfied.
\end{modAsm}

Assumption \ref{mod:InclusionWithStrictInequalities} states that $\mathcal{P}%
_{\Theta }$ is uniformly interior to $\mathcal{P}(t,\zeta ,D)$ relative to $%
\mathsf{H}_{t}$, cf.~Proposition \ref{prop_interior}. If $\mathcal{P}%
_{\Theta }$ happens to be a $\Vert \cdot \Vert _{t,2}$-compact subset of $%
\mathcal{P}(t,\zeta ,D)$ (which in light of compactness of $\Theta $ is,
e.g., the case if the map $\theta \rightarrow p_{\theta }$ is $\Vert \cdot
\Vert _{t,2}$-continuous), Assumption \ref%
{mod:InclusionWithStrictInequalities} is clearly equivalent to $\inf_{x\in
\Omega }p(x,\theta )>\zeta $ and $\Vert p_{\theta }\Vert _{t,2}<D$ for every 
$\theta \in \Theta $ (i.e., equivalent to $\mathcal{P}_{\Theta }$ belonging
to the interior of $\mathcal{P}(t,\zeta ,D)$ relative to $\mathsf{H}_{t}$).

We note that in the correctly specified case, i.e., if there exists a $%
\theta _{0}\in \Theta $ such that $p_{\theta _{0}}$ is a density of $\mathbb{%
P}$, Assumptions \ref{dens:Element}-\ref{dens:InternalPoint} follow
automatically from the respective Assumptions \ref{mod:Inclusion}-\ref%
{mod:InclusionWithStrictInequalities} (and Assumption \ref{Data} trivially
holds).

Occasionally we shall also need to refer to the following assumption.
However, note that Assumption~\ref{mod:Inclusion} together with Assumption~%
\ref{rho: Continuity of rho} below already imply this assumption,
cf.~Proposition \ref{Interrelation} in Appendix \ref{App A}.

\begin{modAsm}
\label{mod: pointwise continuity}

For every $x\in \Omega $, $\theta \mapsto p(x,\theta )$ is a continuous
function on $\Theta $.
\end{modAsm}

\begin{remark}
\normalfont\label{Remark: Equivalence of pointwise and sup-norm convergence
of the parametrization} If Assumption~\ref{mod:Inclusion} is satisfied, then
in view of Proposition~\ref{Proposition: equivalence of pointwise and
sup-norm convergence} in Appendix \ref{App A} the following are equivalent:
(i) Assumption~\ref{mod: pointwise continuity}; (ii) $\theta \mapsto
p_{\theta }$ is continuous as a mapping from $\Theta $ into the space $(%
\mathcal{P}(t,\zeta ,D),\Vert \cdot \Vert _{s,2})$ for every $s$ satisfying $%
0\leq s<t$; (iii) $\theta \mapsto p_{\theta }$ is continuous as a mapping
from $\Theta $ into the space $(\mathcal{P}(t,\zeta ,D),\Vert \cdot \Vert
_{\Omega })$.
\end{remark}

Next are assumptions on the simulation mechanism $\rho (v,\theta )$. Apart
from the already assumed measurability of $\rho (v,\theta )$ in its first
argument, we will need assumptions to control its behaviour in the second
argument. We note that Assumption \ref{rho: Hoelderity of rho} below is
weaker than the corresponding Assumption R.2 in Gach (2010), but we have
been able to obtain the same conclusions as in Gach (2010) by refining the
proofs. Clearly, Assumption \ref{rho: Hoelderity of rho} implies Assumption %
\ref{rho: Continuity of rho}.

\begin{rhoAsm}
\label{rho: Continuity of rho} For every $v\in V$, the simulation mechanism $%
\rho (v,\theta )$ is continuous in $\theta $.
\end{rhoAsm}

\begin{rhoAsm}
\label{rho: Hoelderity of rho} For some constant $\gamma $, $0<\gamma \leq 1$%
, and some measurable function $R:V\rightarrow (0,\infty )$, the simulation
mechanism $\rho :V\times \Theta \rightarrow \Omega $ satisfies 
\begin{equation*}
|\rho (v,\theta ^{\prime })-\rho (v,\theta )|\leq R(v)\Vert \theta ^{\prime
}-\theta \Vert ^{\gamma }
\end{equation*}%
for all $v\in V$ and all $\theta ,\theta ^{\prime }\in \Theta $, with the
function $R$ satisfying $\int_{V}R^{a}d\mu <\infty $ for some $a>0$.
\end{rhoAsm}

Assumptions on the class $\mathcal{P}_{\Theta }$ and on the simulation
mechanism $\rho (v,\theta )$ are obviously closely related. In principle,
the assumptions on $\mathcal{P}_{\Theta }$ could be substituted for by
assumptions on $\rho (v,\theta )$. [Conversely, the existence of a
simulation mechanism having certain required properties can in principle be
deduced from suitable assumptions on $\mathcal{P}_{\Theta }$.] However, the
interrelation between assumptions on $\mathcal{P}_{\Theta }$ and on $\rho
(v,\theta )$ is complicated and intricate, and hence we prefer to work with
the two sets of assumptions as given above. For some results concerning the
relationship between these two sets of assumptions see Proposition \ref%
{Interrelation} in Appendix \ref{App A}.

\section{Non-Parametric Maximum Likelihood Estimators \label{Section:
Auxiliary maximum likelihood estimators}}

We now introduce NPML-estimators, called auxiliary estimators in Section \ref%
{Section: Indirect inference estimators}. Define the (non-parametric)
log-likelihood function based on the given data $X_{1},\ldots ,X_{n}$ as%
\begin{equation*}
L_{n}(p):=L_{n}(p;X_{1},\ldots ,X_{n})=\frac{1}{n}\sum_{i=1}^{n}\log p(X_{i})
\end{equation*}%
for $p\in \mathcal{P}(t,\zeta ,D)$, and based on the simulated data $%
X_{1}(\theta )=\rho (V_{1},\theta ),\ldots ,X_{k}(\theta )=\rho
(V_{k},\theta )$ as%
\begin{equation*}
L_{k}(\theta ,p):=L_{k}(\theta ,p;V_{1},\ldots ,V_{k})=\frac{1}{k}%
\sum_{i=1}^{k}\log p(\rho (V_{i},\theta ))
\end{equation*}%
for $p\in \mathcal{P}(t,\zeta ,D)$ and $\theta \in \Theta $. Note that $%
L_{k}(\theta ,p)=L_{k}(p;X_{1}(\theta ),\ldots ,X_{k}(\theta
))=k^{-1}\sum_{i=1}^{k}\log p(X_{i}(\theta ))$ holds. In view of our
convention for the logarithm, both functions $L_{n}(f)$ and $L_{k}(\theta
,f) $ are in fact well-defined and take their values in $[-\infty ,\infty )$
for any non-negative real-valued function $f$ on $\Omega $.

An NPML-estimator for given $X_{1},\ldots ,X_{n}$ is defined as an element $%
\hat{p}_{n}(\cdot ):=\hat{p}_{n}(\cdot ;X_{1},\ldots ,X_{n})$ of $\mathcal{P}%
(t,\zeta ,D)$ satisfying%
\begin{equation*}
L_{n}(\hat{p}_{n})=\sup_{p\in \mathcal{P}(t,\zeta ,D)}L_{n}(p).
\end{equation*}%
Similarly, an NPML-estimator for given $X_{1}(\theta ),\ldots ,X_{k}(\theta
) $ is an element $\tilde{p}_{k}(\theta )(\cdot ):=\tilde{p}_{k}(\theta
)(\cdot ;V_{1},\ldots ,V_{k})$ of $\mathcal{P}(t,\zeta ,D)$ satisfying%
\begin{equation*}
L_{k}(\theta ,\tilde{p}_{k}(\theta ))=\sup_{p\in \mathcal{P}(t,\zeta
,D)}L_{k}(\theta ,p).
\end{equation*}%
Clearly we have 
\begin{equation}
\tilde{p}_{k}(\theta )(\cdot ;V_{1},\ldots ,V_{k})=\hat{p}_{k}(\cdot
;X_{1}(\theta ),\ldots ,X_{k}(\theta )).  \label{representation}
\end{equation}

In this section we investigate existence, uniqueness, consistency, rates of
convergence, and uniform central limit theorems for NPML-estimators. The
results obtained here go beyond Nickl (2007) in three respects: First, we
show not only existence but also \emph{uniqueness} of the NPML-estimators.
Second, we allow for non-parametric models $\mathcal{P}(t,\zeta ,D)$ where
the lower bound for the densities, i.e., $\zeta $, can be equal to $0$ and
extend the consistency and rate results for the NPML-estimator w.r.t.~the
Sobolev-norms $\Vert \cdot \Vert _{s,2}$ with $s<t$ in Nickl (2007) to this
case. We furthermore also establish \emph{inconsistency} of the
NPML-estimator in the $\Vert \cdot \Vert _{t,2}$-norm. Third, we prove that
the consistency and rate results in Nickl (2007) for $\hat{p}_{n}$ hold for
the NPML-estimators $\tilde{p}_{k}(\theta )$ even \emph{uniformly} over the
parameter space $\Theta $ (provided that $\zeta >0$). Finally, we prove a 
\emph{uniform} Donsker-type theorem which extends Theorem~3 in Nickl (2007)
and shows that, for appropriate classes $\mathcal{F}$, the stochastic
process $(\theta ,f)\mapsto \sqrt{k}\int_{\Omega }(\tilde{p}_{k}(\theta
)-p_{\theta })fd\lambda $ converges weakly in $\ell ^{\infty }(\Theta \times 
\mathcal{F})$ to a Gaussian process.

\subsection{Existence, Uniqueness, and Consistency of NPML-Estimators \label%
{Subsection: Existence-of-AML-estimators}}

In the following theorem we show that the NPML-estimators defined above
exist, are unique, and are measurable (cf.~also Lemma \ref{Borel} in
Appendix \ref{App D2}).

\begin{thm}
\label{Theorem: existence of AML-estimators} (a) There exists a unique $\hat{%
p}_{n}\in \mathcal{P}(t,\zeta ,D)$ such that 
\begin{equation*}
L_{n}(\hat{p}_{n})=\sup_{p\in \mathcal{P}(t,\zeta ,D)}L_{n}(p)
\end{equation*}%
holds. The resulting mapping $\hat{p}_{n}:\Omega ^{n}\rightarrow \mathcal{P}%
(t,\zeta ,D)$ is measurable with respect to the $\sigma $-fields $\mathcal{B}%
(\Omega )^{n}$ and $\mathcal{B}(\mathcal{P}(t,\zeta ,D),\Vert \cdot \Vert
_{\Omega })$. Moreover, $\hat{p}_{n}$ always satisfies $\Vert \hat{p}%
_{n}\Vert _{t,2}=D$.

(b) For each $\theta \in \Theta $ there exists a unique $\tilde{p}%
_{k}(\theta )\in \mathcal{P}(t,\zeta ,D)$ such that 
\begin{equation*}
L_{k}(\theta ,\tilde{p}_{k}(\theta ))=\sup_{p\in \mathcal{P}(t,\zeta
,D)}L_{k}(\theta ,p)
\end{equation*}%
holds. The resulting mapping $\tilde{p}_{k}(\theta ):V^{k}\rightarrow 
\mathcal{P}(t,\zeta ,D)$ is measurable with respect to the $\sigma $-fields $%
\mathcal{V}^{k}$ and $\mathcal{B}(\mathcal{P}(t,\zeta ,D),\Vert \cdot \Vert
_{\Omega })$. Moreover, $\tilde{p}_{k}(\theta )$ always satisfies $\Vert 
\tilde{p}_{k}(\theta )\Vert _{t,2}=D$. Furthermore, if Assumption~\ref{rho:
Continuity of rho} is satisfied, then, for arbitrary fixed values of the
underlying simulated variables $V_{1},\ldots ,V_{k}$, $\theta \mapsto \tilde{%
p}_{k}(\theta )$ is continuous when viewed as a mapping from $\Theta $ into
the space $(\mathcal{P}(t,\zeta ,D),\Vert \cdot \Vert _{\Omega })$.
\end{thm}

\begin{proof}
(a) Let $x_{1},\ldots ,x_{n}$ be given points in $\Omega $. The existence of
a maximizer of $L_{n}(p)=L_{n}(p;x_{1},\ldots ,x_{n})$ follows from the fact
that $L_{n}$ is continuous on the compact space $(\mathcal{P}(t,\zeta
,D),\Vert \cdot \Vert _{\Omega })$ by Part (b1) of Proposition~\ref%
{Properties of L_n and L_k} in Appendix \ref{App B} with $\mathcal{F}=%
\mathcal{P}(t,\zeta ,D)$ and by Proposition~\ref%
{Proposition:Propertiesoftheauxiliarymodel}. We next establish uniqueness:
Denote by $S$ the set of all $p\in \mathcal{P}(t,\zeta ,D)$ that maximize $%
L_{n}$, and note that $S$ is non-empty as just shown. Since $L_{n}$ is a
concave function on the convex set $\mathcal{P}(t,\zeta ,D)$ with values in $%
[-\infty ,\infty )$, a standard argument shows that $S$ is convex. If $S$ is
a subset of the Sobolev \emph{sphere} of radius $D$ we are done, as then $S$
must be a singleton since the Sobolev norm $\Vert \cdot \Vert _{t,2}$, being
a Hilbert norm, is strictly convex. Suppose now $S$ is not a subset of the
Sobolev sphere of radius $D$ and let $p\in S$ with $\Vert p\Vert _{t,2}<D$.
Then there is some $z\in \Omega $ with $p(z)>\zeta $ since the maintained
assumption $\zeta <\lambda ^{-1}(\Omega )$ implies that $\zeta \notin 
\mathcal{P}(t,\zeta ,D)$. By continuity of $p$ we may assume that $z$ is
different from any of the finitely many data points $x_{1},\ldots ,x_{n}$.
We claim that there is a $q\in \mathcal{P}(t,\zeta ,D)$ such that $%
q(x_{i})>p(x_{i})$ whenever $x_{i}=x_{1}$ and $q$ coincides with $p$ on the
remaining (if any) observations $x_{j}$ with $x_{j}\neq x_{1}$. This will
contradict the maximizing property of $p$ (noting that the case $%
L_{n}(q)=L_{n}(p)=-\infty $ is impossible in view of $\lambda (\Omega
)^{-1}\in \mathcal{P}(t,\zeta ,D)$ and $L_{n}(p)\geq L_{n}(\lambda (\Omega
)^{-1})>-\infty $). The existence of such a $q$ can be seen as follows:
Choose $\varepsilon >0$ such that $I:=[z-2\varepsilon ,z+2\varepsilon ]$, $%
\bar{U}:=[x_{1}-2\varepsilon ,x_{1}+2\varepsilon ]$, and $\{x_{j}:x_{j}\neq
x_{1}\}$ are pairwise disjoint subsets of $\Omega $ and $\inf_{x\in
I}p(x)>\zeta $. As $A:=[x_{1}-\varepsilon ,x_{1}+\varepsilon ]$ is a closed
set contained in the open set $U:=(x_{1}-2\varepsilon ,x_{1}+2\varepsilon )$%
, there is a compactly supported $C^{\infty }$-function $f:\Omega
\rightarrow \mathbb{R}$ with values in $[0,1]$ such that $f|_{A}=1$ and $%
f|_{\Omega \setminus U}=0$. For every $y\in \Omega $ let 
\begin{equation*}
\bar{f}(y)=%
\begin{cases}
f(y+x_{1}-z) & \text{if $y+x_{1}-z\in \Omega $,} \\ 
0 & \text{otherwise},%
\end{cases}%
\end{equation*}%
so that $\bar{f}$ is the translation of $f$ by $z-x_{1}$; and define $%
g:\Omega \rightarrow \mathbb{R}$ by $g=f-\bar{f}$. Then $g$ has values in $%
[-1,1]$, integrates to $0$, and is contained in $\mathsf{W}_{2}^{t}(\Omega )$
since it is $C^{\infty }$ and has compact support in $\Omega $. Since $\Vert
p\Vert _{t,2}<D$ and $\inf_{x\in I}p(x)>\zeta $, we can find a scalar $\beta
>0$ such that $\Vert \beta g\Vert _{t,2}\leq D-\Vert p\Vert _{t,2}$ and $%
\beta \leq \inf_{x\in I}p(x)-\zeta $. Let $q=p+\beta g$ and observe that $%
\Vert q\Vert _{t,2}\leq \Vert p\Vert _{t,2}+\Vert \beta g\Vert _{t,2}\leq D$%
. Further, $q(x)\geq \zeta $ for every $x\in \Omega $, which can be seen as
follows: For $x\in \Omega \setminus I$ we have that $g(x)\geq 0$, and hence $%
q(x)\geq p(x)\geq \zeta $. If $x\in I$, then $q(x)\geq p(x)-\beta \geq
p(x)-\inf_{x\in I}p(x)+\zeta \geq \zeta $, where the first inequality holds
because $g(x)\geq -1$ for every $x\in \Omega $, the second inequality holds
by the choice of $\beta $, and the third one does so since $x\in I$ and
therefore $p(x)-\inf_{x\in I}p(x)\geq 0$. It follows that $q\in \mathcal{P}%
(t,\zeta ,D)$. Since $\beta >0$ and $g(x_{1})=1$, $q(x_{i})>p(x_{i})$
whenever $x_{i}=x_{1}$. Furthermore, $q$ coincides with $p$ on the remaining
(if any) data points because $g$ is $0$ there. The existence of $q$
contradicts the maximizing property of $p$, and consequently $S$ is a subset
of the Sobolev sphere of radius $D$. We thus have established uniqueness as
well as $\Vert \hat{p}_{n}\Vert _{t,2}=D$.

To see that $\hat{p}_{n}:\Omega ^{n}\rightarrow \mathcal{P}(t,\zeta ,D)$ is
measurable, we apply Lemma~A3 in P\"{o}tscher and Prucha (1997), making use
of Proposition~\ref{Properties of L_n and L_k}(a),(b1) in Appendix \ref{App
B}. [Because $L_{n}$ potentially can attain the value $-\infty $, we apply
this lemma to the\emph{\ real-valued} function $\arctan (L_{n})$ rather than
to $L_{n}$, where we use the usual convention $\arctan (-\infty )=-\pi /2$.]

(b) The same arguments as above establish existence, uniqueness, and
measurability of $\tilde{p}_{k}(\theta )$, as well as $\Vert \tilde{p}%
_{k}(\theta )\Vert _{t,2}=D$, for any fixed $\theta \in \Theta $. To see
that the mapping $\theta \mapsto \tilde{p}_{k}(\theta )$ is continuous as
claimed, apply Lemma~\ref{lem: Continuous selections} in Appendix \ref{App B}
with $X=\Theta $, $Y=(\mathcal{P}(t,\zeta ,D),\Vert \cdot \Vert _{\Omega })$%
, $u(x,y)=L_{k}(\theta ,p)$, and $v(x)=\tilde{p}_{k}(\theta )$. Note that $(%
\mathcal{P}(t,\zeta ,D),\Vert \cdot \Vert _{\Omega })$ is a compact metric
space by Proposition~\ref{Proposition:Propertiesoftheauxiliarymodel} and
that, under Assumption~\ref{rho: Continuity of rho}, $L_{k}(\theta ,p)$ is
continuous on $\Theta \times (\mathcal{P}(t,\zeta ,D),\Vert \cdot \Vert
_{\Omega })$, as can be seen by applying Part~(b2) of Proposition~\ref%
{Properties of L_n and L_k} in Appendix \ref{App B} with $\mathcal{F}=%
\mathcal{P}(t,\zeta ,D)$.
\end{proof}

\begin{remark}
\normalfont\label{Remark: meas_pos} (i) The mapping $\hat{p}_{n}:\Omega
\times \Omega ^{n}\rightarrow \mathbb{R}$ is continuous in the first
argument and $\mathcal{B}(\Omega )^{n}$-measurable in the second argument.
Since $\Omega $ is separable, $\hat{p}_{n}$ is consequently jointly
measurable. Similarly, the mappings $\tilde{p}_{k}(\theta ):\Omega \times
V^{k}\rightarrow \mathbb{R}$ are jointly measurable for all $\theta \in
\Theta $.

(ii) For any $x_{1},\ldots ,x_{n}$ in $\Omega $, we have that $\hat{p}%
_{n}(x_{i})=\hat{p}_{n}(x_{i};x_{1},\ldots ,x_{n})>0$ for $i=1,\ldots ,n$.
This follows from the observation made in the above proof that $L_{n}(\hat{p}%
_{n})>-\infty $ must hold. By a similar argument we have that $\tilde{p}%
_{k}(\theta )(\rho (v_{i},\theta ))=\tilde{p}_{k}(\theta )(\rho
(v_{i},\theta );v_{1},\ldots ,v_{k})>0$ for $i=1,\ldots ,k$ and for every $%
\theta \in \Theta $.
\end{remark}

We next turn to consistency of the NPML-estimators. Theorem \ref{Theorem:
existence of AML-estimators} already shows that $\hat{p}_{n}$ cannot be
consistent in the $\Vert \cdot \Vert _{t,2}$-norm as $\Vert \hat{p}_{n}\Vert
_{t,2}=D$ always holds and $\mathcal{P}(t,\zeta ,D)$ contains densities with 
$\Vert \cdot \Vert _{t,2}$-norm less than $D$ (under our assumptions on $%
\zeta $ and $D$). A similar remark applies to $\tilde{p}_{k}(\theta )$.
However, this does not preclude consistency of the NPML-estimators in other
norms as we show next. To this end define for any \emph{non-negative}
measurable function $f$ on $\Omega $ and for any $\theta \in \Theta $%
\begin{equation*}
L(f)=\int_{\Omega }\log fd\hspace{0.1ex}\mathbb{P}
\end{equation*}%
and%
\begin{equation*}
L(\theta ,f)=\int_{V}\log f(\rho (\cdot ,\theta ))d\mu
\end{equation*}%
provided the respective integral is defined. If $f\in \mathsf{L}^{\infty
}(\Omega )$, then both functions are well-defined and take their values in $%
[-\infty ,\infty )$. We note that the restrictions of $L(f)$ to $\mathcal{P}%
(t,\zeta ,D)$ and of $L(\theta ,f)$ to $\Theta \times \mathcal{P}(t,\zeta
,D) $ are real-valued in case $\zeta >0$. We will make use of the following
simple facts which are proved in Appendix \ref{App B}.

\begin{lem}
\label{lem: p_0 is the unique maximizer}\hspace*{0ex}(a) $%
L(p_{\blacktriangle })$ is well-defined and satisfies $L(p_{\blacktriangle
})>-\infty $, provided Assumption~\ref{Data} holds. Similarly, for every $%
\theta \in \Theta $, $L(\theta ,p_{\theta })$ is well-defined and satisfies $%
L(\theta ,p_{\theta })>-\infty $.

(b) If Assumption \ref{dens:Element} is satisfied, then $p_{\blacktriangle }$
is the unique maximizer of the function $L(\cdot )$ over $\mathcal{P}%
(t,\zeta ,D)$.

(c) If $p_{\theta }\in \mathcal{P}(t,\zeta ,D)$ for a given $\theta \in
\Theta $, then $p_{\theta }$ is the unique maximizer of the function $%
L(\theta ,\cdot )$ over $\mathcal{P}(t,\zeta ,D)$.
\end{lem}

The consistency result is now given below. Under the additional assumption
that $\zeta $ is positive, Part (a) of the subsequent theorem already
follows from Proposition~6 in Nickl (2007).

\begin{thm}
\label{Theorem: consistency of AML-estimators}\hspace*{0ex}(a) Let
Assumption \ref{dens:Element} be satisfied. Then%
\begin{equation*}
\lim_{n\rightarrow \infty }\Vert \hat{p}_{n}-p_{\blacktriangle }\Vert
_{s,2}=0\quad \mathbb{P}\text{-a.s.}
\end{equation*}%
for every $s$, $0\leq s<t$; in particular, $\lim_{n\rightarrow \infty }\Vert 
\hat{p}_{n}-p_{\blacktriangle }\Vert _{\Omega }=0$ $\ \mathbb{P}$-a.s.

(b) Let $p_{\theta }\in \mathcal{P}(t,\zeta ,D)$ for a given $\theta \in
\Theta $. Then, for the given $\theta $,%
\begin{equation*}
\lim_{k\rightarrow \infty }\Vert \tilde{p}_{k}(\theta )-p_{\theta }\Vert
_{s,2}=0\quad \mu \text{-a.s.}
\end{equation*}%
for every $s$, $0\leq s<t$; in particular, $\lim_{k\rightarrow \infty }\Vert 
\tilde{p}_{k}(\theta )-p_{\theta }\Vert _{\Omega }=0$ $\ \mu $-a.s.

(c) Let Assumptions~\ref{mod:Inclusion}, \ref{mod: Strict inequality}, and %
\ref{rho: Continuity of rho} be satisfied. Then%
\begin{equation*}
\lim_{k\rightarrow \infty }\,\sup_{\theta \in \Theta }\Vert \tilde{p}%
_{k}(\theta )-p_{\theta }\Vert _{s,2}=0\quad \mu \text{-a.s.}
\end{equation*}%
for every $s$, $0\leq s<t$; in particular, $\lim_{k\rightarrow \infty
}\sup_{\theta \in \Theta }\Vert \tilde{p}_{k}(\theta )-p_{\theta }\Vert
_{\Omega }=0$ $\ \mu $-a.s.
\end{thm}

\begin{proof}
(a) In view of Part (c) of Proposition~\ref{prop:SobolevembedsinHoelder}, we
may restrict ourselves to the case $1/2<s<t$. Note that $|L(p_{%
\blacktriangle })|<\infty $ by Assumption~\ref{dens:Element} and Part (a) of
Lemma~\ref{lem: p_0 is the unique maximizer}; also note that the random
variables $\log p_{\blacktriangle }(X_{i})$ are $\mathbb{P}$%
-a.s.~real-valued. By Kolmogorov's strong law of large numbers we then have%
\begin{equation}
\lim_{n\rightarrow \infty }|L_{n}(p_{\blacktriangle })-L(p_{\blacktriangle
})|=0\quad \mathbb{P}\text{-a.s.}  \label{Formel: Ln-L}
\end{equation}%
Let $\varepsilon _{l}$ be positive real numbers that converge monotonously
to $0$ as $l\rightarrow \infty $. Apply the uniform law of large numbers in
Part (d1) of Proposition~\ref{Properties of L_n and L_k} in Appendix \ref%
{App B} with $\mathcal{F}=\left\{ p+\varepsilon _{l}:p\in \mathcal{P}%
(t,\zeta ,D)\right\} $ to see that 
\begin{equation}
\lim_{n\rightarrow \infty }\sup_{p\in \mathcal{P}(t,\zeta
,D)}|L_{n}(p+\varepsilon _{l})-L(p+\varepsilon _{l})|=0\quad \mathbb{P}\text{%
-a.s.}  \label{Formel: sup Ln-L}
\end{equation}%
for every $l\in \mathbb{N}$. In the following arguments we fix an arbitrary
element of the probability $1$ event where the statements in (\ref{Formel:
Ln-L}) and (\ref{Formel: sup Ln-L}) hold true. We now prove that $\Vert \hat{%
p}_{n}-p_{\blacktriangle }\Vert _{s,2}$ converges to $0$ by showing that any
subsequence $\hat{p}_{n^{\prime }}$ of $\hat{p}_{n}$ has another subsequence
converging to $p_{\blacktriangle }$ in the Sobolev norm $\Vert \cdot \Vert
_{s,2}$. Because $\mathcal{P}(t,\zeta ,D)$ is compact in $\mathsf{W}%
_{2}^{s}(\Omega )$ by Proposition~\ref%
{Proposition:Propertiesoftheauxiliarymodel}, there is a subsequence $\hat{p}%
_{n^{\prime \prime }}$ of $\hat{p}_{n^{\prime }}$ and some $p^{\ast }\in 
\mathcal{P}(t,\zeta ,D)$ such that $\Vert \hat{p}_{n^{\prime \prime
}}-p^{\ast }\Vert _{s,2}$ converges to $0$. Now use Assumption~\ref%
{dens:Element}, the definition of $\hat{p}_{n^{\prime \prime }}$ as
maximizer, and the monotonicity of the logarithm to obtain%
\begin{eqnarray}
L_{n^{\prime \prime }}(p_{\blacktriangle }) &\leq &L_{n^{\prime \prime }}(%
\hat{p}_{n^{\prime \prime }})\leq L_{n^{\prime \prime }}(\hat{p}_{n^{\prime
\prime }}+\varepsilon _{l})  \notag \\
&\leq &L(\hat{p}_{n^{\prime \prime }}+\varepsilon _{l})+\sup_{p\in \mathcal{P%
}(t,\zeta ,D)}|L_{n^{\prime \prime }}(p+\varepsilon _{l})-L(p+\varepsilon
_{l})|.  \label{eq: Upper bound via ULLN}
\end{eqnarray}%
The first term on the r.h.s.~of (\ref{eq: Upper bound via ULLN}) converges
to $L(p^{\ast }+\varepsilon _{l})$ since $\Vert \hat{p}_{n^{\prime \prime
}}-p^{\ast }\Vert _{s,2}$, and hence also $\Vert \hat{p}_{n^{\prime \prime
}}-p^{\ast }\Vert _{\Omega }$, converges to $0$ and since $L(\cdot
+\varepsilon _{l})$ is sup-norm continuous on $\mathcal{P}(t,\zeta ,D)$ by
Part (c1) of Proposition~\ref{Properties of L_n and L_k} in Appendix \ref%
{App B}. The supremum on the r.h.s.~of (\ref{eq: Upper bound via ULLN}) goes
to $0$ and $L_{n^{\prime \prime }}(p_{\blacktriangle })$ converges to $%
L(p_{\blacktriangle })$ in view of (\ref{Formel: Ln-L}) and (\ref{Formel:
sup Ln-L}). It follows that 
\begin{equation}
L(p_{\blacktriangle })\leq L(p^{\ast }+\varepsilon _{l}).
\label{eq: L(p_0) less equal L(p*)}
\end{equation}%
The sequence of functions $\log (p^{\ast }+\varepsilon _{l})$ is
monotonously non-increasing in $l$ with pointwise limit $\log \,p^{\ast }$,
and is bounded above by the integrable function $\log (p^{\ast }+\varepsilon
_{1})$. Using the theorem of monotone convergence, we conclude from (\ref%
{eq: L(p_0) less equal L(p*)}) that $L(p_{\blacktriangle })\leq L(p^{\ast })$%
. Hence, $p^{\ast }=p_{\blacktriangle }$ by Part (b) of Lemma~\ref{lem: p_0
is the unique maximizer}.

(b) Follows analogously as Part (a) with $p_{\blacktriangle }$ replaced by $%
p_{\theta }$.

(c) As in the proof of Part (a), we may restrict ourselves to the case $%
1/2<s<t$. Define $\zeta ^{\#}=\inf_{\Omega \times \Theta }p(x,\theta )$. By
hypothesis, $\zeta ^{\#}>0$, and $\mathcal{P}(t,\zeta ^{\#},D)$ is non-empty
as it contains $\mathcal{P}_{\Theta }$. We may now apply Part (d2) of
Proposition~\ref{Properties of L_n and L_k} in Appendix \ref{App B} with $%
\mathcal{F}=\mathcal{P}(t,\zeta ^{\#},D)$ to get%
\begin{equation}
\lim_{k\rightarrow \infty }\sup_{\Theta \times \mathcal{P}(t,\zeta
^{\#},D)}|L_{k}(\theta ,p)-L(\theta ,p)|=0\quad \mu \text{-a.s.}  \label{f1}
\end{equation}%
Let $\varepsilon _{l}$ be as in the proof of Part (a). For each $l\in 
\mathbb{N}$, Part (d2) of Proposition~\ref{Properties of L_n and L_k} in
Appendix \ref{App B} with $\mathcal{F}=\left\{ p+\varepsilon _{l}:p\in 
\mathcal{P}(t,\zeta ,D)\right\} $ implies that 
\begin{equation}
\lim_{k\rightarrow \infty }\sup_{\theta \in \Theta }\sup_{p\in \mathcal{P}%
(t,\zeta ,D)}\left\vert L_{k}(\theta ,p+\varepsilon _{l})-L(\theta
,p+\varepsilon _{l})\right\vert =0\quad \mu \text{-a.s.}  \label{f2}
\end{equation}
In the following arguments we fix an arbitrary element of the probability $1$
event where (\ref{f1}) and (\ref{f2}) hold. Assume that $\sup_{\theta \in
\Theta }\Vert \tilde{p}_{k}(\theta )-p_{\theta }\Vert _{s,2}$ does not
converge to $0$. Then there is some $\eta >0$ such that for every $k\in 
\mathbb{N}$ there are $k^{\prime }\in \mathbb{N}$, $k^{\prime }\geq k$, and $%
\theta _{k^{\prime }}\in \Theta $ that satisfy 
\begin{equation}
\Vert \tilde{p}_{k^{\prime }}(\theta _{k^{\prime }})-p_{\theta _{k^{\prime
}}}\Vert _{s,2}>\eta .  \label{Greater eta}
\end{equation}%
By compactness of $\Theta $ and compactness of $\mathcal{P}(t,\zeta ,D)$ as
a subset of $\mathsf{W}_{2}^{s}(\Omega )$, we find a subsequence $\tilde{p}%
_{k^{\prime \prime }}(\theta _{k^{\prime \prime }})$ of $\tilde{p}%
_{k^{\prime }}(\theta _{k^{\prime }})$ such that $\theta _{k^{\prime \prime
}}$ converges to $\theta ^{\ast }$ for some $\theta ^{\ast }\in \Theta $,
and $\Vert \tilde{p}_{k^{\prime \prime }}(\theta _{k^{\prime \prime
}})-p^{\ast }\Vert _{s,2}$ converges to $0$ for some $p^{\ast }\in \mathcal{P%
}(t,\zeta ,D)$. So, if $p^{\ast }$ equals $p_{\theta ^{\ast }}$ (which we
verify below), then $\Vert \tilde{p}_{k^{\prime \prime }}(\theta _{k^{\prime
\prime }})-p_{\theta ^{\ast }}\Vert _{s,2}$ converges to $0$. Consequently, $%
\Vert \tilde{p}_{k^{\prime \prime }}(\theta _{k^{\prime \prime }})-p_{\theta
_{k^{\prime \prime }}}\Vert _{s,2}$ converges to $0$ because $p_{\theta
_{k^{\prime \prime }}}$ converges to $p_{\theta ^{\ast }}$ in $(\mathcal{P}%
(t,\zeta ,D),\Vert \cdot \Vert _{s,2})$ in view of Proposition~\ref%
{Interrelation} in Appendix \ref{App A} and Remark~\ref{Remark: Equivalence
of pointwise and sup-norm convergence of the parametrization}. This is in
contradiction to (\ref{Greater eta}) and therefore in contradiction to the
assumption that $\sup_{\theta \in \Theta }\Vert \tilde{p}_{k}(\theta
)-p_{\theta }\Vert _{s,2}$ does not converge to $0$.

It remains to show that $p^{\ast }$ equals $p_{\theta ^{\ast }}$. Use
Assumption \ref{mod:Inclusion}, the definition of $\tilde{p}_{k^{\prime
\prime }}(\theta _{k^{\prime \prime }})$ as maximizer, and the monotonicity
of the logarithm to obtain%
\begin{eqnarray}
L_{k^{\prime \prime }}(\theta _{k^{\prime \prime }},p_{\theta _{k^{\prime
\prime }}}) &\leq &L_{k^{\prime \prime }}(\theta _{k^{\prime \prime }},%
\tilde{p}_{k^{\prime \prime }}(\theta _{k^{\prime \prime }}))\leq
L_{k^{\prime \prime }}(\theta _{k^{\prime \prime }},\tilde{p}_{k^{\prime
\prime }}(\theta _{k^{\prime \prime }})+\varepsilon _{l})  \notag \\
&\leq &L(\theta _{k^{\prime \prime }},\tilde{p}_{k^{\prime \prime }}(\theta
_{k^{\prime \prime }})+\varepsilon _{l})  \notag \\
&&+\sup_{\theta \in \Theta }\sup_{p\in \mathcal{P}(t,\zeta ,D)}\left\vert
L_{k^{\prime \prime }}(\theta ,p+\varepsilon _{l})-L(\theta ,p+\varepsilon
_{l})\right\vert .  \label{Asses}
\end{eqnarray}%
The first term on the r.h.s.~of (\ref{Asses}) converges to $L(\theta ^{\ast
},p^{\ast }+\varepsilon _{l})$ since $\theta _{k^{\prime \prime }}$
converges to $\theta ^{\ast }$, $\Vert \tilde{p}_{k^{\prime \prime }}(\theta
_{k^{\prime \prime }})-p^{\ast }\Vert _{s,2}$, and hence also $\Vert \tilde{p%
}_{k^{\prime \prime }}(\theta _{k^{\prime \prime }})-p^{\ast }\Vert _{\Omega
}$, converges to $0$, and $L(\cdot ,\cdot +\varepsilon _{l})$ is a
continuous function on $\Theta \times (\mathcal{P}(t,\zeta ,D),\Vert \cdot
\Vert _{\Omega })$ by Part (c2) of Proposition~\ref{Properties of L_n and
L_k} in Appendix \ref{App B}. Recall that the supremum on the r.h.s.~of (\ref%
{Asses}) goes to $0$ in view of (\ref{f2}). Further, the supremum on the
r.h.s.~of the inequality%
\begin{eqnarray*}
\lefteqn{|L_{k^{\prime \prime }}(\theta _{k^{\prime \prime }},p_{\theta
_{k\prime \prime }})-L(\theta ^{\ast },p_{\theta ^{\ast }})|} \\
&\leq &\sup_{\Theta \times \mathcal{P}(t,\zeta ^{\#},D)}|L_{k^{\prime \prime
}}(\theta ,p)-L(\theta ,p)|+|L(\theta _{k^{\prime \prime }},p_{\theta
_{k^{\prime \prime }}})-L(\theta ^{\ast },p_{\theta ^{\ast }})|
\end{eqnarray*}%
converges to $0$ by (\ref{f1}). The second term on the r.h.s.~goes to $0$ as 
$\theta _{k^{\prime \prime }}$ converges to $\theta ^{\ast }$, $\Vert
p_{\theta _{k^{\prime \prime }}}-p_{\theta ^{\ast }}\Vert _{s,2}$, and hence
also $\Vert p_{\theta _{k^{\prime \prime }}}-p_{\theta ^{\ast }}\Vert
_{\Omega }$, converges to $0$, and $L(\theta ,p)$ is a continuous function
on $\Theta \times (\mathcal{P}(t,\zeta ^{\#},D),\Vert \cdot \Vert _{\Omega
}) $ by Part (c2) of Proposition~\ref{Properties of L_n and L_k} in Appendix %
\ref{App B}. Hence, the l.h.s.~of (\ref{Asses}) goes to $L(\theta ^{\ast
},p_{\theta ^{\ast }})$. It follows that%
\begin{equation}
L(\theta ^{\ast },p_{\theta ^{\ast }})\leq L(\theta ^{\ast },p^{\ast
}+\varepsilon _{l}).  \label{eq: L(theta*,p_theta*) less L(theta*,p*)}
\end{equation}%
The sequence of functions $\log \,(p^{\ast }+\varepsilon _{l})(\rho (\cdot
,\theta ^{\ast }))$ is monotonously non-increasing in $l$ with pointwise
limit $\log \,p^{\ast }(\rho (\cdot ,\theta ^{\ast }))$, and is bounded
above by the integrable function $\log (p^{\ast }+\varepsilon _{1})(\rho
(\cdot ,\theta ^{\ast }))$. Using the theorem of monotone convergence and (%
\ref{eq: L(theta*,p_theta*) less L(theta*,p*)}), we conclude that $L(\theta
^{\ast },p_{\theta ^{\ast }})\leq L(\theta ^{\ast },p^{\ast })$. Hence, $%
p^{\ast }=p_{\theta ^{\ast }}$ by Part (c) of Lemma~\ref{lem: p_0 is the
unique maximizer}.
\end{proof}

\begin{remark}
\normalfont\label{Remark 12} For later use we note the following: (i) Let
Assumption \ref{dens:Element} be satisfied, and suppose $\chi \geq 0$
satisfies $\inf_{x\in \Omega }p_{\blacktriangle }(x)>\chi $. It follows from
Part (a) of Theorem~\ref{Theorem: consistency of AML-estimators} that there
are events $A_{n}\in \mathcal{B}(\Omega )^{n}$ that have $\mathbb{P}^{n}$%
-probability tending to $1$ as $n\rightarrow \infty $ on which $\inf_{x\in
\Omega }\hat{p}_{n}(x)>\chi $ holds.

(ii) Let $p_{\theta }\in \mathcal{P}(t,\zeta ,D)$ for a given $\theta \in
\Theta $ be satisfied, and suppose $\chi (\theta )\geq 0$ satisfies $%
\inf_{x\in \Omega }p(x,\theta )>\chi (\theta )$ for the given $\theta $. It
follows from Part (b) of Theorem~\ref{Theorem: consistency of AML-estimators}
that for the given $\theta $ there are events $B_{k}(\theta )\in \mathcal{V}%
^{k}$ that have $\mu ^{k}$-probability tending to $1$ as $k\rightarrow
\infty $ on which $\inf_{x\in \Omega }\tilde{p}_{k}(\theta )(x)>\chi (\theta
)$ holds.

(iii) Let Assumptions~\ref{mod:Inclusion} and \ref{rho: Continuity of rho}
be satisfied, and suppose $\chi \geq 0$ satisfies $\inf_{\Omega \times
\Theta }p(x,\theta )>\chi $. It follows from Part (c) of Theorem~\ref%
{Theorem: consistency of AML-estimators} that there are events $B_{k}\in 
\mathcal{V}^{k}$ that have $\mu ^{k}$-probability tending to $1$ as $%
k\rightarrow \infty $ on which $\inf_{\theta \in \Theta }\inf_{x\in \Omega }%
\tilde{p}_{k}(\theta )(x)>\chi $ holds.
\end{remark}

\subsection{Rates of Convergence for NPML-Estimators\label{Rates}}

Following ideas of van~de~Geer~(1993), Nickl (2007, Proposition 6) obtained
convergence rates for the NPML-estimator $\hat{p}_{n}$ in various
Sobolev-norms as%
\begin{equation}
\Vert \hat{p}_{n}-p_{\blacktriangle }\Vert _{s,2}=O_{\mathbb{P}}^{\ast
}(n^{-(t-s)/(2t+1)})  \label{nickl}
\end{equation}%
for every $0\leq s\leq t$, provided Assumption \ref{dens:Element} and $\zeta
>0$ hold. Modulo measure-theoretic nuisances, this immediately gives an
analogous result for $\Vert \tilde{p}_{k}(\theta )-p_{\theta }\Vert _{s,2}$
for each $\theta \in \Theta $. [The complication here is that the result in
Nickl (2007) is proved for data generating processes defined as coordinate
projections on a product space, which is not the case for $X_{i}(\theta )$;
cf.~the proof of Part (b) of the subsequent proposition.] In Section \ref%
{Section: A uniform Donsker theorem for AML-estimators} below, however, we
shall need convergence rates for $\sup_{\theta \in \Theta }\Vert \tilde{p}%
_{k}(\theta )-p_{\theta }\Vert _{s,2}$, i.e., convergence rates that hold
uniformly w.r.t.~$\theta \in \Theta $. Before we turn to these uniform
results, we provide an extension of Nickl's (2007) rate result in that we
avoid the restriction $\zeta >0$. Note that Assumption \ref%
{dens:StrictInequality} already follows from Assumption \ref{dens:Element}
in case $\zeta >0$.

\begin{prop}
\label{pointwise_rate}(a) Under Assumptions \ref{dens:Element} and \ref%
{dens:StrictInequality} we have $\Vert \hat{p}_{n}-p_{\blacktriangle }\Vert
_{s,2}=O_{\mathbb{P}}(n^{-(t-s)/(2t+1)})$ for every $0\leq s\leq t$. (b) If $%
p_{\theta }\in \mathcal{P}(t,\zeta ,D)$ and $\inf_{x\in \Omega }p(x,\theta
)>0$ hold for a given $\theta \in \Theta $, then $\Vert \tilde{p}_{k}(\theta
)-p_{\theta }\Vert _{s,2}=O_{\mathbb{\mu }}(k^{-(t-s)/(2t+1)})$ for every $%
0\leq s\leq t$ and the given $\theta $.
\end{prop}

\begin{proof}
(a) Measurability of $\Vert \hat{p}_{n}-p_{\blacktriangle }\Vert _{s,2}$ is
established in Proposition \ref{Borel_1} in Appendix \ref{App D2}.\ The
result is trivial in case $s=t$ since $\mathcal{P}(t,\zeta ,D)$ is a bounded
subset of \emph{$\mathsf{W}$}$_{2}^{t}(\Omega )$. Hence assume $s<t$. If $%
\zeta >0$, the result follows from Proposition 6 in Nickl (2007). Now
suppose $\zeta =0$. By Assumption \ref{dens:StrictInequality} we can then
choose $\chi >0=\zeta $ such that $\inf_{x\in \Omega }p_{\blacktriangle
}(x)>\chi $ holds. By Remark \ref{Remark 12}(i) we have that $\hat{p}_{n}\in 
\mathcal{P}(t,\chi ,D)$ on events $A_{n}\in \mathcal{B}(\Omega )^{n}$ that
have probability tending to $1$ as $n\rightarrow \infty $. Since $\mathcal{P}%
(t,\chi ,D)\subseteq \mathcal{P}(t,\zeta ,D)$, the NPML-estimator $\hat{p}%
_{n}$ over $\mathcal{P}(t,\zeta ,D)$ coincides with the NPML-estimator over
the smaller set $\mathcal{P}(t,\chi ,D)$ on these events, and the latter
estimator satisfies (\ref{nickl}) by Proposition 6 in Nickl (2007).

(b) In view of (\ref{representation}) and since $(x_{1},\ldots
,x_{k})\mapsto \hat{p}_{k}(\cdot ;x_{1},\ldots ,x_{k})$ is a measurable
mapping from $\Omega ^{k}$ into $(\mathcal{P}(t,\zeta ,D),\Vert \cdot \Vert
_{\Omega })$, cf.~Theorem \ref{Theorem: existence of AML-estimators}, $%
\tilde{p}_{k}(\theta )$ has the same law as $\hat{p}_{k}(\cdot ;Z_{1},\ldots
,Z_{k})$, where $(Z_{1},\ldots ,Z_{k})$ has the same distribution as $%
(X_{1}(\theta ),\ldots ,X_{k}(\theta ))$ but the $Z_{i}$ are given by the
coordinate projections on $(\Omega ^{\mathbb{N}},\mathcal{B}(\Omega )^{%
\mathbb{N}})$. Since $\Vert \cdot \Vert _{\Omega }$ and $\Vert \cdot \Vert
_{s,2}$ for $s\leq t$ generate the same Borel $\sigma $-field on $\mathcal{P}%
(t,\zeta ,D)$ (cf.~Lemma \ref{Borel} in Appendix \ref{App D2}), $\Vert 
\tilde{p}_{k}(\theta )-p_{\theta }\Vert _{s,2}$ is measurable and has the
same distribution as $\Vert \hat{p}_{k}(\cdot ;Z_{1},\ldots
,Z_{k})-p_{\theta }\Vert _{s,2}$. Now apply the already established Part (a)
to $\hat{p}_{k}(\cdot ;Z_{1},\ldots ,Z_{k})$.
\end{proof}

\bigskip

In case $s=t$, in fact $\Vert \hat{p}_{n}-p_{\blacktriangle }\Vert
_{s,2}\leq 2D$ and $\Vert \tilde{p}_{k}(\theta )-p_{\theta }\Vert _{s,2}\leq
2D$ hold under the assumptions of the above proposition. The next
proposition is instrumental in proving the uniform-in-$\theta $ convergence
rate result.

\begin{prop}
\label{Bracketing entropy of F*} Let $\mathcal{F}$ be a (non-empty) bounded
subset of \emph{$\mathsf{W}$}$_{2}^{s}(\Omega )$ with $s>1/2$. Suppose
Assumption \ref{rho: Hoelderity of rho} holds.

(a) Then the $\mathcal{L}^{2}(\mu )$-bracketing metric entropy of%
\begin{equation*}
\mathcal{F}^{\ast }=\{f(\rho (\cdot ,\theta )):\theta \in \Theta ,\,f\in 
\mathcal{F}\}
\end{equation*}%
satisfies%
\begin{equation}
H_{[\hspace{0.75ex}]}(\varepsilon ,\mathcal{F}^{\ast },\Vert \cdot \Vert
_{2,\mu })\lesssim \varepsilon ^{-1/s}.
\label{bracketing metric entropy of F*}
\end{equation}%
In particular, $\mathcal{F}^{\ast }$ is $\mu $-Donsker.

(b) Suppose the elements of $\mathcal{F}$ are bounded below by some $\chi >0$%
. Then the $\mathcal{L}^{2}(\mu )$-bracketing metric entropy of%
\begin{equation*}
\log \mathcal{F}^{\ast }=\left\{ \log f(\rho (\cdot ,\theta )):\theta \in
\Theta ,\,f\in \mathcal{F}\right\}
\end{equation*}%
satisfies%
\begin{equation*}
H_{[\hspace{0.75ex}]}(\varepsilon ,\log \mathcal{F}^{\ast },\Vert \cdot
\Vert _{2,\mu })\lesssim \varepsilon ^{-1/s}.
\end{equation*}
\end{prop}

We note that in the subsequent uniform-in-$\theta $ convergence rate result
Assumption \ref{mod: Strict inequality} already follows from Assumption \ref%
{mod:Inclusion} in case $\zeta >0$.

\begin{thm}
\label{Uniform rate of AML-estimators}\hspace*{0ex}Let Assumptions~\ref%
{mod:Inclusion}, \ref{mod: Strict inequality}, and \ref{rho: Hoelderity of
rho} be satisfied. Then 
\begin{equation}
\sup_{\theta \in \Theta }\Vert \tilde{p}_{k}(\theta )-p_{\theta }\Vert
_{s,2}=O_{\mu }(k^{-(t-s)/(2t+1)})\quad \text{as }k\rightarrow \infty
\label{rate of sup}
\end{equation}%
for every $0\leq s<t$. [In case $s=t$, the above supremum is bounded by $2D$%
.]
\end{thm}

\begin{proof}
Measurability of $\sup_{\theta \in \Theta }\Vert \tilde{p}_{k}(\theta
)-p_{\theta }\Vert _{s,2}$ for $0\leq s<t$ is established in Proposition \ref%
{Borel_1} in Appendix \ref{App D2}.\ The claim in parentheses follows since $%
\tilde{p}_{k}(\theta )\in \mathcal{P}(t,\zeta ,D)$ by construction and $%
p_{\theta }\in \mathcal{P}(t,\zeta ,D)$ by\ Assumption~\ref{mod:Inclusion}.
We now distinguish two cases:

Case 1: Assume first that $\zeta >0$ and $s=0$. We then verify the
conditions of Theorem~\ref{Uniform rates of convergence: parameter set} in
Appendix \ref{App E} with $(\Lambda ,\mathcal{A},P)=(V^{\mathbb{N}},\mathcal{%
V}^{\mathbb{N}},\mu ^{\mathbb{N}})$, $S=\Theta $, $T=\mathcal{P}(t,\zeta ,D)$%
, $d(p,q)=\Vert p-q\Vert _{2}$, $H_{k}(\sigma ,\tau )=L_{k}(\theta ,p)$, $%
H(\sigma ,\tau )=L(\theta ,p)$, $\hat{\tau}_{k}(\sigma )=\tilde{p}%
_{k}(\theta )$, and $\tau (\sigma )=p_{\theta }$. Condition (\ref{Maximizer
condition}) is satisfied by definition of the NPML-estimators $\tilde{p}%
_{k}(\theta )$. Condition (\ref{Concavity}) follows from the second-order
Taylor expansion of $L(\theta ,\cdot )$ around the density $p_{\theta }$:
using Proposition\thinspace \ref{prop:DerivativesofthescaledlikelihoodNickl}
in Appendix \ref{App B} we obtain%
\begin{eqnarray*}
L(\theta ,p)-L(\theta ,p_{\theta }) &=&\mathbf{D}L(\theta ,p_{\theta
})(p-p_{\theta })+\frac{1}{2}\mathbf{D}^{2}L(\theta ,\bar{p})(p-p_{\theta
},p-p_{\theta }) \\
&=&-\frac{1}{2}\int_{\Omega }\frac{(p-p_{\theta })^{2}}{\bar{p}^{2}}%
p_{\theta }d\lambda \leq -\frac{1}{2}\zeta \left( C_{t}D\right) ^{-2}\Vert
p-p_{\theta }\Vert _{2}^{2},
\end{eqnarray*}%
where $\bar{p}$ is some density on the line segment joining $p$ and $%
p_{\theta }$; note that $\bar{p}\in \mathcal{P}(t,\zeta ,D)$ by convexity of
this set, and hence satisfies $\Vert \bar{p}\Vert _{\Omega }\leq C_{t}D$.
This proves condition (\ref{Concavity}) in Theorem~\ref{Uniform rates of
convergence: parameter set} with $C=2^{-1}\zeta \left( C_{t}D\right) ^{-2}$
and $\alpha =2$, both constants being independent of $\theta $ and $p$.

Next we verify condition~(\ref{Modulus of continuity}): set 
\begin{equation*}
\mathcal{G}_{\delta }=\left\{ \log p(\rho (\cdot ,\theta ))-\log p_{\theta
}(\rho (\cdot ,\theta )):\theta \in \Theta ,\,p\in \mathcal{P}(t,\zeta
,D),\,\Vert p-p_{\theta }\Vert _{2}\leq \delta \right\}
\end{equation*}%
for $\delta >0$, which is clearly non-empty. Then clearly%
\begin{equation*}
\func{E}^{\ast }\sup_{\theta \in \Theta }\,\sup_{\substack{ p\in \mathcal{P}%
(t,\zeta ,D),  \\ \Vert p-p_{\theta }\Vert _{2}\leq \delta }}\left\vert 
\sqrt{k}(L_{k}-L)(\theta ,p)-\sqrt{k}(L_{k}-L)(\theta ,p_{\theta
})\right\vert =\func{E}^{\ast }\left\Vert \sqrt{k}(\mu _{k}-\mu )\right\Vert
_{\mathcal{G}_{\delta }}
\end{equation*}%
where $\func{E}^{\ast }$ denotes the outer expectation. Since we have
temporarily assumed $\zeta >0$, the logarithm is Lipschitz on $[\zeta
,\infty )$ with Lipschitz constant $\zeta ^{-1}$. This implies that $%
\mathcal{G}_{\delta }$ is bounded by $B:=2\zeta ^{-1}C_{t}D$ in the sup-norm
and by $\eta (\delta ):=\zeta ^{-1}C_{t}^{1/2}D^{1/2}\delta $ in the $%
\mathcal{L}^{2}(\mu )$-norm. Consequently, 
\begin{equation*}
\func{E}^{\ast }\left\Vert \sqrt{k}(\mu _{k}-\mu )\right\Vert _{\mathcal{G}%
_{\delta }}\leq (1696+64\sqrt{2})\,I_{[\hspace{0.75ex}]}(\eta (\delta ),%
\mathcal{G}_{\delta },\Vert \cdot \Vert _{2,\mu })\left[ 1+\frac{B}{\eta
(\delta )^{2}\sqrt{k}}\,I_{[\hspace{0.75ex}]}(\eta (\delta ),\mathcal{G}%
_{\delta },\Vert \cdot \Vert _{2,\mu })\right]
\end{equation*}%
by Theorem~\ref{Van der Vaart lemma} in Appendix \ref{App E}. Since 
\begin{eqnarray*}
\mathcal{G}_{\delta } &\subseteq &\left\{ \log p(\rho (\cdot ,\theta ))-\log
p_{\theta }(\rho (\cdot ,\theta )):\theta \in \Theta ,\,p\in \mathcal{P}%
(t,\zeta ,D)\right\} \\
&\subseteq &\left\{ \log p(\rho (\cdot ,\theta )):\theta \in \Theta ,\,p\in 
\mathcal{P}(t,\zeta ,D)\right\} -\left\{ \log p(\rho (\cdot ,\theta
)):\theta \in \Theta ,\,p\in \mathcal{P}(t,\zeta ,D)\right\} \!,
\end{eqnarray*}%
we have that%
\begin{equation*}
N_{[\hspace{0.75ex}]}(\varepsilon ,\mathcal{G}_{\delta },\Vert \cdot \Vert
_{2,\mu })\leq N_{[\hspace{0.75ex}]}(\varepsilon /2,\left\{ \log p(\rho
(\cdot ,\theta )):\theta \in \Theta ,\,p\in \mathcal{P}(t,\zeta ,D)\right\}
,\Vert \cdot \Vert _{2,\mu })^{2}.
\end{equation*}%
Applying Proposition~\ref{Bracketing entropy of F*}(b) with $s=t$ and $%
\mathcal{F}=\mathcal{P}(t,\zeta ,D)$ we get from this inequality%
\begin{eqnarray*}
I_{[\hspace{0.75ex}]}(\eta (\delta ),\mathcal{G}_{\delta },\Vert \cdot \Vert
_{2,\mu }) &\lesssim &\int_{(0,\eta (\delta )]}\sqrt{1+\varepsilon ^{-1/t}}%
d\varepsilon \lesssim \max (\eta (\delta ),\int_{(0,\eta (\delta
)]}\varepsilon ^{-1/2t}d\varepsilon ) \\
&\lesssim &\max (\delta ,\delta ^{1-1/2t}).
\end{eqnarray*}%
Hence there is some constant $L$, $0<L<\infty $, such that 
\begin{equation*}
\func{E}^{\ast }\left\Vert \sqrt{k}(\mu _{k}-\mu )\right\Vert _{\mathcal{G}%
_{\delta }}\leq L\max (\delta ,\delta ^{1-1/2t})\left[ 1+\frac{\max (\delta
,\delta ^{1-1/2t})}{\delta ^{2}\sqrt{k}}\right]
\end{equation*}%
holds for all $\delta >0$. Write $\varphi _{k}(\delta )$ for the r.h.s.~of
the last display and note that $\delta \mapsto \delta ^{-\beta }\varphi
_{k}(\delta )$ is non-increasing for $\beta =1$. This establishes condition (%
\ref{Modulus of continuity}) in Theorem~\ref{Uniform rates of convergence:
parameter set}.

Condition (\ref{Rate of convergence condition}) in that theorem is satisfied
for $\alpha =2$ and $r_{k}=k^{t/(2t+1)}$. This gives the desired rate and
completes the proof in case $\zeta >0$ and $s=0$. Now suppose $\zeta >0$ but 
$0<s<t$. Recall that $\sup_{\theta \in \Theta }\Vert \tilde{p}_{k}(\theta
)-p_{\theta }\Vert _{t,2}\leq 2D$. The result then follows from the
interpolation inequality 
\begin{equation*}
\Vert f\Vert _{s,2}\leq C_{s,t}\,\Vert f\Vert _{t,2}^{s/t}\,\Vert f\Vert
_{2}^{(t-s)/t}
\end{equation*}%
for $f\in \QTR{up}{\mathsf{W}}_{2}^{t}(\Omega )$, where $C_{s,t}>0$; see
Theorem~1.9.6 and Remark~1.9.1 in Lions and Magenes (1972).

Case 2: Suppose now $\zeta =0$ and $0\leq s<t$. In view of Assumption \ref%
{mod: Strict inequality} we may choose $\chi >0$ such that $\inf_{\Omega
\times \Theta }p(x,\theta )>\chi $. Then, by Remark~\ref{Remark 12}(iii),
there are events that have probability tending to $1$ on which $\inf_{\theta
\in \Theta }\,\inf_{x\in \Omega }\tilde{p}_{k}(\theta )(x)>\chi $ holds
true. Since $\mathcal{P}(t,\chi ,D)\subseteq \mathcal{P}(t,\zeta ,D)$, we
have that on these events $\tilde{p}_{k}(\theta )$ coincides with the
NPML-estimators over the smaller set $\mathcal{P}(t,\chi ,D)$. The result
now follows from what has already been established in Case 1 since
Assumption \ref{mod:Inclusion} (and \ref{mod: Strict inequality}) is also
satisfied with respect to $\mathcal{P}(t,\chi ,D)$.
\end{proof}

\subsection{Donsker-type Theorems for NPML-Estimators\label{Section: A
uniform Donsker theorem for AML-estimators}}

Nickl (2007) established Part (a) of the following Donsker-type result under
the additional assumption that $\zeta >0$ holds. Part (b) is (modulo
measure-theoretic nuisances) a simple consequence of Part (a).

\begin{thm}
\label{Theorem: Nickl}Let $\mathcal{F}$ be a non-empty bounded subset of $%
\mathsf{W}_{2}^{s}(\Omega )$ for some $s>1/2$.

(a) Suppose Assumption~\ref{dens:InternalPoint} is satisfied. Then, for all
real $j>1/2$,%
\begin{equation}
\sup_{f\in \mathcal{F}}\left\vert \sqrt{n}\int_{\Omega }(\hat{p}%
_{n}-p_{\blacktriangle })fd\lambda -\sqrt{n}(\mathbb{P}_{n}-\mathbb{P}%
)f\right\vert =o_{\mathbb{P}}(n^{-(\min (s,t)-j)/(2t+1)})  \label{UCLT_0}
\end{equation}%
as $n\rightarrow \infty $; in particular, the l.h.s.~of the above display is 
$o_{\mathbb{P}}(1)$ as $n\rightarrow \infty $. Consequently, the stochastic
process $f\mapsto \sqrt{n}\int_{\Omega }(\hat{p}_{n}-p_{\blacktriangle
})fd\lambda $ converges weakly to a $\mathbb{P}$-Brownian bridge in $\ell
^{\infty }(\mathcal{F})$.

(b) Suppose $p_{\theta }\in \mathcal{P}(t,\zeta ,D)$, $\inf_{x\in \Omega
}p(x,\theta )>\zeta $, and $\Vert p_{\theta }\Vert _{t,2}<D$ hold for a
given $\theta \in \Theta $. Then, for the given $\theta $, a result
analogous to Part (a) holds for the process $f\mapsto \sqrt{k}\int_{\Omega }(%
\tilde{p}_{k}(\theta )-p_{\theta })fd\lambda $ with $\mathbb{P}_{k}$ and $%
\mathbb{P}$, respectively, replaced by $\mathbb{P}_{\theta ,k}$ and $\mathbb{%
P}_{\theta }$, where $\mathbb{P}_{\theta ,k}$ is the empirical measure of $%
X_{1}(\theta ),\ldots ,X_{k}(\theta )$ and $\mathbb{P}_{\theta }$ is the
probability measure corresponding to $p_{\theta }$.
\end{thm}

\begin{proof}
(a) Measurability of the l.h.s.~of (\ref{UCLT_0}) follows from Proposition %
\ref{Borel_2} in Appendix \ref{App D2}. For $\zeta >0$ the result follows
immediately from Theorem 3 in Nickl (2007). Now suppose $\zeta =0$. In view
of Assumption \ref{dens:InternalPoint} we may choose $\chi >0$ such that $%
\inf_{x\in \Omega }p_{\blacktriangle }(x)>\chi $. Then, by Remark~\ref%
{Remark 12}(i), there are events that have probability tending to $1$ on
which $\inf_{x\in \Omega }\hat{p}_{n}(x)>\chi $ holds true. Since $\mathcal{P%
}(t,\chi ,D)\subseteq \mathcal{P}(t,\zeta ,D)=\mathcal{P}(t,0,D)$, we have
that on these events $\hat{p}_{n}$ coincides with the NPML-estimators over
the smaller set $\mathcal{P}(t,\chi ,D)$. Since $\chi >0$ and since
Assumption \ref{dens:InternalPoint} is also satisfied relative to $\mathcal{P%
}(t,\chi ,D)$, the result now follows from what has already been established.

(b) Note that $\mathfrak{X}_{k}(\breve{x},f)$ and $\sup_{f\in \mathcal{F}%
}\left\vert \mathfrak{X}_{k}(\breve{x},f)-\mathfrak{Y}_{k}(\breve{x}%
,f)\right\vert $ defined in Proposition \ref{Borel_2}(a) in Appendix \ref%
{App D2} are Borel measurable on $\Omega ^{k}$. Consequently, 
\begin{equation*}
\sup_{f\in \mathcal{F}}\left\vert \mathfrak{X}_{k}(X_{1}(\theta ),\ldots
,X_{k}(\theta ),f)-\mathfrak{Y}_{k}(X_{1}(\theta ),\ldots ,X_{k}(\theta
),f)\right\vert
\end{equation*}%
and 
\begin{equation*}
\sup_{f\in \mathcal{F}}\left\vert \mathfrak{X}_{k}(Z_{1},\ldots ,Z_{k},f)-%
\mathfrak{Y}_{k}(Z_{1},\ldots ,Z_{k},f)\right\vert
\end{equation*}%
have the same distribution, where the $Z_{i}$ are as in the proof of
Proposition \ref{pointwise_rate}. Furthermore, it follows that the
finite-dimensional distributions of the processes $f\mapsto \mathfrak{X}%
_{k}(X_{1}(\theta ),\ldots ,X_{k}(\theta ),f)$ and $f\mapsto \mathfrak{X}%
_{k}(Z_{1},\ldots ,Z_{k},f)$ coincide. It is easy to see that the maps $%
f\rightarrow \mathfrak{X}_{k}(\breve{x},f)$ belong to $\mathsf{C}^{0}(%
\mathcal{F},\Vert \cdot \Vert _{\Omega })$, the space of bounded uniformly
continuous functions on $(\mathcal{F},\Vert \cdot \Vert _{\Omega })$.
Consequently, $\mathfrak{X}_{k}(\breve{x},\cdot )$ is Borel measurable as a
random element in $\mathsf{C}^{0}(\mathcal{F},\Vert \cdot \Vert _{\Omega })$%
, since the Borel $\sigma $-field on this space is generated by the
point-evaluations (observe that $(\mathcal{F},\Vert \cdot \Vert _{\Omega })$
is totally bounded in view of Lemma \ref{Lemma: bounded subsets are
uniformly Donsker} in Appendix \ref{App D}). Since $\mathsf{C}^{0}(\mathcal{F%
},\Vert \cdot \Vert _{\Omega })$ is Polish by total boundedness of $(%
\mathcal{F},\Vert \cdot \Vert _{\Omega })$, the entire laws of the processes 
$f\mapsto \mathfrak{X}_{k}(X_{1}(\theta ),\ldots ,X_{k}(\theta ),f)$ and $%
f\mapsto \mathfrak{X}_{k}(Z_{1},\ldots ,Z_{k},f)$ on $\mathsf{C}^{0}(%
\mathcal{F},\Vert \cdot \Vert _{\Omega })$, and hence on $\ell ^{\infty }(%
\mathcal{F})$, coincide. In view of (\ref{representation}), Part (b) now
follows from applying the already established Part (a) to $\hat{p}_{k}(\cdot
;Z_{1},\ldots ,Z_{k})$.
\end{proof}

\bigskip

The next theorem shows that a weak limit theorem for the stochastic process $%
(\theta ,f)\mapsto \sqrt{k}\int_{\Omega }(\tilde{p}_{k}(\theta )-p_{\theta
})fd\lambda $ can be obtained even in the space $\ell ^{\infty }(\Theta
\times \mathcal{F})$. A corollary of this is then a uniform-in-$\theta $
version of Part (b) of the above theorem. The proof of this theorem largely
follows the ideas in Nickl (2007): Loosely speaking, a mean-value expansion
of $\mathbf{D}L_{k}(\theta ,\tilde{p}_{k}(\theta ))(\cdot )$, analogous to
the one in the classical parametric case, shows that this can be represented
as the sum of the score evaluated at the true density $p_{\theta }$, i.e., $%
\mathbf{D}L_{k}(\theta ,p_{\theta })(\cdot )$, plus a second derivative term
applied to the estimation error $(\tilde{p}_{k}(\theta )-p_{\theta },\cdot )$%
. [For given $\theta \in \Theta $, the Fr\'{e}chet-derivative of $L_{k}$
with respect to the second argument is here denoted by $\mathbf{D}%
L_{k}(\theta ,\cdot )$.] The score, evaluated at the true density $p_{\theta
}$ and properly scaled, turns out to be an empirical process having a
Gaussian limit. The second derivative term turns out to coincide with $%
-\int_{\Omega }(\tilde{p}_{k}(\theta )-p_{\theta })fd\lambda $ up to
negligible terms. [An important ingredient for establishing negligibility
are the uniform-in-$\theta $ convergence rates for $\tilde{p}_{k}(\theta )$
in different Sobolev norms that have been established in the previous
section.] Apart from a series of technical difficulties not present in the
classical parametric case, the major difficulty is then the following: in
the classical parametric case the usual assumption that the true parameter
belongs to the interior of the parameter space together with consistency
implies that the estimator is eventually an interior point, implying that
the score evaluated at the maximizer is zero. In the present case, while $%
p_{\theta }$ is an interior point of $\mathcal{P}(t,\zeta ,D)$ relative to $%
\mathsf{H}_{t}$ as a consequence of the assumptions underlying Theorem \ref%
{Theorem: Uniform Donsker-type theorem}, the estimator $\tilde{p}_{k}(\theta
)$ is, however, \emph{not} an interior point of the domain $\mathcal{P}%
(t,\zeta ,D)$ (relative to $\mathsf{H}_{t}$) over which optimization is
performed, as shown in Theorem \ref{Theorem: existence of AML-estimators};
in particular, $\tilde{p}_{k}(\theta )$ is \emph{not} consistent w.r.t.~the $%
\left\Vert \cdot \right\Vert _{t,2}$-norm. As a consequence, one can not
conclude that the score evaluated at the maximizer is zero. [Trying to save
this argument directly by using an $\left\Vert \cdot \right\Vert _{s,2}$%
-norm with $s<t$ does not work either: while $\tilde{p}_{k}(\theta )$ is
consistent in the $\left\Vert \cdot \right\Vert _{s,2}$-norm, $p_{\theta }$
is then \emph{not} an interior point of $\mathcal{P}(t,\zeta ,D)$ relative
to $\mathsf{H}_{s}$.] Hence, a different reasoning is needed to show that $%
\mathbf{D}L_{k}(\theta ,\tilde{p}_{k}(\theta ))(\cdot )$, although not
necessarily zero, is of sufficiently small order. This is provided in the
subsequent lemma, which is essentially a uniform version of Lemma 4 in Nickl
(2007). The proof as given below makes use of Proposition \ref{prop_interior}
which allows us to simplify the arguments given in Nickl (2007). In the
following lemma let $\mathsf{H}_{t}^{0}$ denote the linear subspace of $%
\mathsf{W}_{2}^{t}(\Omega )$ that is parallel to $\mathsf{H}_{t}$.

\begin{lem}
\label{derivativeorder}Suppose Assumptions~\ref%
{mod:InclusionWithStrictInequalities} and \ref{rho: Hoelderity of rho} are
satisfied and $\zeta >0$ holds. Let $\mathcal{G}$ be a non-empty bounded
subset of $\mathsf{H}_{t}^{0}\subseteq \mathsf{W}_{2}^{t}(\Omega )$. Then 
\begin{equation}
\sup_{\theta \in \Theta }\,\sup_{g\in \mathcal{G}}\left\vert \mathbf{D}%
L_{k}(\theta ,\tilde{p}_{k}(\theta ))(g)\right\vert =o_{\mu
}(k^{-(t-j)/(2t+1)-1/2})  \label{order of derivative}
\end{equation}%
for every real $j>1/2$.
\end{lem}

\begin{proof}
Measurability of the l.h.s.~of (\ref{order of derivative}) follows from
Proposition \ref{Borel_2}(c) in Appendix \ref{App D2}. W.l.o.g. we may
assume $1/2<j<t$. By Assumption \ref{mod:InclusionWithStrictInequalities}
and Proposition \ref{prop_interior}(b) we can find $\delta >0$ small enough
such that 
\begin{equation*}
p_{\theta }+w\in \mathcal{P}(t,\zeta ,D)
\end{equation*}%
holds for every $\theta \in \Theta $ and every $w\in \mathcal{U}_{t,\delta
}\cap \mathsf{H}_{t}^{0}$. Note that $\delta $ does not depend on $\theta $.
Since $\tilde{p}_{k}(\theta )$ maximizes $L_{k}(\theta ,\cdot )$ (which is
differentiable in view of Proposition \ref%
{prop:DerivativesofthescaledlikelihoodNickl} as $\zeta >0$ is assumed) over $%
\mathcal{P}(t,\zeta ,D)$ we conclude that 
\begin{equation*}
\mathbf{D}L_{k}(\theta ,\tilde{p}_{k}(\theta ))(p_{\theta }+w-\tilde{p}%
_{k}(\theta ))\leq 0
\end{equation*}%
holds for all $\theta \in \Theta $ and all $w\in \mathcal{U}_{t,\delta }\cap 
\mathsf{H}_{t}^{0}$. This implies 
\begin{equation*}
\mathbf{D}L_{k}(\theta ,\tilde{p}_{k}(\theta ))(w)\leq \mathbf{D}%
L_{k}(\theta ,\tilde{p}_{k}(\theta ))(\tilde{p}_{k}(\theta )-p_{\theta })
\end{equation*}%
for all $\theta \in \Theta $ and $w\in \mathcal{U}_{t,\delta }\cap \mathsf{H}%
_{t}^{0}$. Since $\mathcal{U}_{t,\delta }\cap \mathsf{H}_{t}^{0}$ is
invariant under multiplication by $-1$, we obtain%
\begin{eqnarray*}
\sup_{\theta \in \Theta }\sup_{w\in \mathcal{U}_{t,\delta }\cap \mathsf{H}%
_{t}^{0}}\left\vert \mathbf{D}L_{k}(\theta ,\tilde{p}_{k}(\theta
))(w)\right\vert &\leq &\sup_{\theta \in \Theta }\left\vert \mathbf{D}%
L_{k}(\theta ,\tilde{p}_{k}(\theta ))(\tilde{p}_{k}(\theta )-p_{\theta
})\right\vert \\
&\leq &\sup_{\theta \in \Theta }\left\vert (\mathbf{D}L_{k}(\theta ,\tilde{p}%
_{k}(\theta ))-\mathbf{D}L(\theta ,\tilde{p}_{k}(\theta )))(\tilde{p}%
_{k}(\theta )-p_{\theta })\right\vert \\
&&+\sup_{\theta \in \Theta }\left\vert (\mathbf{D}L(\theta ,\tilde{p}%
_{k}(\theta ))-\mathbf{D}L(\theta ,p_{\theta }))(\tilde{p}_{k}(\theta
)-p_{\theta })\right\vert \\
&\leq &\sup_{\theta \in \Theta }\left\Vert \tilde{p}_{k}(\theta )-p_{\theta
}\right\Vert _{j,2}\sup_{\Theta \times \mathcal{P}(t,\zeta ,D)}\left\Vert 
\mathbf{D}L_{k}(\theta ,p)-\mathbf{D}L(\theta ,p)\right\Vert _{\mathcal{U}%
_{j,1}} \\
&&+\zeta ^{-1}\sup_{\theta \in \Theta }\left\Vert \tilde{p}_{k}(\theta
)-p_{\theta }\right\Vert _{2}^{2},
\end{eqnarray*}%
where we have repeatedly used Proposition~\ref%
{prop:DerivativesofthescaledlikelihoodNickl}, in particular to establish
that $\mathbf{D}L(\theta ,p_{\theta }))(\tilde{p}_{k}(\theta )-p_{\theta
})=0 $. Now use Theorem~\ref{Uniform rate of AML-estimators} and Proposition~%
\ref{Closeness of the derivatives of the scaled likelihood} with $\alpha =1$
and $\mathcal{H}_{1}=\mathcal{U}_{j,1}$ to conclude that the r.h.s.\ of the
last display is 
\begin{equation*}
O_{\mu }(k^{-(t-j)/(2t+1)-1/2})+O_{\mu }(k^{-2t/(2t+1)})=O_{\mu
}(k^{-(t-j)/(2t+1)-1/2})
\end{equation*}%
since $j>1/2$. A fortiori this holds for all $j>1/2$ and thus proves the
result for the case where $\mathcal{G}$ is contained in $\mathcal{U}%
_{t,\delta }\cap \mathsf{H}_{t}^{0}$. Since (\ref{order of derivative}) is
homogenous w.r.t.~scaling of $\mathcal{G}$ and since $\delta $ does not
depend on $\mathcal{G}$, the just mentioned inclusion can, however, always
be achieved by rescaling.
\end{proof}

\bigskip

We note that the lemma can easily be extended to the case $\zeta =0$ by
making use of Remark \ref{Remark 12}(iii). The main result is now the
following.

\begin{thm}
\label{Theorem: Uniform Donsker-type theorem} Suppose Assumptions~\ref%
{mod:InclusionWithStrictInequalities} and \ref{rho: Hoelderity of rho} are
satisfied. Let $\mathcal{F}$ be a non-empty bounded subset of $\mathsf{W}%
_{2}^{s}(\Omega )$ for some $s>1/2$. Then:

(a) For all real $j>1/2$,%
\begin{equation}
\sup_{\theta \in \Theta }\,\sup_{f\in \mathcal{F}}\left\vert \sqrt{k}%
\int_{\Omega }(\tilde{p}_{k}(\theta )-p_{\theta })fd\lambda -\sqrt{k}(\mu
_{k}-\mu )f(\rho (\cdot ,\theta ))\right\vert =o_{\mu }(k^{-(\min
(s,t)-j)/(2t+1)})  \label{UCLT}
\end{equation}%
as $k\rightarrow \infty $; in particular, the l.h.s.~of the above display is 
$o_{\mu }(1)$ as $k\rightarrow \infty $.

(b) There exists a zero-mean Gaussian process $\mathbb{G}$ indexed by $%
\Theta \times \mathcal{F}$ with bounded sample paths such that the
stochastic process $(\theta ,f)\mapsto \sqrt{k}\int_{\Omega }(\tilde{p}%
_{k}(\theta )-p_{\theta })fd\lambda $ converges weakly to $\mathbb{G}(\theta
,f)$ in $\ell ^{\infty }(\Theta \times \mathcal{F})$. The process $\mathbb{G}
$ is measurable as a mapping with values in $\ell ^{\infty }(\Theta \times 
\mathcal{F})$, has separable range, and has sample paths that are uniformly
continuous with respect to the pseudo-metric $d((\theta ,f),(\theta ^{\prime
},g))=\left( \limfunc{Var}[\mathbb{G}(\theta ,f)-\mathbb{G}(\theta ^{\prime
},g)]\right) ^{1/2}$. Its covariance function is given by%
\begin{equation*}
\limfunc{Cov}[\mathbb{G}(\theta ,f),\mathbb{G}(\theta ^{\prime
},g)]=\int_{V}\left( f(\rho (\cdot ,\theta ))-\int_{V}f(\rho (\cdot ,\theta
))d\mu \right) \left( g(\rho (\cdot ,\theta ^{\prime }))-\int_{V}g(\rho
(\cdot ,\theta ^{\prime }))d\mu \right) d\mu .
\end{equation*}

(c)%
\begin{equation*}
\sup_{\theta \in \Theta }\,\sup_{f\in \mathcal{F}}\left\vert \sqrt{k}%
\int_{\Omega }(\tilde{p}_{k}(\theta )-p_{\theta })fd\lambda \right\vert
=O_{\mu }(1)\quad \text{as }k\rightarrow \infty .
\end{equation*}
\end{thm}

\begin{proof}
Part (a):\textbf{\ }Measurability of the l.h.s.~of (\ref{UCLT}) follows from
Proposition \ref{Borel_2}(b) in Appendix \ref{App D2}.

\emph{Step~1: }We first consider the case $\zeta >0$. Let $\mathcal{G}$ be a
non-empty bounded subset of $\mathsf{H}_{t}^{0}$. Applying the pathwise
mean-value theorem to the function $\mathbf{D}L_{k}(\theta ,\cdot )(g)$,
adding and subtracting a term, and using Proposition~\ref%
{prop:DerivativesofthescaledlikelihoodNickl} leads to%
\begin{eqnarray*}
\mathbf{D}L_{k}(\theta ,\tilde{p}_{k}(\theta ))(g) &=&\mathbf{D}L_{k}(\theta
,p_{\theta })(g)+\mathbf{D}^{2}L_{k}(\theta ,\bar{p}_{k}(\theta ))(\tilde{p}%
_{k}(\theta )-p_{\theta },g) \\
&=&\left( \mu _{k}-\mu \right) (p_{\theta }^{-1}g)(\rho (\cdot ,\theta ))+%
\mathbf{D}^{2}L(\theta ,p_{\theta })(\tilde{p}_{k}(\theta )-p_{\theta },g) \\
&&+\left[ \mathbf{D}^{2}L_{k}(\theta ,\bar{p}_{k}(\theta ))-\mathbf{D}%
^{2}L(\theta ,p_{\theta })\right] (\tilde{p}_{k}(\theta )-p_{\theta },g),
\end{eqnarray*}%
where $\bar{p}_{k}(\theta )=\xi \tilde{p}_{k}(\theta )+(1-\xi )p_{\theta }$
for some $\xi \in (0,1)$; note that $\bar{p}_{k}(\theta )\in \mathcal{P}%
(t,\zeta ,D)$ by convexity. In the above display we have also made use of
the fact that $\mu (p_{\theta }^{-1}g)(\rho (\cdot ,\theta ))=0$ since $g\in 
\mathsf{H}_{t}^{0}$. Again adding and subtracting a term and using
Proposition~\ref{prop:DerivativesofthescaledlikelihoodNickl} this leads to%
\begin{eqnarray*}
\mathbf{D}L_{k}(\theta ,\tilde{p}_{k}(\theta ))(g) &=&\left( \mu _{k}-\mu
\right) (p_{\theta }^{-1}g)(\rho (\cdot ,\theta ))-\int_{\Omega }p_{\theta
}^{-1}(\tilde{p}_{k}(\theta )-p_{\theta })gd\lambda \\
&&+\left[ \mathbf{D}^{2}L_{k}(\theta ,\bar{p}_{k}(\theta ))-\mathbf{D}%
^{2}L(\theta ,\bar{p}_{k}(\theta ))\right] (\tilde{p}_{k}(\theta )-p_{\theta
},g) \\
&&+\int_{\Omega }\bar{p}_{k}^{-2}(\theta )p_{\theta }^{-1}(\bar{p}%
_{k}^{2}(\theta )-p_{\theta }^{2})(\tilde{p}_{k}(\theta )-p_{\theta
})gd\lambda .
\end{eqnarray*}

Consequently, for every real $j$ with $1/2<j<t$ we obtain%
\begin{eqnarray}
&&\sup_{\theta \in \Theta }\sup_{g\in \mathcal{G}}\left\vert \int_{\Omega
}p_{\theta }^{-1}(\tilde{p}_{k}(\theta )-p_{\theta })gd\lambda -\left( \mu
_{k}-\mu \right) (p_{\theta }^{-1}g)(\rho (\cdot ,\theta ))\right\vert 
\notag \\
&\leq &\sup_{\theta \in \Theta }\sup_{g\in \mathcal{G}}\left\vert \mathbf{D}%
L_{k}(\theta ,\tilde{p}_{k}(\theta ))(g)\right\vert +  \notag \\
&&\sup_{\theta \in \Theta }\sup_{g\in \mathcal{G}}\left\vert \left[ \mathbf{D%
}^{2}L_{k}(\theta ,\bar{p}_{k}(\theta ))-\mathbf{D}^{2}L(\theta ,\bar{p}%
_{k}(\theta ))\right] (\tilde{p}_{k}(\theta )-p_{\theta },g)\right\vert 
\notag \\
&&+\sup_{\theta \in \Theta }\sup_{g\in \mathcal{G}}\left\vert \int_{\Omega }%
\bar{p}_{k}^{-2}(\theta )p_{\theta }^{-1}(\bar{p}_{k}^{2}(\theta )-p_{\theta
}^{2})(\tilde{p}_{k}(\theta )-p_{\theta })gd\lambda \right\vert  \notag \\
&=&I+II+III,  \label{basic_0}
\end{eqnarray}%
where $I=o_{\mu }(k^{-(t-j)/(2t+1)-1/2})$ by Lemma \ref{derivativeorder}. We
next bound expressions $II$ and $III$:

Clearly, 
\begin{equation*}
II\leq \sup_{\theta \in \Theta }\Vert \tilde{p}_{k}(\theta )-p_{\theta
}\Vert _{j,2}\sup_{\Theta \times \mathcal{P}(t,\zeta ,D)}\left\Vert \mathbf{D%
}^{2}L(\theta ,p)-\mathbf{D}^{2}L_{k}(\theta ,p)\right\Vert _{\mathcal{U}%
_{j,1}\times \mathcal{G}}
\end{equation*}

The first supremum in the above display is $O_{\mu }(k^{-(t-j)/(2t+1)})$ by
Theorem~\ref{Uniform rate of AML-estimators}. Since $\mathcal{G}$ is bounded
in $\mathsf{W}_{2}^{t}(\Omega )$ and hence also in $\mathsf{W}%
_{2}^{j}(\Omega )$ as $j<t$ (cf.~Proposition~\ref%
{prop:SobolevembedsinHoelder}), and since $\mathcal{U}_{j,1}$ is clearly
bounded in $\mathsf{W}_{2}^{j}(\Omega )$, the second supremum in the above
display is $O_{\mu }(k^{-1/2})$ by Proposition~\ref{Closeness of the
derivatives of the scaled likelihood}, when applied with $\alpha =2$, $%
\mathcal{H}_{1}=\mathcal{U}_{j,1}$, and $\mathcal{H}_{2}=\mathcal{G}$. This
shows that the expression \emph{II} is $O_{\mu }(k^{-(t-j)/(2t+1)-1/2})$ for
every real $j$ with $1/2<j<t$.

Next, observe that $\left\vert \bar{p}_{k}(\theta )-p_{\theta }\right\vert
=\xi \left\vert \tilde{p}_{k}(\theta )-p_{\theta }\right\vert \leq
\left\vert \tilde{p}_{k}(\theta )-p_{\theta }\right\vert $ and that $\bar{p}%
_{k}(\theta )\geq \zeta $, $p_{\theta }\geq \zeta $ as these functions
belong to $\mathcal{P}(t,\zeta ,D)$. Hence%
\begin{equation*}
III\leq 2\zeta ^{-3}C_{t}^{2}DG\sup_{\theta \in \Theta }\left\Vert \tilde{p}%
_{k}(\theta )-p_{\theta }\right\Vert _{2}^{2},
\end{equation*}%
where $G<\infty $ is a $\left\Vert \cdot \right\Vert _{t,2}$-norm bound for $%
\mathcal{G}$. (Here we have repeatedly used Proposition~\ref%
{prop:SobolevembedsinHoelder}(b)). Theorem~\ref{Uniform rate of
AML-estimators} then shows that expression $III$ is $O_{\mu
}(k^{-2t/(2t+1)}) $. Putting things together we obtain that the l.h.s.~of (%
\ref{basic_0}) is $O_{\mu }^{\ast }(k^{-(t-j)/(2t+1)-1/2})$ for every real $%
j $ with $1/2<j<t$, and hence a fortiori for every real $j>1/2$.
Consequently,%
\begin{equation}
\sup_{\theta \in \Theta }\sup_{g\in \mathcal{G}}\sqrt{k}\left\vert
\int_{\Omega }p_{\theta }^{-1}(\tilde{p}_{k}(\theta )-p_{\theta })gd\lambda
-\left( \mu _{k}-\mu \right) (p_{\theta }^{-1}g)(\rho (\cdot ,\theta
))\right\vert =o_{\mu }^{\ast }(k^{-(t-j)/(2t+1)})  \label{basic}
\end{equation}%
for every real $j>1/2$.

Let now $\mathcal{F}$ be a nonempty bounded subset of $\QTR{up}{\mathsf{W}}%
_{2}^{t}(\Omega )$ and let $B<\infty $ denote a $\left\Vert \cdot
\right\Vert _{t,2}$-norm bound for $\mathcal{F}$. Define $\pi _{\theta
^{\prime }}(f)=(f-\int_{\Omega }fp_{\theta ^{\prime }}d\lambda )p_{\theta
^{\prime }}$ for any $f\in \QTR{up}{\mathsf{W}}_{2}^{t}(\Omega )$ and $%
\theta ^{\prime }\in \Theta $. Then, using Proposition\thinspace \ref%
{prop:SobolevembedsinHoelder}(a) and the fact that $p_{\theta ^{\prime }}\in 
\mathcal{P}(t,\zeta ,D)$ by Assumption \ref%
{mod:InclusionWithStrictInequalities}, gives%
\begin{eqnarray}
\sup_{\theta ^{\prime }\in \Theta }\sup_{f\in \mathcal{F}}\Vert \pi _{\theta
^{\prime }}(f)\Vert _{t,2} &\leq &M_{t}\sup_{\theta ^{\prime }\in \Theta
}\sup_{f\in \mathcal{F}}\left[ \left\Vert f-\int_{\Omega }fp_{\theta
^{\prime }}\,d\lambda \right\Vert _{t,2}\Vert p_{\theta ^{\prime }}\Vert
_{t,2}\right]  \notag \\
&\leq &M_{t}D\left[ B+\sup_{f\in \mathcal{F}}\left\Vert f\right\Vert
_{\Omega }\left\Vert 1\right\Vert _{t,2}\right]  \notag \\
&\leq &M_{t}DB(1+C_{t}\lambda (\Omega )^{1/2})<\infty .
\label{TheOperatorPiIsBounded}
\end{eqnarray}%
This shows that the set 
\begin{equation*}
\mathcal{G}(\Theta ,\mathcal{F})=\left\{ \pi _{\theta ^{\prime }}(f):f\in 
\mathcal{F},\theta ^{\prime }\in \Theta \right\}
\end{equation*}%
is a nonempty bounded subset of $\QTR{up}{\mathsf{W}}_{2}^{t}(\Omega )$. In
fact, it is a subset of $\mathsf{H}_{t}^{0}$ by definition of $\pi _{\theta
^{\prime }}$. It is now easy to see that applying (\ref{basic}) to $\mathcal{%
G}(\Theta ,\mathcal{F})$ implies (\ref{UCLT}) in the case $s=t$. The case $%
s>t$ immediately follows, since every nonempty bounded subset of $\QTR{up}{%
\mathsf{W}}_{2}^{s}(\Omega )\mathcal{\ }$with $s>t$ can also be viewed as a
nonempty bounded subset of $\QTR{up}{\mathsf{W}}_{2}^{t}(\Omega )$ by
Proposition \ref{prop:SobolevembedsinHoelder}(c). This proves Part (a) in
case $\zeta >0$ and $s\geq t$.

\emph{Step~2:} We now consider the case where $\zeta >0$ and $1/2<s<t$. For
every $f\in \mathcal{F}$ let $u_{k}(f)\in \mathsf{W}_{2}^{t}(\Omega )$ be
the approximators defined in the proof of Proposition~1 in Nickl (2007).
They have the following properties:%
\begin{equation}
\sup_{f\in \mathcal{F}}\Vert u_{k}(f)\Vert _{t,2}=O(k^{(t-s)/(2t+1)})\quad 
\text{as }k\rightarrow \infty ,  \label{Formel:uk(f)}
\end{equation}%
where $\sup_{f\in \mathcal{F}}\Vert u_{k}(f)\Vert _{t,2}$ is finite for
every $k\in \mathbb{N}$; and, for every $r$, $0\leq r<s$, 
\begin{equation}
\sup_{f\in \mathcal{F}}\Vert f-u_{k}(f)\Vert
_{r,2}=O(k^{-(s-r)/(2t+1)})\quad \text{as }k\rightarrow \infty .
\label{Formel:f-uk(f)}
\end{equation}

We have that%
\begin{eqnarray}
&&\sup_{\theta \in \Theta }\sup_{f\in \mathcal{F}}\left\vert \sqrt{k}%
\int_{\Omega }(\tilde{p}_{k}(\theta )-p_{\theta })fd\lambda -\sqrt{k}(\mu
_{k}-\mu )f(\rho (\cdot ,\theta ))\right\vert  \notag \\
&\leq &\sup_{\theta \in \Theta }\sup_{f\in \mathcal{F}}\left\vert \sqrt{k}%
\int_{\Omega }(\tilde{p}_{k}(\theta )-p_{\theta })(f-u_{k}(f))d\lambda
\right\vert  \notag \\
&&+\sup_{\theta \in \Theta }\sup_{f\in \mathcal{F}}\left\vert \sqrt{k}(\mu
_{k}-\mu )(f(\rho (\cdot ,\theta ))-u_{k}(f)(\rho (\cdot ,\theta
)))\right\vert  \notag \\
&&+\sup_{\theta \in \Theta }\sup_{f\in \mathcal{F}}\left\vert \sqrt{k}%
\int_{\Omega }(\tilde{p}_{k}(\theta )-p_{\theta })u_{k}(f)d\lambda -\sqrt{k}%
(\mu _{k}-\mu )u_{k}(f)(\rho (\cdot ,\theta ))\right\vert  \notag \\
&=&IV+V+VI.  \label{final}
\end{eqnarray}%
We now derive bounds for each of the above expressions:

Using (\ref{Formel:f-uk(f)}) with $r=0$, the Cauchy-Schwarz inequality, and
Theorem~\ref{Uniform rate of AML-estimators} we obtain%
\begin{equation*}
IV\leq \sqrt{k}\sup_{f\in \mathcal{F}}\Vert f-u_{k}(f)\Vert
_{2}\,\sup_{\theta \in \Theta }\Vert \tilde{p}_{k}(\theta )-p_{\theta }\Vert
_{2}=O_{\mu }(k^{-(s-1/2)/(2t+1)}).
\end{equation*}

Next, choose an arbitrary real $j$ such that $1/2<j<s$ and observe that%
\begin{eqnarray}
V &=&\sup_{\theta \in \Theta }\sup_{f\in \mathcal{F}}\left\vert \sqrt{k}(\mu
_{k}-\mu )(f-u_{k}(f))(\rho (\cdot ,\theta ))\right\vert  \notag \\
&\leq &\left( \sup_{\theta \in \Theta }\sup_{h\in \mathcal{U}%
_{j,1}}\left\vert \sqrt{k}(\mu _{k}-\mu )h(\rho (\cdot ,\theta ))\right\vert
\right) \sup_{f\in \mathcal{F}}\Vert f-u_{k}(f)\Vert _{j,2}  \notag \\
&=&\Vert \sqrt{k}(\mu _{k}-\mu )\Vert _{\mathcal{U}_{j,1}^{\ast }}\sup_{f\in 
\mathcal{F}}\Vert f-u_{k}(f)\Vert _{j,2},  \label{Formel: II in Part 2}
\end{eqnarray}%
where%
\begin{equation*}
\mathcal{U}_{j,1}^{\ast }=\left\{ h(\rho (\cdot ,\theta )):\theta \in \Theta
,\,h\in \mathcal{U}_{j,1}\right\} .
\end{equation*}%
Since $j>1/2$, the class of functions $\mathcal{U}_{j,1}^{\ast }$ is $\mu $%
-Donsker by Proposition~\ref{Bracketing entropy of F*}(a), hence 
\begin{equation*}
\left\Vert \sqrt{k}(\mu _{k}-\mu )\right\Vert _{\mathcal{U}_{j,1}^{\ast
}}=O_{\mu }(1)
\end{equation*}%
in view of Prohorov's theorem, measurability following from Proposition \ref%
{Borel_2}. Making use of (\ref{Formel:f-uk(f)}), it follows that the
r.h.s.~of (\ref{Formel: II in Part 2}), and hence Expression~\emph{V}, is $%
O_{\mu }(k^{-(s-j)/(2t+1)})$.

Finally note that Expression~\emph{VI} is bounded by%
\begin{equation*}
\sup_{\theta \in \Theta }\,\sup_{h\in \mathcal{U}_{t,1}}\left\vert \sqrt{k}%
\int_{\Omega }(\tilde{p}_{k}(\theta )-p_{\theta })hd\lambda -\sqrt{k}(\mu
_{k}-\mu )h(\rho (\cdot ,\theta ))\right\vert \sup_{f\in \mathcal{F}}\Vert
u_{k}(f)\Vert _{t,2}.
\end{equation*}%
Since $\mathcal{U}_{t,1}$ is a nonempty bounded subset of $\mathsf{W}%
_{2}^{t}(\Omega )$ and since Part (a) has already been established in \emph{%
Step 1 }for such sets of functions, the first term on the r.h.s.~of the last
display is $o_{\mu }(k^{-(t-j)/(2t+1)})$, and using (\ref{Formel:uk(f)}), we
conclude that 
\begin{equation*}
VI=o_{\mu }(k^{-(s-j)/(2t+1)}).
\end{equation*}

The above bounds imply that the l.h.s.~of (\ref{final}) is $O_{\mu
}(k^{-(s-j)/(2t+1)})$ for all $1/2<j<s$, and hence is $o_{\mu
}(k^{-(s-j)/(2t+1)})$ for all $j>1/2$. This completes the proof of Part (a)
of the theorem in case $\zeta >0$.

\emph{Step~3: }We next consider the case $\zeta =0$. In view of Assumption %
\ref{mod:InclusionWithStrictInequalities} we may choose $\chi >0$ such that $%
\inf_{\Omega \times \Theta }p(x,\theta )>\chi $. Then, by Remark~\ref{Remark
12}(iii), there are events that have probability tending to $1$ on which $%
\inf_{\theta \in \Theta }\,\inf_{x\in \Omega }\tilde{p}_{k}(\theta )(x)>\chi 
$ holds true. Since $\mathcal{P}(t,\chi ,D)\subseteq \mathcal{P}(t,0,D)=%
\mathcal{P}(t,\zeta ,D)$, we have that on these events $\tilde{p}_{k}(\theta
)$ coincides with the NPML-estimators over the smaller set $\mathcal{P}%
(t,\chi ,D)$. Part (a) in case $\zeta =0$ now follows from what has already
been established in the preceding two steps (applied to the NPML-estimator
based on $\mathcal{P}(t,\chi ,D)$ instead of $\mathcal{P}(t,\zeta ,D)$ and
noting that Assumption \ref{mod:InclusionWithStrictInequalities} is also
satisfied relative to $\mathcal{P}(t,\chi ,D)$).

Part (b): In view of Part (a) it is sufficient to show that $(\theta
,f)\mapsto \sqrt{k}(\mu _{k}-\mu )f(\rho (\cdot ,\theta ))$ converges weakly
in $\ell ^{\infty }(\Theta \times \mathcal{F})$ to $\mathbb{G}(\theta ,f)$.
To this end, let 
\begin{equation*}
H(\varphi )(\theta ,f)=\varphi (f(\rho (\cdot ,\theta )))
\end{equation*}%
for every $\varphi \in \ell ^{\infty }(\mathcal{F}^{\ast })$, $\theta \in
\Theta $, and $f\in \mathcal{F}$, where $\mathcal{F}^{\ast }=\left\{ f(\rho
(\cdot ,\theta )):\theta \in \Theta ,\,f\in \mathcal{F}\right\} $. Note that
the resulting mapping $H:\ell ^{\infty }(\mathcal{F}^{\ast })\rightarrow
\ell ^{\infty }(\Theta \times \mathcal{F})$ is continuous since $H$ is
linear and%
\begin{equation*}
\Vert H(\varphi )\Vert _{\Theta \times \mathcal{F}}=\sup_{\theta \in \Theta
}\,\sup_{f\in \mathcal{F}}\left\vert \varphi (f(\rho (\cdot ,\theta
)))\right\vert =\Vert \varphi \Vert _{\mathcal{F}^{\ast }}
\end{equation*}%
for all $\varphi \in \ell ^{\infty }(\mathcal{F}^{\ast })$. In fact, $H$ is
an isometry. Since $\mathcal{F}^{\ast }$ is $\mu $-Donsker by Proposition~%
\ref{Bracketing entropy of F*}(a), $\sqrt{k}(\mu _{k}-\mu )$ converges
weakly in $\ell ^{\infty }(\mathcal{F}^{\ast })$ to a $\mu $-Brownian bridge 
$\mathbb{G}^{\ast }$, that is, $\mathbb{G}^{\ast }$ is a mean-zero Gaussian
process indexed by $\mathcal{F}^{\ast }$, which is measurable as a mapping
with values in $\ell ^{\infty }(\mathcal{F}^{\ast })$, has covariance
function%
\begin{eqnarray*}
&&\limfunc{Cov}[\mathbb{G}^{\ast }(f(\rho (\,\cdot \,,\theta ))),\mathbb{G}%
^{\ast }(g(\rho (\,\cdot \,,\theta ^{\prime })))] \\
&=&\int_{V}\left( f(\rho (\cdot ,\theta ))-\int_{V}f(\rho (\cdot ,\theta
))d\mu \right) \left( g(\rho (\cdot ,\theta ^{\prime }))-\int_{V}g(\rho
(\cdot ,\theta ^{\prime }))d\mu \right) d\mu ,
\end{eqnarray*}%
and has sample paths that are uniformly continuous with respect to the
pseudo-metric%
\begin{equation*}
d^{\ast }(f(\rho (\cdot ,\theta )),g(\rho (\cdot ,\theta ^{\prime
})))=\left( \limfunc{Var}[\mathbb{G}^{\ast }(f(\rho (\,\cdot \,,\theta )))-%
\mathbb{G}^{\ast }(g(\rho (\,\cdot \,,\theta ^{\prime })))]\right) ^{1/2}.
\end{equation*}%
Since the empirical process $\sqrt{k}(\mu _{k}-\mu )$ indexed by $\mathcal{F}%
^{\ast }$ is mapped into the process $(\theta ,f)\mapsto \sqrt{k}(\mu
_{k}-\mu )f(\rho (\cdot ,\theta ))$ by the map $H$, the continuous mapping
theorem shows that the latter process converges weakly in $\ell ^{\infty
}(\Theta \times \mathcal{F})$ to $\mathbb{G}:=H(\mathbb{G}^{\ast })$. The
properties of $\mathbb{G}$ claimed in the theorem follow easily from the
corresponding properties of the $\mu $-Brownian bridge $\mathbb{G}^{\ast }$
and the fact that $H$ is an isometry.

Part (c): Follows directly from Part (b) in view of Prohorov's theorem, with
measurability again following from Proposition \ref{Borel_2}(b) in Appendix %
\ref{App D2}.
\end{proof}

\bigskip

We next obtain a corollary showing that $\sqrt{k}\int_{\Omega }(\tilde{p}%
_{k}(\theta )-p_{\theta })(\cdot )d\lambda $ converges in $\ell ^{\infty }(%
\mathcal{F})$ to $\mathbb{G}(\theta )$ \emph{uniformly} over $\Theta $,
where $\mathbb{G}(\theta )(f):=\mathbb{G}(\theta ,f)$ for all $f\in \mathcal{%
F}$. For this we recall the following definitions: Let $(S,d)$ be a metric
space. For probability spaces $(\Lambda _{1},\mathcal{A}_{1},P_{1})$, $%
(\Lambda _{2},\mathcal{A}_{2},P_{2})$ and mappings $Y_{1}:\Lambda
_{1}\rightarrow S$, $Y_{2}:\Lambda _{2}\rightarrow S$ such that $Y_{2}$ is $%
\mathcal{A}_{2}$-$\mathcal{B}(S,d)$-measurable and has separable range
define an analogue of the dual bounded Lipschitz metric by%
\begin{equation*}
\beta _{(S,d)}(Y_{1},Y_{2})=\sup \left\{ \left\vert \int_{\Lambda
_{1}}^{\ast }h(Y_{1})dP_{1}-\int_{\Lambda _{2}}h(Y_{2})dP_{2}\right\vert
:\Vert h\Vert _{BL(S,d)}\leq 1\right\} ,
\end{equation*}%
where $\int^{\ast }$ denotes the outer integral and $\Vert \cdot \Vert
_{BL(S,d)}$ denotes the bounded Lipschitz norm; cf.~the definition on
p.\thinspace 115 in Dudley (1999). By Theorem~3.6.4 in Dudley (1999), $%
Y_{n}\rightsquigarrow Y$ (where $Y$ is measurable and has separable range)\
if and only if 
\begin{equation*}
\lim_{n\rightarrow \infty }\beta _{(S,d)}(Y_{n},Y)=0.
\end{equation*}

\begin{cor}
\label{Corollary: 17} Let the hypotheses of Theorem~\ref{Theorem: Uniform
Donsker-type theorem} be satisfied. Then, for every $\theta \in \Theta $, $%
\mathbb{G}(\theta )=\mathbb{G}(\theta ,\cdot )$ is a measurable mapping with
values in $\ell ^{\infty }(\mathcal{F})$ that has separable range.
Furthermore, 
\begin{equation*}
\lim_{k\rightarrow \infty }\,\sup_{\theta \in \Theta }\beta _{\ell ^{\infty
}(\mathcal{F})}(\sqrt{k}\int_{\Omega }(\tilde{p}_{k}(\theta )-p_{\theta
})(\cdot )d\lambda ,\mathbb{G}(\theta )(\cdot ))=0.
\end{equation*}%
[In fact, $\mathbb{G}(\theta )$ is a $P_{\theta }$-Brownian bridge where $%
P_{\theta }$ denotes the probability measure corresponding to $p_{\theta }$.]
\end{cor}

\begin{proof}
Let $\theta \in \Theta $ be fixed, and define $H_{\theta }(\varphi
)(f)=\varphi (\theta ,f)$ for every $\varphi \in \ell ^{\infty }(\Theta
\times \mathcal{F})$ and $f\in \mathcal{F}$. This gives a Lipschitz mapping $%
H_{\theta }:\ell ^{\infty }(\Theta \times \mathcal{F})\rightarrow \ell
^{\infty }(\mathcal{F})$ whose Lipschitz constant is $1$ and hence is
independent of $\theta $. Clearly, $\mathbb{G}(\theta )=H_{\theta }(\mathbb{G%
})$ holds. Since $\mathbb{G}$ is a measurable mapping with separable range
in $\ell ^{\infty }(\Theta \times \mathcal{F})$ by Part~(b) of Theorem~\ref%
{Theorem: Uniform Donsker-type theorem}, this shows that, for every $\theta
\in \Theta $, $\mathbb{G}(\theta )$ is measurable with separable range in $%
\ell ^{\infty }(\mathcal{F})$. Further, since the composition of Lipschitz
mappings with Lipschitz constant at most $1$ is again Lipschitz with
Lipschitz constant at most $1$, it follows that%
\begin{eqnarray*}
&&\sup_{\theta \in \Theta }\beta _{\ell ^{\infty }(\mathcal{F})}(\sqrt{k}%
\int_{\Omega }(\tilde{p}_{k}(\theta )-p_{\theta })(\cdot )d\lambda ,\mathbb{G%
}(\theta )(\cdot )) \\
&=&\sup_{\theta \in \Theta }\beta _{\ell ^{\infty }(\mathcal{F})}(H_{\theta
}(\sqrt{k}\int_{\Omega }(\tilde{p}_{k}(\bullet )-p_{\bullet })(\cdot
)d\lambda ),H_{\theta }(\mathbb{G}(\bullet )(\cdot ))) \\
&\leq &\beta _{\ell ^{\infty }(\Theta \times \mathcal{F})}(\sqrt{k}%
\int_{\Omega }(\tilde{p}_{k}(\bullet )-p_{\bullet })(\cdot )d\lambda ,%
\mathbb{G}(\bullet )(\cdot )).
\end{eqnarray*}%
The r.h.s., and therefore the l.h.s., of the previous display converges to $%
0 $ by Part~(b) of Theorem~\ref{Theorem: Uniform Donsker-type theorem}. That 
$\mathbb{G}(\theta )$ is in fact a $P_{\theta }$-Brownian bridge indexed by $%
\mathcal{F}$ easily follows from Part~(b) of Theorem~\ref{Theorem: Uniform
Donsker-type theorem} and the transformation theorem.
\end{proof}

The statement in Corollary~\ref{Corollary: 17} is in fact independent of any
distance describing the concept of weak convergence in $\ell ^{\infty }(%
\mathcal{F})$, see Remark 18 in Gach and P\"{o}tscher (2010) for more
discussion.

\begin{remark}
\label{acanonical}\normalfont We have assumed that the processes $(X_{i})$
and $(V_{i})$ are canonically defined, i.e., are given by the respective
coordinate projections of the measurable space $(\Omega ^{\mathbb{N}}\times
V^{\mathbb{N}},\mathcal{B}(\Omega )^{\mathbb{N}}\otimes \mathcal{V}^{\mathbb{%
N}})$. We have made this assumption to be able to freely use results from
empirical process theory as well as from Nickl (2007) which typically are
formulated in this canonical setting. However, the measurability results in
Appendix \ref{App D2} show that all results of the paper continue to hold if 
$(X_{i})$ and $(V_{i})$ are defined on an arbitrary probability space.
\end{remark}

\section{Simulation-Based Minimum Distance Estimators\label{Section:
Indirect inference estimators}}

We next study simulation-based minimum distance (indirect inference)
estimators when the auxiliary density estimators are the NPML-estimators $%
\hat{p}_{n}$ and $\tilde{p}_{k}(\theta )$ based on the given auxiliary model 
$\mathcal{P}(t,\zeta ,D)$. To this end we define for every $\theta \in
\Theta $%
\begin{equation}
\mathbb{Q}_{n,k}(\theta )=%
\begin{cases}
\int_{\Omega }(\hat{p}_{n}-\tilde{p}_{k}(\theta ))^{2}\hat{p}%
_{n}^{-1}d\lambda & \text{if }\hat{p}_{n}(x)>0\text{ for all }x\in \Omega ,
\\ 
0 & \text{otherwise,}%
\end{cases}
\label{eq:DefinitionOfQ_n,k}
\end{equation}%
and%
\begin{equation*}
\mathbb{Q}_{n}(\theta )=%
\begin{cases}
\int_{\Omega }(\hat{p}_{n}-p_{\theta })^{2}\hat{p}_{n}^{-1}d\lambda & \text{%
if }\hat{p}_{n}(x)>0\text{ for all }x\in \Omega , \\ 
0 & \text{otherwise.}%
\end{cases}%
\end{equation*}%
Note that $\mathbb{Q}_{n,k}$ as well as $\mathbb{Q}_{n}$ take their values
in $[0,\infty ].$ By separability of $\Omega $ and continuity of $\hat{p}%
_{n} $, the set $\{\hat{p}_{n}(x)>0$ for all $x\in \Omega \}$ belongs to the 
$\sigma $-field $\mathcal{B}(\Omega )^{n}$. Since $\hat{p}_{n}$ and $\tilde{p%
}_{k}(\theta )$, respectively, are jointly measurable by Remark~\ref{Remark:
meas_pos}(i), it follows from Tonelli's theorem that $\mathbb{Q}%
_{n,k}(\theta )$ is $\mathcal{B}(\Omega )^{n}\otimes \mathcal{V}^{k}$%
-measurable and that $\mathbb{Q}_{n}(\theta )$ is $\mathcal{B}(\Omega )^{n}$%
-measurable for every $\theta \in \Theta $. [Assigning the value $0$ on the
complement of $\left\{ \hat{p}_{n}(x)>0\text{ for all }x\in \Omega \right\} $
to both objective functions is arbitrary and irrelevant for the asymptotic
considerations to follow.]

A simulation-based minimum distance (SMD) estimator is now a mapping $\hat{%
\theta}_{n,k}:\Omega ^{n}\times V^{k}\rightarrow \Theta $ that minimizes $%
\mathbb{Q}_{n,k}$ over $\Theta $ whenever the minimum exists (and is defined
arbitrarily otherwise). Similarly, a minimum distance (MD)\ estimator is a
mapping $\hat{\theta}_{n}:\Omega ^{n}\rightarrow \Theta $ that minimizes $%
\mathbb{Q}_{n}$ over $\Theta $ whenever the minimum exists (and is defined
arbitrarily otherwise). The MD-estimator is of course only feasible if a
closed form expression for $p_{\theta }$ can be found; here it serves as an
auxiliary device for proving asymptotic results for the SMD-estimator.

Furthermore, whenever Assumption \ref{dens:StrictInequality} is satisfied,
we define 
\begin{equation*}
Q(\theta )=\int_{\Omega }(p_{\blacktriangle }-p_{\theta
})^{2}p_{\blacktriangle }^{-1}d\lambda ,
\end{equation*}%
which takes its values in $[0,\infty ]$. In view of convergence of $\hat{p}%
_{n}$ to $p_{\blacktriangle }$ and of $\tilde{p}_{k}(\theta )$ to $p_{\theta
}$ (under the assumptions of Theorem \ref{Theorem: consistency of
AML-estimators}), $Q$ can be viewed as the limiting counterpart of both $%
\mathbb{Q}_{n,k}$ as well as $\mathbb{Q}_{n}$.

\subsection{Consistency of SMD-Estimators}

Before turning to consistency, we show that MD- and SMD-estimators in fact
minimize their corresponding objective function at least on events that have
probability tending to $1$. Note that in the following proposition the
statement of Part (c) is stronger than the one of Part (b), but also
requires additional assumptions.

\begin{prop}
\label{Proposition: Existence of simulation-based II-estimators} Let
Assumption~\ref{rho: Continuity of rho} be satisfied.

(a) Suppose $\zeta >0$ holds. Then any SMD-estimator $\hat{\theta}_{n,k}$
minimizes $\mathbb{Q}_{n,k}$ for every \linebreak $(x_{1},\ldots
,x_{n},v_{1},\ldots ,v_{k})\in \Omega ^{n}\times V^{k}$. Furthermore, there
exists an SMD-estimator that is $\mathcal{B}(\Omega )^{n}\otimes \mathcal{V}%
^{k}$-$\mathcal{B}(\Theta )$-measurable.

(b) Suppose $\zeta =0$ and Assumptions \ref{dens:Element} and \ref%
{dens:StrictInequality} hold. Then there are events $A_{n}\in \mathcal{B}%
(\Omega )^{n}$ having probability converging to $1$ as $n\rightarrow \infty $
such that, on the events $A_{n}\times V^{k}$ and for every $k\in \mathbb{N}$%
, any SMD-estimator $\hat{\theta}_{n,k}$ minimizes $\mathbb{Q}_{n,k}$.

(c) Suppose $\zeta =0$ and Assumptions \ref{dens:Element}, \ref%
{dens:StrictInequality}, \ref{mod:Inclusion}, and \ref{mod: Strict
inequality} hold. Then, for every constant $\chi >0$ satisfying $\inf_{x\in
\Omega }p_{\blacktriangle }(x)>\chi $ and $\inf_{\Omega \times \Theta
}p(x,\theta )>\chi $, there are events $C_{n,k}\in \mathcal{B}(\Omega
)^{n}\otimes \mathcal{V}^{k}$ that have probability tending to $1$ as $\min
(n,k)\rightarrow \infty $ such that on $C_{n,k}$ any SMD-estimator $\hat{%
\theta}_{n,k}$ coincides with an SMD-estimator that is obtained from using $%
\mathcal{P}(t,\chi ,D)$ instead of $\mathcal{P}(t,\zeta ,D)$ as the
underlying auxiliary model.
\end{prop}

\begin{proof}
(a) By Proposition~\ref{Proposition: properties of Q_n and Q_n,k}(b) in
Appendix \ref{App G}, $\mathbb{Q}_{n,k}$ is continuous and real-valued on
the compact set $\Theta $ for each $(x_{1},\ldots ,x_{n},v_{1},\ldots
,v_{k})\in \Omega ^{n}\times V^{k}$ implying that any $\hat{\theta}_{n,k}$
is a minimizer for each $(x_{1},\ldots ,x_{n},v_{1},\ldots ,v_{k})$. Since $%
\mathbb{Q}_{n,k}$ is also a measurable function in $(x_{1},\ldots
,x_{n},v_{1},\ldots ,v_{k})$ for each fixed $\theta \in \Theta $, as shown
earlier, the existence of a measurable selection follows from Lemma~A3 in P%
\"{o}tscher and Prucha (1997).

(b) By Remark~\ref{Remark 12}(i) there are events $A_{n}\in \mathcal{B}%
(\Omega )^{n}$ that have probability tending to $1$ as $n\rightarrow \infty $
on which $\inf_{x\in \Omega }\hat{p}_{n}(x)>2^{-1}\inf_{x\in \Omega
}p_{\blacktriangle }(x)>0$. From Proposition~\ref{Proposition: properties of
Q_n and Q_n,k}(b) it follows that $\mathbb{Q}_{n,k}$ is continuous and
real-valued on $\Theta $ for each $(x_{1},\ldots ,x_{n},v_{1},\ldots
,v_{k})\in A_{n}\times V^{k}$. Compactness of $\Theta $ completes the proof.

(c) Let $\chi $ be as in the proposition. Set $C_{n,k}=A_{n}\times B_{k}$,
where $A_{n}$ and $B_{k}$ are as in Remarks~\ref{Remark 12}(i) and (iii),
and observe that $C_{n,k}$ has probability tending to $1$ as $\min
(n,k)\rightarrow \infty $. By Remark~\ref{Remark 12}, we have on $C_{n,k}$
that $\inf_{x\in \Omega }\hat{p}_{n}(x)>\chi $ and $\inf_{\Omega \times
\Theta }\tilde{p}_{k}(\theta )(x)>\chi $. Since $\mathcal{P}(t,\chi
,D)\subseteq \mathcal{P}(t,\zeta ,D)$, it follows that on $C_{n,k}$ the
NPML-estimators $\hat{p}_{n}$ and $\tilde{p}_{k}(\theta )$, respectively,
coincide with the corresponding NPML-estimators based on the auxiliary model 
$\mathcal{P}(t,\chi ,D)$ instead of $\mathcal{P}(t,\zeta ,D)$. Therefore, on 
$C_{n,k}$, the objective function $\mathbb{Q}_{n,k}$ coincides with the
corresponding objective function based on the auxiliary model $\mathcal{P}%
(t,\chi ,D)$, and thus $\hat{\theta}_{n,k}$ coincides with the corresponding
SMD-estimator based on the auxiliary model $\mathcal{P}(t,\chi ,D)$.
\end{proof}

\bigskip

The proofs of Parts (a) and (b) of the subsequent proposition are analogous
to the proofs of Proposition \ref{Proposition: Existence of simulation-based
II-estimators} above. Part (c) follows immediately from compactness of $%
\Theta $ and Lemma~\ref{Lemma: Generic continuity of Q} in Appendix \ref{App
G}.

\begin{prop}
\label{Existence of II-estimators} Suppose $\mathcal{P}_{\Theta }\subseteq 
\mathcal{L}^{2}(\Omega )$ and $\theta \mapsto p_{\theta }$ is a continuous
map from $\Theta $ into $(\mathcal{L}^{2}(\Omega ),\Vert \cdot \Vert _{2})$.

(a) Suppose $\zeta >0$ holds. Then any MD-estimator $\hat{\theta}_{n}$
minimizes $\mathbb{Q}_{n}$ for every $(x_{1},\ldots ,x_{n})\in \Omega ^{n}$.
Furthermore, there exists an MD-estimator $\hat{\theta}_{n}$ that is $%
\mathcal{B}(\Omega )^{n}$-$\mathcal{B}(\Theta )$-measurable.

(b) Suppose $\zeta =0$ and Assumptions \ref{dens:Element} and \ref%
{dens:StrictInequality} hold. Then there are events $A_{n}\in \mathcal{B}%
(\Omega )^{n}$ that have probability tending to $1$ as $n\rightarrow \infty $
such that, on these events, any MD-estimator $\hat{\theta}_{n}$ minimizes $%
\mathbb{Q}_{n}$. [In fact, more is true: If $\chi >0$ satisfies $\inf_{x\in
\Omega }p_{\blacktriangle }(x)>\chi $, then, on $A_{n}$, any MD-estimator $%
\hat{\theta}_{n}$ coincides with an MD-estimator that is obtained by using $%
\mathcal{P}(t,\chi ,D)$ instead of $\mathcal{P}(t,\zeta ,D)$ as the
underlying auxiliary model.]

(c) Suppose Assumption \ref{dens:StrictInequality} is satisfied. Then $Q$
attains its minimum on $\Theta $.
\end{prop}

\begin{remark}
\label{Remark: Conditions for L2-continuity}\normalfont Assumption~\ref{mod:
pointwise continuity} together with a uniform integrability condition on $%
\left\{ p_{\theta }^{2}:\theta \in \Theta \right\} $ clearly implies that $%
\mathcal{P}_{\Theta }\subseteq \mathcal{L}^{2}(\Omega )$ and that $\theta
\mapsto p_{\theta }$ is a continuous mapping from $\Theta $ into $(\mathcal{L%
}^{2}(\Omega ),\Vert \cdot \Vert _{2})$. In particular, Assumptions~\ref%
{mod:Inclusion} and \ref{mod: pointwise continuity} together are sufficient.
\end{remark}

\begin{prop}
\label{Proposition: Consistency of II-estimators} (a) Let Assumptions \ref%
{dens:Element}, \ref{dens:StrictInequality}, \ref{mod:Inclusion}, \ref{mod:
Strict inequality}, and \ref{rho: Continuity of rho} be satisfied. If $Q$
has a unique minimizer $\theta _{0}^{\ast }$ over $\Theta $, then any
SMD-estimator $\hat{\theta}_{n,k}$ converges to $\theta _{0}^{\ast }$ in
outer probability as $\min (n,k)\rightarrow \infty $.

(b) Suppose $\mathcal{P}_{\Theta }\subseteq \mathcal{L}^{2}(\Omega )$ and $%
\theta \mapsto p_{\theta }$ is a continuous map from $\Theta $ into $(%
\mathcal{L}^{2}(\Omega ),\Vert \cdot \Vert _{2})$. Let Assumptions \ref%
{dens:Element} and \ref{dens:StrictInequality} be satisfied. If $Q$ has a
unique minimizer $\theta _{0}^{\ast }$ over $\Theta $, then any MD-estimator 
$\hat{\theta}_{n}$ converges to $\theta _{0}^{\ast }$ in outer probability
as $n\rightarrow \infty $.
\end{prop}

\begin{proof}
(a) Note that $Q$ is continuous (by Remark \ref{Remark: Conditions for
L2-continuity}, Proposition \ref{Interrelation} in Appendix \ref{App A}, and
Proposition~\ref{Proposition: properties of Q_n and Q_n,k}(c) in Appendix %
\ref{App G}), and that $Q(\theta )>Q(\theta _{0}^{\ast })$ for any $\theta
\neq \theta _{0}^{\ast }$ by assumption. Furthermore, $\mathbb{Q}%
_{n,k}(\theta )$ converges to $Q(\theta )$ uniformly over the compact set $%
\Theta $ in outer probability as $\min (n,k)\rightarrow \infty $ by
Proposition~\ref{Proposition: Uniform convergence properties of II-objective
functions}(b) in Appendix \ref{App G}. A standard argument together with
Proposition~\ref{Proposition: Existence of simulation-based II-estimators}
gives the result. For more details see Gach and P\"{o}tscher (2010).

(b) Analogous.
\end{proof}

\begin{remark}
\normalfont(i) It follows from Proposition~\ref{Interrelation} in Appendix %
\ref{App A} together with Remark~\ref{Remark: Conditions for L2-continuity}
that the assumptions of Proposition \ref{Existence of II-estimators}(c) are
satisfied under the assumptions of Part (a) of the above proposition (and
they are trivially satisfied under the assumptions of Part (b)).
Consequently, under the assumptions of the above proposition, $Q$ always has
a minimizer over $\Theta $. Hence, the assumption in the above proposition
that $Q$ has a unique minimizer is in fact only a uniqueness assumption.

(ii) We do not strive for utmost generality in the consistency result for
MD-estimators; possible relaxations lie in weakening the assumptions that $%
\mathcal{P}_{\Theta }\subseteq \mathcal{L}^{2}(\Omega )$ and that $\theta
_{0}^{\ast }$ is unique.
\end{remark}

\subsection{Asymptotic Normality of SMD-Estimators}

We next show that SMD- and MD-estimators are asymptotically normally
distributed, with their asymptotic variance-covariance matrix coinciding
with the inverse of the Fisher-information matrix in case the parametric
model $\mathcal{P}_{\Theta }$ is correctly specified. We first prove the
result for MD-estimators and then show how this can be carried over to
SMD-estimators. To this end we introduce a further assumption which is
standard in maximum likelihood theory.

\begin{modAsm}
\label{mod: domination conditions} The interior $\Theta ^{\circ }$ of $%
\Theta \subseteq \mathbb{R}^{m}$ is non-empty. For every $x\in \Omega $ the
function $\theta \mapsto p(x,\theta )$ is twice continuously partially
differentiable on $\Theta ^{\circ }$, and the following domination
conditions hold for all $i,j=1,\ldots ,m$: 
\begin{equation*}
\int_{\Omega }\sup_{\theta \in \Theta ^{\circ }}\left\vert \frac{\partial p}{%
\partial \theta _{i}}(x,\theta )\right\vert ^{2}d\lambda (x)<\infty ,\quad
\int_{\Omega }\sup_{\theta \in \Theta ^{\circ }}\left\vert \frac{\partial
^{2}p}{\partial \theta _{i}\partial \theta _{j}}(x,\theta )\right\vert
d\lambda (x)<\infty .
\end{equation*}
\end{modAsm}

We note that under the assumptions of the subsequent theorem, as well as
under the assumptions of Theorem~\ref{Theorem: asymptotic normality of the
simulation-based II-estimator}, the function $Q$ always possesses a
minimizer (cf.~Proposition \ref{Existence of II-estimators}(c) and Remark %
\ref{Remark: Conditions for L2-continuity}, as well as Proposition \ref%
{Interrelation} in Appendix \ref{App A} in case of Theorem~\ref{Theorem:
asymptotic normality of the simulation-based II-estimator}); furthermore,
the Hessian matrix of $Q(\theta )$ exists for every $\theta \in \Theta
^{\circ }$, cf.~Lemma~\ref{DerivativesOfQ_nAndQ} in Appendix \ref{App G}
which provides an explicit formula. We shall write $J(\theta )$ for $1/2$
times the Hessian matrix of $Q(\theta )$.

\begin{thm}
\label{Theorem: asymptotic normality of the II-estimator} Let Assumptions %
\ref{dens:InternalPoint}, \ref{mod:Inclusion}, \ref{mod: Strict inequality}, %
\ref{mod: pointwise continuity}, \ref{mod: domination conditions} be
satisfied. Suppose that the minimizer $\theta _{0}^{\ast }$ of $Q$ over $%
\Theta $ is unique and belongs to $\Theta ^{\circ }$, and suppose that the
matrix $J(\theta _{0}^{\ast })$ is positive definite. Furthermore, assume
that the first-order partial derivatives $\frac{\partial p}{\partial \theta
_{i}}(\cdot ,\theta _{0}^{\ast })$ belong to $\mathsf{W}_{2}^{s}(\Omega )$
for some $s>1/2$ and for all $i=1,\ldots ,m$. Then 
\begin{equation*}
\sqrt{n}(\hat{\theta}_{n}-\theta _{0}^{\ast })\rightsquigarrow N(0,J(\theta
_{0}^{\ast })^{-1}I(\theta _{0}^{\ast })J(\theta _{0}^{\ast })^{-1})\quad 
\text{as $n\rightarrow \infty $},
\end{equation*}%
where $I(\theta _{0}^{\ast })$ is given by%
\begin{equation*}
\int_{\Omega }\frac{\partial p}{\partial \theta }(\cdot ,\theta _{0}^{\ast })%
\frac{\partial p}{\partial \theta ^{\prime }}(\cdot ,\theta _{0}^{\ast
})p_{\theta _{0}^{\ast }}^{2}p_{\blacktriangle }^{-3}d\lambda -\int_{\Omega }%
\frac{\partial p}{\partial \theta }(\cdot ,\theta _{0}^{\ast })p_{\theta
_{0}^{\ast }}p_{\blacktriangle }^{-1}d\lambda \int_{\Omega }\frac{\partial p%
}{\partial \theta ^{\prime }}(\cdot ,\theta _{0}^{\ast })p_{\theta
_{0}^{\ast }}p_{\blacktriangle }^{-1}d\lambda ,
\end{equation*}%
which is well-defined and nonnegative definite. If, additionally, $\mathcal{P%
}_{\Theta }$ is correctly specified in the sense that $p_{\blacktriangle
}=p_{\theta _{0}}$ a.e.~for some $\theta _{0}\in \Theta $, then $\theta
_{0}^{\ast }=\theta _{0}$ and $I(\theta _{0})=J(\theta _{0})$ hold, and $%
I(\theta _{0})$ coincides with the Fisher-information matrix.
\end{thm}

\begin{proof}
\textbf{Step 1:} Assume first that $\zeta >0$. By Proposition~\ref%
{Proposition: Consistency of II-estimators}(b), $\hat{\theta}_{n}$ belongs
to a sufficiently small open ball, centered at $\theta _{0}^{\ast }$ and
contained in $\Theta ^{\circ }$, on subsets $E_{n}$ of the sample space that
have inner probability tending to $1$ as $n\rightarrow \infty $.
Consequently,%
\begin{equation*}
\frac{\partial \mathbb{Q}_{n}}{\partial \theta }(\hat{\theta}_{n})=0
\end{equation*}%
holds on $E_{n}$. Applying the mean-value theorem to each component of $%
\partial \mathbb{Q}_{n}/\partial \theta $ then yields on $E_{n}$%
\begin{equation}
\sqrt{n}\frac{\partial \mathbb{Q}_{n}}{\partial \theta }(\theta _{0}^{\ast
})+J(\theta _{0}^{\ast })\sqrt{n}(\hat{\theta}_{n}-\theta _{0}^{\ast
})+(H_{n}-J(\theta _{0}^{\ast }))\sqrt{n}(\hat{\theta}_{n}-\theta _{0}^{\ast
})=0,  \label{expansion}
\end{equation}%
where $H_{n}$ is the Hessian matrix of $\mathbb{Q}_{n}$ with $i$-th row
evaluated at some mean value $\bar{\theta}_{n,i}$ on the line segment that
joins $\theta _{0}^{\ast }$ and $\hat{\theta}_{n}$. Observe that $H_{n}$
converges to the invertible matrix $J(\theta _{0}^{\ast })$ in outer
probability by Proposition~\ref{Proposition: Consistency of II-estimators},
Proposition~\ref{Lemma: convergence of D2Q_n} in Appendix \ref{App G}, and
continuity of $J(\theta )$ on $\Theta ^{\circ }$ (cf.~Lemma~\ref%
{DerivativesOfQ_nAndQ} in Appendix \ref{App G}). We next show that the score
evaluated at $\theta _{0}^{\ast }$ satisfies a central limit theorem. To
this end let $v\in \mathbb{R}^{m}$ be arbitrary, and use Lemma~\ref%
{DerivativesOfQ_nAndQ}(a) to obtain 
\begin{eqnarray*}
v^{\prime }\sqrt{n}\frac{\partial \mathbb{Q}_{n}}{\partial \theta }(\theta
_{0}^{\ast }) &=&2\sqrt{n}\int_{\Omega }(\hat{p}_{n}-p_{\blacktriangle
})^{2}v^{\prime }\frac{\partial p}{\partial \theta }(\cdot ,\theta
_{0}^{\ast })\frac{p_{\theta _{0}^{\ast }}}{\hat{p}_{n}p_{\blacktriangle
}^{2}}d\lambda \\
&&\qquad -2\sqrt{n}\int_{\Omega }(\hat{p}_{n}-p_{\blacktriangle })v^{\prime }%
\frac{\partial p}{\partial \theta }(\cdot ,\theta _{0}^{\ast })\frac{%
p_{\theta _{0}^{\ast }}}{p_{\blacktriangle }^{2}}d\lambda \\
&&\qquad -2\sqrt{n}\int_{\Omega }(p_{\blacktriangle }-p_{\theta _{0}^{\ast
}})v^{\prime }\frac{\partial p}{\partial \theta }(\cdot ,\theta _{0}^{\ast })%
\frac{1}{p_{\blacktriangle }}d\lambda \\
&=&\text{I}+\text{II}+\text{III}.
\end{eqnarray*}%
Observe that Expression~\emph{III} equals $\sqrt{n}v^{\prime }(\partial
Q/\partial \theta )(\theta _{0}^{\ast })$ by Lemma~\ref{DerivativesOfQ_nAndQ}%
(b) in Appendix \ref{App G}. Since $\theta _{0}^{\ast }$ is an interior
minimizer of $Q$ by assumption, Expression~\emph{III} is $0$.

\textit{Convergence of} \textit{\emph{I:}} By assumption $v^{\prime }\frac{%
\partial p}{\partial \theta }(\cdot ,\theta _{0}^{\ast })$ belongs to $%
\mathsf{W}_{2}^{s}(\Omega )$ with $s>1/2$ and is thus sup-norm bounded by $%
C_{s}\Vert v^{\prime }\frac{\partial p}{\partial \theta }(\cdot ,\theta
_{0}^{\ast })\Vert _{s,2}<\infty $. Clearly, $\left\Vert p_{\theta
_{0}^{\ast }}\hat{p}_{n}^{-1}p_{\blacktriangle }^{-2}\right\Vert _{\Omega
}\leq \zeta ^{-3}C_{t}D$ holds in view of Assumption \ref{mod:Inclusion}.
Hence, 
\begin{equation*}
\text{I}\leq 2C_{s}\left\Vert v^{\prime }\frac{\partial p}{\partial \theta }%
(\cdot ,\theta _{0}^{\ast })\right\Vert _{s,2}\zeta ^{-3}C_{t}D\sqrt{n}\Vert 
\hat{p}_{n}-p_{\blacktriangle }\Vert _{2}^{2}.
\end{equation*}%
Consequently, Expression~\emph{I} converges to $0$ in outer probability by
Proposition~\ref{pointwise_rate}(a) applied with $s=0$.

\textit{Convergence of} \textit{\emph{II:}} Set $r=\min (s,t)>1/2$. Observe
that $-2v^{\prime }(\partial p/\partial \theta )(\cdot ,\theta _{0}^{\ast
})\in \mathsf{W}_{2}^{r}(\Omega )$ by assumption, that $p_{\theta _{0}^{\ast
}}\in \mathsf{W}_{2}^{r}(\Omega )$ by Assumption~\ref{mod:Inclusion}, and
that $p_{\blacktriangle }\in \mathsf{W}_{2}^{r}(\Omega )$ by Assumption~\ref%
{dens:Element}. Since $\zeta >0$ has been assumed, it follows that 
\begin{equation*}
f:=-2v^{\prime }\frac{\partial p}{\partial \theta }(\cdot ,\theta _{0}^{\ast
})\frac{p_{\theta _{0}^{\ast }}}{p_{\blacktriangle }^{2}}
\end{equation*}%
belongs to $\mathsf{W}_{2}^{r}(\Omega )$ in view of Proposition~\ref%
{prop:SobolevembedsinHoelder}(a),(d). Applying Theorem \ref{Theorem: Nickl}%
(a) with $\mathcal{F}=\{f\}$ we obtain that \textit{\emph{II }}converges in
distribution to a centered normal distribution with variance $4v^{\prime
}I(\theta _{0}^{\ast })v$. By the Cram\'{e}r-Wold device, $\sqrt{n}(\partial 
\mathbb{Q}_{n}/\partial \theta )(\theta _{0}^{\ast })$ asymptotically
follows a centered normal distribution with variance-covariance matrix $%
4I(\theta _{0}^{\ast })$. Nonnegative definiteness of $I(\theta _{0}^{\ast
}) $ is now an immediate consequence and the asymptotic distribution of $%
\sqrt{n}(\hat{\theta}_{n}-\theta _{0}^{\ast })$ follows easily from (\ref%
{expansion}). The claims under correct specification of the model $\mathcal{P%
}_{\Theta } $ follow easily from Lemma~\ref{DerivativesOfQ_nAndQ}(b) in
Appendix \ref{App G}.

\textbf{Step~2:} Now assume that $\zeta =0$. Note that $\inf_{x\in \Omega
}p_{\blacktriangle }(x)>0$ and $\inf_{\Omega \times \Theta }p(x,\theta )>0$
because of Assumptions~\ref{dens:InternalPoint} and \ref{mod: Strict
inequality}. Let $\chi >0$ be such that $\inf_{x\in \Omega
}p_{\blacktriangle }(x)>\chi $ and $\inf_{\Omega \times \Theta }p(x,\theta
)>\chi $. Then it follows from Proposition~\ref{Existence of II-estimators}%
(b) that there are events that have probability tending to $1$ such that on
these events $\hat{\theta}_{n}$ coincides with an MD-estimator $\check{\theta%
}_{n}$ that is based on $\mathcal{P}(t,\chi ,D)$ instead of $\mathcal{P}%
(t,\zeta ,D)$. Since the assumptions of the theorem are also satisfied with $%
\mathcal{P}(t,\chi ,D)$ instead of $\mathcal{P}(t,\zeta ,D)$, applying to $%
\check{\theta}_{n}$ what has already been established in Step~1 completes
the proof.
\end{proof}

\bigskip

The following lemma will be instrumental in proving the asymptotic normality
result for SMD-estimators.

\begin{lem}
\label{Square root lemma} Let $U\subseteq \mathbb{R}^{m}$ be a (non-empty)
open, convex set. Let $f:U\rightarrow \mathbb{R}$ and $g:U\rightarrow 
\mathbb{R}$ be functions such that $g$ is twice partially differentiable on $%
U$ with Hessian satisfying 
\begin{equation}
\inf_{x\in U}y^{\prime }\frac{\partial ^{2}g}{\partial x\partial x^{\prime }}%
(x)y\geq K\Vert y\Vert ^{2}  \label{Positive definiteness of D2g}
\end{equation}%
for all $y\in \mathbb{R}^{m}$ and some $0<K<\infty $. If $u$ is a minimizer
of $f$ over $U$ and $v$ is a minimizer of $g$ over $U$, then 
\begin{equation*}
\Vert u-v\Vert \leq 2K^{-1/2}\sqrt{\Vert f-g\Vert _{U}}.
\end{equation*}
\end{lem}

\begin{proof}
Suppose that minimizers $u$ and $v$ exist, since otherwise there is nothing
to prove. As $v$ is a minimizer of the twice partially differentiable
function $g$ on the convex open set $U$, we have (by a pathwise Taylor
series expansion) that 
\begin{equation*}
g(u)=g(v)+\frac{1}{2}(u-v)^{\prime }\frac{\partial ^{2}g}{\partial x\partial
x^{\prime }}(\bar{v})(u-v),
\end{equation*}%
where $\bar{v}$ lies in the convex hull of $\{u,v\}\subseteq U$. By (\ref%
{Positive definiteness of D2g}) we obtain%
\begin{equation}
\Vert u-v\Vert \leq \sqrt{2}K^{-1/2}\sqrt{|g(u)-g(v)|}.
\label{Closeness of the u-values}
\end{equation}%
Next, note the inequality%
\begin{equation*}
f(u)-g(u)\leq f(u)-g(v)\leq f(v)-g(v)
\end{equation*}%
which implies%
\begin{equation*}
|f(u)-g(v)|\leq \Vert f-g\Vert _{U},
\end{equation*}%
which in turn yields 
\begin{equation*}
|g(u)-g(v)|\leq |g(u)-f(u)|+|f(u)-g(v)|\leq 2\Vert f-g\Vert _{U}.
\end{equation*}%
Plugged into (\ref{Closeness of the u-values}) this proves the result.
\end{proof}

\bigskip

The asymptotic normality result for SMD-estimators is now as follows.

\begin{thm}
\label{Theorem: asymptotic normality of the simulation-based II-estimator}
Let Assumptions~\ref{dens:InternalPoint}, \ref{mod:Inclusion}, \ref{mod:
domination conditions}, \ref{rho: Hoelderity of rho} be satisfied. Suppose
that the minimizer $\theta _{0}^{\ast }$ of $Q$ over $\Theta $ is unique and
belongs to $\Theta ^{\circ }$, suppose that the matrix $J(\theta _{0}^{\ast
})$ is positive definite, and assume that the first-order partial
derivatives $\frac{\partial p}{\partial \theta _{i}}(\cdot ,\theta
_{0}^{\ast })$ belong to $\mathsf{W}_{2}^{s}(\Omega )$ for some $s>1/2$ and
for all $i=1,\ldots ,m$. Suppose further that either\emph{\ }(i) Assumption~%
\ref{mod: Strict inequality} is satisfied and $k(n)$ satisfies $%
k(n)/n^{2+1/t}\rightarrow \infty $ as $n\rightarrow \infty $; or (ii)
Assumption~\ref{mod:InclusionWithStrictInequalities} is satisfied and $k(n)$
satisfies $k(n)/n^{2}\rightarrow \infty $ as $n\rightarrow \infty $. Then 
\begin{equation*}
\sqrt{n}(\hat{\theta}_{n,k(n)}-\theta _{0}^{\ast })\rightsquigarrow
N(0,J(\theta _{0}^{\ast })^{-1}I(\theta _{0}^{\ast })J(\theta _{0}^{\ast
})^{-1})\quad \text{as $n\rightarrow \infty $},
\end{equation*}%
where $I(\theta _{0}^{\ast })$ is given as in Theorem \ref{Theorem:
asymptotic normality of the II-estimator}, is\ well-defined, and is
nonnegative definite. If, additionally, $\mathcal{P}_{\Theta }$ is correctly
specified in the sense that $p_{\blacktriangle }=p_{\theta _{0}}$ a.e.~for
some $\theta _{0}\in \Theta $, then $\theta _{0}^{\ast }=\theta _{0}$ and $%
I(\theta _{0})=J(\theta _{0})$ hold, and $I(\theta _{0})$ coincides with the
Fisher-information matrix.
\end{thm}

\begin{proof}
\textbf{Step 1:} Assume that $\zeta >0$. Observe first that the assumptions
of the current theorem imply the assumptions of Theorem~\ref{Theorem:
asymptotic normality of the II-estimator}, noting that Assumption~\ref{mod:
pointwise continuity} follows from Assumptions \ref{mod:Inclusion} and \ref%
{rho: Hoelderity of rho} in view of Proposition~\ref{Interrelation} in
Appendix \ref{App A}. It hence suffices to prove that 
\begin{equation}
\sqrt{n}(\hat{\theta}_{n,k(n)}-\hat{\theta}_{n})=o_{\func{Pr}}^{\ast
}(1)\quad \text{as $n\rightarrow \infty $.}  \label{diff}
\end{equation}%
We achieve this by applying Lemma~\ref{Square root lemma} to the objective
functions $\mathbb{Q}_{n,k}$ and $\mathbb{Q}_{n}$: Let $U$ be a sufficiently
small open, convex neighbourhood of $\theta _{0}^{\ast }$ that is contained
in $\Theta ^{\circ }$ such that the smallest eigenvalue of $J(\theta )$ is
bounded from below by a positive constant for all $\theta \in U$, the
constant not depending on $\theta $. Such a set $U$ exists, since $J(\theta
_{0}^{\ast })$ is positive definite by assumption and $J(\theta )$ is
continuous on $\Theta ^{\circ }$ by Lemma~\ref{DerivativesOfQ_nAndQ} in
Appendix \ref{App G}. Since for all $i,j=1,\ldots ,m$ 
\begin{equation*}
\sup_{\theta \in \Theta ^{\circ }}\left\vert \frac{\partial ^{2}\mathbb{Q}%
_{n}}{\partial \theta _{i}\partial \theta _{j}}(\theta )-\frac{\partial ^{2}Q%
}{\partial \theta _{i}\partial \theta _{j}}(\theta )\right\vert =o_{\mathbb{P%
}}(1)\quad \text{as }n\rightarrow \infty
\end{equation*}%
by Proposition~\ref{Lemma: convergence of D2Q_n} in Appendix \ref{App G}, it
follows that there are events $E_{n}$ having probability tending to $1$ as $%
n\rightarrow \infty $ such that on $E_{n}$%
\begin{equation*}
\inf_{\theta \in U}y^{\prime }\frac{\partial ^{2}\mathbb{Q}_{n}}{\partial
\theta \partial \theta ^{\prime }}(\theta )y\geq K\Vert y\Vert ^{2}\quad 
\text{for all }y\in \mathbb{R}^{m}
\end{equation*}%
holds for some constant $K>0$ which does not depend on $n$ or the data. By
Propositions~\ref{Proposition: Consistency of II-estimators}, $\hat{\theta}%
_{n}$ and $\hat{\theta}_{n,k(n)}$ belong to $U$ on subsets $E_{n}^{\prime }$
of the sample space whose inner probability goes to $1$ as $n\rightarrow
\infty $. For the rest of the proof of Step~1 we restrict our reasoning to
the events $E_{n}\cap E_{n}^{\prime }$, and note that they have inner
probability tending to $1$ as $n\rightarrow \infty $. By Proposition~\ref%
{Existence of II-estimators}(a) and Proposition~\ref{Proposition: Existence
of simulation-based II-estimators}(a) the estimators $\hat{\theta}_{n}$ and $%
\hat{\theta}_{n,k(n)}$, respectively, minimize the objective functions $%
\mathbb{Q}_{n}$ and $\mathbb{Q}_{n,k(n)}$. Hence, we may apply Lemma~\ref%
{Square root lemma} with $f=\mathbb{Q}_{n,k(n)}|U$, $g=\mathbb{Q}_{n}|U$, $u=%
\hat{\theta}_{n,k(n)}$, and $v=\hat{\theta}_{n}$ to obtain 
\begin{equation*}
\Vert \hat{\theta}_{n,k(n)}-\hat{\theta}_{n}\Vert \leq 2K^{-1/2}\sqrt{\Vert 
\mathbb{Q}_{n,k(n)}-\mathbb{Q}_{n}\Vert _{U}}.
\end{equation*}%
It follows from Proposition~\ref{Proposition: Uniform convergence properties
of II-objective functions}(c) in Appendix \ref{App G} and the choice of $%
k(n) $ that (\ref{diff}) holds under (i) as well as under (ii).

\textbf{Step~2:} Now assume that $\zeta =0$. Note that $\inf_{x\in \Omega
}p_{\blacktriangle }(x)>0$ and $\inf_{\Omega \times \Theta }p(x,\theta )>0$
because of Assumptions~\ref{dens:InternalPoint} and \ref{mod: Strict
inequality} (\ref{mod:InclusionWithStrictInequalities}, respectively). Let $%
\chi >0$ be such that $\inf_{x\in \Omega }p_{\blacktriangle }(x)>\chi $ and $%
\inf_{\Omega \times \Theta }p(x,\theta )>\chi $. Then it follows from
Proposition~\ref{Existence of II-estimators}(b) and Proposition~\ref%
{Proposition: Existence of simulation-based II-estimators}(c) that there are
events $C_{n,k(n)}$ having probability tending to $1$ as $n\rightarrow
\infty $ such that on these events $\hat{\theta}_{n,k(n)}$ coincides with a
SMD-estimator $\check{\theta}_{n,k(n)}$ that is based on $\mathcal{P}(t,\chi
,D)$ instead of $\mathcal{P}(t,\zeta ,D)$. Since the assumptions of the
theorem are also satisfied with $\mathcal{P}(t,\chi ,D)$ instead of $%
\mathcal{P}(t,\zeta ,D)$, applying to $\check{\theta}_{n,k(n)}$ what has
already been established in Step~1 completes the proof.
\end{proof}

\begin{remark}
\label{rem26}\normalfont(i) The preceding theorem was proved by showing that 
$\hat{\theta}_{n,k(n)}$ and $\hat{\theta}_{n}$ are sufficiently close (with
Lemma \ref{Square root lemma} being instrumental here) and by applying
Theorem \ref{Theorem: asymptotic normality of the II-estimator}. The reason
for going this route instead of directly applying a mean-value expansion to
the score $\partial \mathbb{Q}_{n,k(n)}/\partial \theta $ is that this would
require knowledge about differentiability properties of the mapping $\theta
\mapsto \tilde{p}_{k(n)}(\theta )$, which we were unable to obtain. [The
usual approach to establish such differentiability properties via the
implicit function theorem is not feasible here since $\tilde{p}%
_{k(n)}(\theta )$ falls on the boundary of $\mathcal{P}(t,\zeta ,D)$ as
shown in Proposition \ref{Theorem: existence of AML-estimators}.] A
consequence of the method of proof chosen is that we have to assume at least 
$k(n)/n^{2}\rightarrow \infty $. It is likely, that if the more direct
method of proof via expansion of the score $\partial \mathbb{Q}%
_{n,k(n)}/\partial \theta $ can be made to work, this would deliver
asymptotic normality under weaker conditions on $k(n)$.

(ii) Nickl and P\"{o}tscher (2010) consider spline projection density
estimators rather than NPML-estimators. Because of the simpler structure of
these estimators, this allows them to also employ the alternative route via
a mean-value expansion, leading to an asymptotic normality result under
weaker growth-conditions on $k(n)$. We note that Nickl and P\"{o}tscher
(2010) consider only the correctly specified case. In this case and when $%
k(n)/n^{2}\rightarrow \infty $ is assumed (as is in the present paper), the
assumptions employed in Nickl and P\"{o}tscher (2010) and in the present
paper are quite comparable, some differences being due to the different
non-parametric estimators considered.

(iii) The asymptotic normality results given here are for a fixed underlying
data-generating mechanism $\mathbb{P}$. Under appropriate assumptions,
corresponding results that are uniform in the underlying data-generating
mechanism can be obtained, see Chapter 7 in Gach (2010).
\end{remark}

\appendix{}

\section{Appendix: Proofs for Sections \protect\ref{Section: Notation and
basic definitions} and \protect\ref{Section: The framework}\label{App A}}

\textbf{Proof of Proposition \ref{Proposition:Propertiesoftheauxiliarymodel}%
: }(a) The implications (i) in (ii) and (ii) in (iii) are obvious. If $p$ is
an element of $\mathcal{P}(t,\zeta ,D)$, we have $1=\int_{\Omega
}p\,d\lambda \geq \int_{\Omega }\zeta d\lambda =\zeta \lambda (\Omega )$
showing that $\zeta \leq \lambda (\Omega )^{-1}$. Furthermore, the
Cauchy-Schwarz inequality implies $1=\left\Vert p\right\Vert _{1}\leq
\left\Vert p\right\Vert _{2}\left\Vert 1\right\Vert _{2}\leq \left\Vert
p\right\Vert _{t,2}\left\Vert 1\right\Vert _{2}\leq D\lambda (\Omega )^{1/2}$%
, which implies $\lambda (\Omega )^{-1}\leq D^{2}$. Thus (iii) implies (i).

(b) Suppose (i) holds. Then $\lambda (\Omega )^{-1}\in \mathcal{P}(t,\zeta
,D)$ by Part (a). Suppose $p\in \mathcal{P}(t,\zeta ,D)$. If now $\zeta
=\lambda (\Omega )^{-1}$, then $p-\lambda (\Omega )^{-1}\geq 0$. But clearly 
$\int_{\Omega }\left( p-\lambda (\Omega )^{-1}\right) d\lambda =0$, implying
that $p=\lambda (\Omega )^{-1}$ $\lambda $-a.e., and hence everywhere by
continuity of $p$. If $\lambda (\Omega )^{-1}=D^{2}$, then $\left\Vert
p\right\Vert _{1}=\left\Vert p\right\Vert _{2}\left\Vert 1\right\Vert _{2}$
follows from the calculations in the proof of Part (a). But this shows that $%
p$ is $\lambda $-a.e., and hence everywhere by continuity of $p$,
proportional to the constant function $1$, the proportionality factor
necessarily being $\lambda (\Omega )^{-1}$. This proves that (i) implies
(ii). That (ii) implies (iii) is trivial. Since the constant density $%
\lambda (\Omega )^{-1}$ belongs to $\mathcal{P}(t,\zeta ,D)$ by Part (a),
(iii) is equivalent to (ii). To show that (ii) implies (i), assume that $%
\zeta <\lambda (\Omega )^{-1}<D^{2}$. Choose $\varepsilon >0$ small enough
such that $\zeta <\lambda (\Omega )^{-1}-\varepsilon $ holds. Then define $f$
to be the restriction to $\Omega $ of the affine function that has the value 
$\lambda (\Omega )^{-1}-\varepsilon $ at the left endpoint of $\Omega $ and $%
\lambda (\Omega )^{-1}+\varepsilon $ at the right endpoint. By construction $%
f\in \mathsf{W}_{2}^{t}(\Omega )$, integrates to $1$, satisfies $%
\inf_{\Omega }f\geq \zeta $, and $\Vert f\Vert _{t,2}\leq D$ provided $%
\varepsilon $ is small enough. That is, $f$ is a further element of $%
\mathcal{P}(t,\zeta ,D)$, contradicting (ii).

(c) Note that $\mathcal{P}(t,\zeta ,D)$ is non-empty by Part (a). Since the
defining conditions are convex, it is convex. That $\mathcal{P}(t,\zeta ,D)$
is compact as claimed follows from Lemma~3 in Nickl (2007). {[}Note that the
proof of this lemma does not use that $\zeta >0$, as is implicit there, and
therefore is also valid for $\zeta =0$.{] }$\blacksquare $

\textbf{Proof of Proposition \ref{prop_interior}: }Since (a) is a special
case of (b) it suffices to prove the latter: Suppose $\mathcal{P}^{\prime }$
satisfies (i) and (ii), and choose $\delta >0$ small enough such that $%
\delta <D-\sup_{p\in \mathcal{P}^{\prime }}\Vert p\Vert _{t,2}$ and $%
C_{t}\delta <\inf_{x\in \Omega ,p\in \mathcal{P}^{\prime }}p(x)-\zeta $
hold, where $C_{t}$ is the constant appearing in Proposition \ref%
{prop:SobolevembedsinHoelder}. For every $p\in \mathcal{P}^{\prime }$ and $%
f\in \mathsf{W}_{2}^{t}(\Omega )$ with $\Vert f\Vert _{t,2}\leq \delta $ we
then have $\Vert p+f\Vert _{t,2}\leq \Vert p\Vert _{t,2}+\Vert f\Vert
_{t,2}\leq \sup_{p\in \mathcal{P}^{\prime }}\Vert p\Vert _{t,2}+\delta <D$
and $\inf_{\Omega }(p+f)\geq \inf_{\Omega }p-\sup_{\Omega }f\geq \inf_{x\in
\Omega ,p\in \mathcal{P}^{\prime }}p(x)-C_{t}\delta >\zeta $ (for the latter
using Proposition \ref{prop:SobolevembedsinHoelder}). This shows that $%
\mathcal{U}_{t,\delta }(p)\cap \mathsf{H}_{t}$ is a subset of $\mathcal{P}%
(t,\zeta ,D)$ for every $p\in \mathcal{P}^{\prime }$. Conversely, suppose $%
\mathcal{P}^{\prime }$ is uniformly interior to $\mathcal{P}(t,\zeta ,D)$
relative to $\mathsf{H}_{t}$. We first establish (i): Let $\delta >0$ be the
radius figuring in the definition of being uniformly interior and let $p\in 
\mathcal{P}^{\prime }$ be arbitrary. Choose a $q\in \mathsf{H}_{t}$
different from $p$ and define $f=\delta (q-p)/(2\Vert q-p\Vert _{t,2})$.
[Note that $q$ and hence $f$ may depend on $p$.] Then $f\neq 0$, $\Vert
f\Vert _{t,2}=\delta /2<\delta $, and $\int_{\Omega }fd\lambda =0$ hold.
Observe that $p+f$ and $p-f$ then both belong to $\mathcal{U}_{t,\delta
}(p)\cap \mathsf{H}_{t}$ and hence to $\mathcal{P}(t,\zeta ,D)$, since $%
\mathcal{U}_{t,\delta }(p)\cap \mathsf{H}_{t}\subseteq \mathcal{P}(t,\zeta
,D)$ by assumption; in particular $\Vert p+f\Vert _{t,2}\leq D$ and $\Vert
p-f\Vert _{t,2}\leq D$ is satisfied. Since the Sobolev-norm originates from
an inner product, we have $\Vert p+f\Vert _{t,2}^{2}+\Vert p-f\Vert
_{t,2}^{2}=2\left[ \Vert p\Vert _{t,2}^{2}+\Vert f\Vert _{t,2}^{2}\right] $
and thus $\Vert p\Vert _{t,2}^{2}\leq D^{2}-\delta ^{2}/4$. Since this is
true for every $p\in \mathcal{P}^{\prime }$ we obtain (i). We finally prove
(ii): Let $x_{n}\in \Omega $ and $p_{n}\in \mathcal{P}^{\prime }$ satisfy $%
p_{n}(x_{n})\rightarrow \inf_{x\in \Omega ,p\in \mathcal{P}^{\prime }}p(x)$.
The sequence $x_{n}$ has a cluster point $x_{0}$ in the closure $\bar{\Omega}
$ of the interval $\Omega $. There exists a sufficiently small neighborhood $%
A$ of $x_{0}$ in $\bar{\Omega}$ and a $C^{\infty }$ function $h$ satisfying $%
h(x)=-1$ for all $x\in A\cap \Omega $ (which is non-empty) as well as $%
\int_{\Omega }hd\lambda =0$. Furthermore, $h$ can be chosen to be bounded
with all its derivatives having compact support contained in $\Omega $;
consequently,\ $h\in \mathsf{W}_{2}^{t}(\Omega )$. Since $\mathcal{P}%
^{\prime }$ is uniformly interior to $\mathcal{P}(t,\zeta ,D)$ relative to $%
\mathsf{H}_{t}$ by assumption, it follows that $p_{n}+\alpha h\in \mathcal{P}%
(t,\zeta ,D)$ for sufficiently small $\alpha >0$, where $\alpha $ can be
chosen independently of $n$. Consequently, $\inf_{\Omega }\left(
p_{n}+\alpha h\right) \geq \zeta $ must hold. But this implies $%
p_{n}(x_{n})\geq \inf_{\Omega }p_{n}=\inf_{A\cap \Omega }p_{n}=\inf_{A\cap
\Omega }\left( p_{n}-\alpha \right) +\alpha =\inf_{A\cap \Omega }\left(
p_{n}+\alpha h\right) +\alpha \geq \inf_{\Omega }\left( p_{n}+\alpha
h\right) +\alpha \geq \zeta +\alpha $, which in turn implies $\inf_{x\in
\Omega ,p\in \mathcal{P}^{\prime }}p(x)\geq \zeta +\alpha >\zeta $. Finally,
we prove Part (c): Note that $\lambda (\Omega )^{-1}\in \mathcal{P}(t,\zeta
,D)$ by Proposition \ref{Proposition:Propertiesoftheauxiliarymodel}. It is
interior to $\mathcal{P}(t,\zeta ,D)$ relative to $\mathsf{H}_{t}$ by Part
(a) of the current proposition and the assumption $\zeta <\lambda (\Omega
)^{-1}<D^{2}$.{\ The second claim then follows from Theorem V.2.1. in
Dunford and Schwartz (1966). }$\blacksquare $

\begin{prop}
\label{Proposition: equivalence of pointwise and sup-norm convergence} Let $%
p_{n},p\in \mathcal{P}(t,\zeta ,D)$. Then the following statements are
equivalent: (i) $\Vert p_{n}-p\Vert _{\Omega }$ converges to $0$; (ii) $%
p_{n} $ converges pointwise to $p$; (iii) $p_{n}$ converges to $p$ a.e.;
(iv) $p_{n}$ converges to $p$ on a dense subset of $\Omega $; (v) $\Vert
p_{n}-p\Vert _{r,2}$ converges to $0$ for some $r$ satisfying $0\leq r<t$;
(vi) $\Vert p_{n}-p\Vert _{r,2}$ converges to $0$ for all $r$ satisfying $%
0\leq r<t$.
\end{prop}

\begin{proof}
To show that (v) implies (vi), it suffices, in light of Part~(c) of
Proposition~\ref{prop:SobolevembedsinHoelder}, to show that $\Vert
p_{n}-p\Vert _{s,2}$ converges to $0$ for arbitrary $s\geq r$ satisfying $%
1/2<s<t$. Since $\mathcal{P}(t,\zeta ,D)$ is a compact subset of $\mathsf{W}%
_{2}^{s}(\Omega )$ in view of Proposition~\ref%
{Proposition:Propertiesoftheauxiliarymodel}, for any subsequence $%
p_{n^{\prime }}$ of $p_{n}$ there exists a further subsequence $p_{n^{\prime
\prime }}$ of $p_{n^{\prime }}$ and a $p^{\ast }\in \mathcal{P}(t,\zeta ,D)$
such that $\Vert p_{n^{\prime \prime }}-p^{\ast }\Vert _{s,2}$ converges to $%
0$. By Part~(c) of Proposition~\ref{prop:SobolevembedsinHoelder}, we then
have that also $\Vert p_{n^{\prime \prime }}-p^{\ast }\Vert _{r,2}$
converges to $0$ since $s\geq r$. Because also $\Vert p_{n^{\prime \prime
}}-p\Vert _{r,2}$ converges to $0$ as a consequence of (v) and keeping in
mind that $p$ and $p^{\ast }$ are continuous, it follows that $p^{\ast }=p$.
This shows that $\Vert p_{n}-p\Vert _{s,2}$ converges to $0$. Furthermore,
(i) implies (ii), (ii) implies (iii), and (iii) implies (iv). That (vi)
implies (i) is a direct consequence of Part~(b) of Proposition~\ref%
{prop:SobolevembedsinHoelder}. It remains to show that (iv) implies (v).
Choose $r$ such that $1/2<r<t$. The same compactness argument as above shows
that for any subsequence $p_{n^{\prime }}$ of $p_{n}$ there exists a further
subsequence $p_{n^{\prime \prime }}$ of $p_{n^{\prime }}$ and a $p^{\ast
}\in \mathcal{P}(t,\zeta ,D)$ such that $\Vert p_{n^{\prime \prime
}}-p^{\ast }\Vert _{r,2}$ converges to $0$. By Part~(b) of Proposition~\ref%
{prop:SobolevembedsinHoelder}, we have that $\Vert p_{n^{\prime \prime
}}-p^{\ast }\Vert _{\Omega }$ converges to $0$. Consequently, $p$ and $%
p^{\ast }$ coincide on a dense subset of $\Omega $. Since $p$ and $p^{\ast }$
are continuous, they are identical. This shows that $\Vert p_{n^{\prime
\prime }}-p\Vert _{r,2}$ converges to $0$, and hence the same is true for
the entire sequence $p_{n}$.
\end{proof}

\begin{remark}
\normalfont\label{pyth} We note that $\mathcal{P}(t,\zeta ,D)$ can
equivalently be written as 
\begin{equation*}
\left\{ p\in \mathsf{W}_{2}^{t}(\Omega ):\int_{\Omega }p\,d\lambda
=1,\,\inf_{x\in \Omega }p(x)\geq \zeta ,\,\Vert p-\lambda ^{-1}(\Omega
)\Vert _{t,2}^{2}\leq D^{2}-\lambda ^{-1}(\Omega )\right\}
\end{equation*}%
because $p-\lambda ^{-1}(\Omega )$ and $1$ are orthogonal in $\mathsf{W}%
_{2}^{t}(\Omega )$. As a consequence, $\mathcal{P}(t,0,D)=\mathcal{P}%
(t,\zeta ,D)$ at least for all $0\leq \zeta \leq \lambda ^{-1}(\Omega
)-C_{t}\left( D^{2}-\lambda ^{-1}(\Omega )\right) ^{1/2}$, since $p\in 
\mathcal{P}(t,0,D)$ implies $\inf_{x\in \Omega }p(x)\geq \zeta $ for such $%
\zeta $ by Proposition \ref{prop:SobolevembedsinHoelder}(b).
\end{remark}

Assumptions on the density functions in the class $\mathcal{P}_{\Theta }$
and on the simulation mechanism $\rho $ are of course related to each other,
but the interrelationship is somewhat intricate. The following proposition
collects two important observations.

\begin{prop}
\label{Interrelation}\hspace*{0ex}If Assumption~\ref{mod:Inclusion} is
satisfied, then Assumption~\ref{rho: Continuity of rho} implies Assumption~%
\ref{mod: pointwise continuity}. However, in general Assumption~\ref{rho:
Continuity of rho} does not imply Assumption~\ref{mod: pointwise continuity}.
\end{prop}

\begin{proof}
The first claim is proved as follows: Let $F(z,\theta )=\int_{\left\{ x\in
\Omega :\,x\leq z\right\} }p_{\theta }\,d\lambda $ be the distribution
function on $\Omega $ that is associated with $p_{\theta }$. Let $\theta
_{n},\theta \in \Theta $ be such that $\theta _{n}$ converges to $\theta $.
Now Assumption~\ref{rho: Continuity of rho} implies that $\rho (\cdot
,\theta _{n})$ converges to $\rho (\cdot ,\theta )$ in distribution under $%
\mu $. Noting that $F(\cdot ,\theta )$ and $F(\cdot ,\theta _{n})$ are the
distribution functions of $\rho (\cdot ,\theta )$ and $\rho (\cdot ,\theta )$%
, respectively, as well as noting that $F(\cdot ,\theta )$ is continuous in
its first argument, it follows that $F(z,\theta _{n})$ converges to $%
F(z,\theta )$ for every $z\in \Omega $. By Assumption \ref{mod:Inclusion}
and sup-norm compactness of $\mathcal{P}(t,\zeta ,D)$ it follows that every
subsequence $p_{\theta _{n^{\prime }}}$ of $p_{\theta _{n}}$ has a further
subsequence $p_{\theta _{n^{\prime \prime }}}$ that converges to an element $%
p^{\ast }\in \mathcal{P}(t,\zeta ,D)$ in the sup-norm. But this clearly
implies that $F(z,\theta _{n^{\prime \prime }})$ converges to $\int_{\left\{
x\in \Omega :\,x\leq z\right\} }p^{\ast }\,d\lambda $ for every $z\in \Omega 
$. It follows that $p^{\ast }=p_{\theta }$ a.e., hence everywhere on $\Omega 
$ by continuity of $p_{\theta }$ and $p^{\ast }$. This proves the first
claim. For a proof of the second claim see Proposition 5 in Gach (2010).
\end{proof}

\section{Appendix: Properties of the Non-Parametric Likelihood Function\label%
{App B}}

\begin{prop}
\label{Properties of L_n and L_k}\hspace*{0ex}

(a) For every non-negative $\mathcal{B}(\Omega )$-measurable real-valued
function $f$ the map $(x_{1},\ldots ,x_{n})\mapsto L_{n}(f;x_{1},\ldots
,x_{n})$ is $\mathcal{B}(\Omega )^{n}$-$\mathcal{B}([-\infty ,\infty ))$%
-measurable, and the map $(v_{1},\ldots ,v_{k})\mapsto L_{k}(\theta
,f;v_{1},\ldots ,v_{k})$ is $\mathcal{V}^{k}$-$\mathcal{B}([-\infty ,\infty
))$-measurable for every $\theta \in \Theta $.

(b) Let $\mathcal{F}$ be a set of non-negative bounded real-valued functions
on $\Omega $.

\qquad (b1) Then, for every $(x_{1},\ldots ,x_{n})\in \Omega ^{n}$, $%
f\mapsto L_{n}(f;x_{1},\ldots ,x_{n})$ is a continuous map from $(\mathcal{F}%
,\Vert \cdot \Vert _{\Omega })$ to $[-\infty ,\infty )$. The same is true
for the map $f\mapsto L_{k}(\theta ,f;v_{1},\ldots ,v_{k})$ for every $%
\theta \in \Theta $ and every $(v_{1},\ldots ,v_{k})\in V^{k}$.

\qquad (b2) If the elements $f\in \mathcal{F}$ are additionally also
continuous and Assumption~\ref{rho: Continuity of rho} is satisfied, then,
for every $(v_{1},\ldots ,v_{k})\in V^{k}$, $(\theta ,f)\mapsto L_{k}(\theta
,f;v_{1},\ldots ,v_{k})$ is a continuous map from $\Theta \times (\mathcal{F}%
,\Vert \cdot \Vert _{\Omega })$ to $[-\infty ,\infty )$.

(c) Let $\mathcal{F}$ be a set of non-negative bounded $\mathcal{B}(\Omega )$%
-measurable real-valued functions on $\Omega $ that are uniformly bounded
away from $0$.

\qquad (c1) Then $L(f)$ is a continuous real-valued function on $(\mathcal{F}%
,\Vert \cdot \Vert _{\Omega })$. The same is true for $L(\theta ,f)$ for
every given $\theta \in \Theta $.

\qquad (c2) If the elements $f\in \mathcal{F}$ are additionally also
continuous and Assumption~\ref{rho: Continuity of rho} is satisfied, then $%
L(\theta ,f)$ is a continuous real-valued function on $\Theta \times (%
\mathcal{F},\Vert \cdot \Vert _{\Omega })$.

(d) Let $\mathcal{F}$ be a sup-norm compact set of non-negative bounded $%
\mathcal{B}(\Omega )$-measurable real-valued functions on $\Omega $ that are
uniformly bounded away from $0$.

$\qquad $(d1) Then 
\begin{equation*}
\lim_{n\rightarrow \infty }\,\sup_{f\in \mathcal{F}}\left\vert
L_{n}(f)-L(f)\right\vert =0\quad \mathbb{P}\text{-a.s.,}
\end{equation*}%
and, for every $\theta \in \Theta $,%
\begin{equation*}
\lim_{k\rightarrow \infty }\,\sup_{f\in \mathcal{F}}\left\vert L_{k}(\theta
,f)-L(\theta ,f)\right\vert =0\quad \mu \text{-a.s.}
\end{equation*}

\qquad (d2) If the elements $f\in \mathcal{F}$ are additionally also
continuous and Assumption~\ref{rho: Continuity of rho} is satisfied, then 
\begin{equation*}
\lim_{k\rightarrow \infty }\,\sup_{\Theta \times \mathcal{F}}\left\vert
L_{k}(\theta ,f)-L(\theta ,f)\right\vert =0\quad \mu \text{-a.s.}
\end{equation*}%
(In Part (d) we use the convention that the supremum is $0$ if $\mathcal{F}$
is empty.)
\end{prop}

\begin{proof}
(a) The first claim is clear as $f$ is $\mathcal{B}(\Omega )$-$\mathcal{B}%
([0,\infty ))$-measurable by hypothesis and the extended logarithm is $%
\mathcal{B}([0,\infty ))$-$\mathcal{B}([-\infty ,\infty ))$-measurable. For
the second claim additionally use that $\rho :V\times \Theta \rightarrow
\Omega $ is $\mathcal{V}$-$\mathcal{B}(\Omega )$-measurable in the first
argument for every $\theta \in \Theta $.

(b) To prove the first claim in Part (b1), fix $(x_{1},\ldots ,x_{n})\in
\Omega ^{n}$. Let $f_{l},f\in \mathcal{F}$ be such that $\Vert f_{l}-f\Vert
_{\Omega }$ converges to $0$. Since setting $\log 0=-\infty $ continuously
extends the logarithm to the interval $[0,\infty )$, $\log f_{l}(x_{i})$
then converges to $\log f(x_{i})$ for every $i$, thus establishing the first
claim. The second claim in Part (b1) is proved analogously. To prove Part
(b2), fix $(v_{1},\ldots ,v_{k})\in V^{k}$ and let $\theta _{l},\theta \in
\Theta $ and $f_{l},f\in \mathcal{F}$ be such that $\left\Vert \theta
_{l}-\theta \right\Vert $ and $\Vert f_{l}-f\Vert _{\Omega }$ converge to $0$%
. Use the triangle inequality to obtain for every $i$%
\begin{eqnarray}
|f_{l}(\rho (v_{i},\theta _{l}))-f(\rho (v_{i},\theta ))| &\leq &|f_{l}(\rho
(v_{i},\theta _{l}))-f(\rho (v_{i},\theta _{l}))|+|f(\rho (v_{i},\theta
_{l}))-f(\rho (v_{i},\theta ))|  \notag \\
&\leq &\Vert f_{l}-f\Vert _{\Omega }+|f(\rho (v_{i},\theta _{l}))-f(\rho
(v_{i},\theta ))|.  \label{Continuity of log}
\end{eqnarray}%
The first expression on the r.h.s.\ of (\ref{Continuity of log}) converges
to $0$ by hypothesis. Making use of Assumption \ref{rho: Continuity of rho}
and the continuity of $f$, the second one converges to $0$ as well.
Continuity of the extended logarithm on $[0,\infty )$ delivers Part (b2).

(c) To prove the first claim in Part (c1), denote by $\xi >0$ the lower
uniform bound of all elements in $\mathcal{F}$. Let $f_{l},f\in \mathcal{F}$
be such that $\Vert f_{l}-f\Vert _{\Omega }$ converges to $0$. Then $\left\{
f_{l}:l\in \mathbb{N}\right\} $ is bounded by some $B$, $0<B<\infty $. Since
the logarithm is bounded on $[\xi ,B]$, the domination condition 
\begin{equation*}
\int_{\Omega }\sup_{l\in \mathbb{N}}\left\vert \log f_{l}(x)\right\vert d%
\mathbb{P}(x)<\infty
\end{equation*}%
is satisfied. By the already established Part (b1) (with $n=1$), $\log
f_{l}(x)$ converges to $\log f(x)$ for every $x\in \Omega $. The first claim
then follows from the theorem of dominated convergence. The second claim in
Part (c1) is proved in exactly the same manner. To prove Part (c2), let $%
\theta _{l},\theta \in \Theta $ and $f_{l},f\in \mathcal{F}$ be such that $%
\left\Vert \theta _{l}-\theta \right\Vert $ and $\Vert f_{l}-f\Vert _{\Omega
}$ converge to $0$. By the same argument as before, the domination condition 
\begin{equation*}
\int_{V}\sup_{\theta \in \Theta }\,\sup_{l\in \mathbb{N}}\left\vert \log
f_{l}(\rho (v,\theta ))\right\vert d\mu (v)<\infty
\end{equation*}%
is satisfied. By the already established Part (b2) (with $k=1$), $\log
f_{l}(\rho (v,\theta _{l}))$ converges to $\log f(\rho (v,\theta ))$ for
every $v\in V$. Part (c2) then follows from the theorem of dominated
convergence.

(d) To prove the first claim in Part (d1), we use Mourier's strong law of
large numbers as given in Corollary~7.10 of Ledoux and Talagrand (1991) with
the separable Banach space $(B,\Vert \cdot \Vert )$ given by $(\mathsf{C}(%
\mathcal{F},\Vert \cdot \Vert _{\Omega }),\Vert \cdot \Vert _{\mathcal{F}})$
and the mapping $X$ given by $X(f)=\log f(X_{1})-\int_{\Omega }\log fd%
\mathbb{P}$ for $f\in \mathcal{F}$. Note that $X$ has values in $\mathsf{C}(%
\mathcal{F},\Vert \cdot \Vert _{\Omega })$ by using the already established
Parts (b1) and (c1) in conjunction with the assumed sup-norm compactness of $%
\mathcal{F}$. Clearly, $X(f)$ is a random variable for every $f\in \mathcal{F%
}$, and hence $X$ is measurable with respect to the $\sigma $-field on $%
\mathsf{C}(\mathcal{F},\Vert \cdot \Vert _{\Omega })$ that is generated by
the point-evaluations. Since this $\sigma $-field coincides with the Borel $%
\sigma $-field on $\mathsf{C}(\mathcal{F},\Vert \cdot \Vert _{\Omega })$
(see, e.g., Problem 1 in Section 1.7 in van der Vaart and Wellner (1996) and
observe that $(\mathcal{F},\Vert \cdot \Vert _{\Omega })$ is a compact
metric space), $X$ is a Borel random mapping. The integrability condition $%
\func{E}\Vert X\Vert <\infty $ follows from%
\begin{equation*}
\int_{\Omega }\sup_{f\in \mathcal{F}}|\log f(x)|d\mathbb{P}(x)<\infty ,
\end{equation*}%
which is true since the elements of $\mathcal{F}$ are uniformly bounded and
uniformly bounded away from $0$ by hypothesis. The second claim in Part (d1)
is proved completely analogously. Part (d2) is proved in a similar manner:
Apply Corollary~7.10 in Ledoux and Talagrand (1991) with $B$ the separable
Banach space of all bounded, continuous functions on $\Theta \times (%
\mathcal{F},\Vert \cdot \Vert _{\Omega })$ equipped with the sup-norm $\Vert
\cdot \Vert _{\Theta \times \mathcal{F}}$ and with $X$ given by $X(\theta
,f)=\log f(\rho (V_{1},\theta ))-\int_{V}\log f(\rho (\cdot ,\theta ))d\mu $%
. Note that by the already established Parts (b2) and (c2) in conjunction
with compactness of $\Theta \times (\mathcal{F},\Vert \cdot \Vert _{\Omega
}) $, $X$ takes its values in the space of (bounded) continuous functions on 
$\Theta \times (\mathcal{F},\Vert \cdot \Vert _{\Omega })$. Again $X$ is a
Borel random mapping. The integrability condition $\func{E}\Vert X\Vert
<\infty $ now follows from%
\begin{equation*}
\int_{V}\sup_{\theta \in \Theta }\,\sup_{f\in \mathcal{F}}|\log f(\rho
(v,\theta ))|d\mu (v)<\infty ,
\end{equation*}%
which is true since the elements of $\mathcal{F}$ are uniformly bounded and
uniformly bounded away from $0$ by hypothesis.
\end{proof}

\textbf{Proof of Lemma \ref{lem: p_0 is the unique maximizer}: }(a) It is
sufficient to show that 
\begin{equation*}
\int_{\Omega }\left( \log (p_{\blacktriangle })\right) ^{-}d\mathbb{P}%
=\int_{\left\{ x\in \Omega :\,p_{\blacktriangle }(x)>0\right\} }\left( \log
(p_{\blacktriangle })\right) ^{-}p_{\blacktriangle }d\lambda <\infty .
\end{equation*}%
By Assumption~\ref{Data} this is equivalent to showing that%
\begin{equation}
\int_{\left\{ x\in \Omega :\,0<p_{\blacktriangle }(x)\leq 1\right\}
}h(p_{\blacktriangle })d\lambda <\infty ,  \label{f3}
\end{equation}%
where $h(y)$ is defined by $h(y)=-y\log y$ for every $y\in (0,1]$. Since $%
h:(0,1]\rightarrow \lbrack 0,\infty )$ can be continuously extended to $%
[0,1] $ by setting $h(0)=0$, it is bounded on the compact interval $[0,1]$,
and a fortiori on $(0,1]$. But this establishes (\ref{f3}) since $\lambda
(\Omega )<\infty $ and thus completes the proof for $L$. The proof for $%
L(\theta ,\cdot )$ is analogous upon observing that%
\begin{equation}
\int_{V}\left( \log p_{\theta }(\rho (\cdot ,\theta ))\right) ^{-}d\mu
=\int_{\left\{ x\in \Omega :\,p(x,\theta )>0\right\} }\left( \log (p_{\theta
})\right) ^{-}p_{\theta }d\lambda  \label{change_variable}
\end{equation}%
by the change of variable theorem.

(b) For any $p\in \mathcal{P}(t,\zeta ,D)$ different from $p_{\blacktriangle
}$, the set $\left\{ x\in \Omega :p(x)\neq p_{\blacktriangle }(x)>0\right\} $
has positive $\mathbb{P}$-probability since $p$ and $p_{\blacktriangle }$
are continuous functions on $\Omega $. In view of the already established
Part (a) the expression $L(p)-L(p_{\blacktriangle })$ is well-defined, and
the strict Jensen inequality gives%
\begin{equation*}
L(p)-L(p_{\blacktriangle })=\int_{\left\{ x\in \Omega :\,p_{\blacktriangle
}(x)>0\right\} }\log \frac{p}{p_{\blacktriangle }}d\mathbb{P}<\log
\int_{\left\{ x\in \Omega :\,p_{\blacktriangle }(x)>0\right\} }\frac{p}{%
p_{\blacktriangle }}d\mathbb{P}\leq 0.
\end{equation*}

(c) Follows similarly to Part (b) in view of the representation 
\begin{equation*}
L(\theta ,p_{\theta })=\int_{\left\{ x\in \Omega :\,p(x,\theta )>0\right\}
}\log (p_{\theta })p_{\theta }d\lambda .
\end{equation*}%
$\blacksquare $

\bigskip

Part (a) of the following proposition is essentially given in Proposition~3
in Nickl (2007). [We note that the set $\mathcal{V}$ defined there is not
sup-norm open as implicitly claimed, the apparently intended definition in
the notation of Nickl (2007) being $\mathcal{V}=\left\{ d\in \QTR{up}{%
\mathsf{L}}^{\infty }(\Omega ):\inf_{x\in \Omega }d(x)>\zeta /2\right\} $.
Inspection of the proof shows that this proposition remains correct for $%
\zeta =0$.] The proof for Part (b) is completely analogous.

\begin{prop}
\label{prop:DerivativesofthescaledlikelihoodNickl} Define $\mathcal{U}%
=\left\{ f\in \QTR{up}{\mathsf{L}}^{\infty }(\Omega ):\inf_{x\in \Omega
}f(x)>0\right\} $. Let $\alpha $ be a positive integer, $f\in \mathcal{U}$,
and $f_{1},\ldots ,f_{\alpha }\in \QTR{up}{\mathsf{L}}^{\infty }(\Omega )$.

(a) The $\alpha $-th Fr\'{e}chet derivatives of $L_{n}:\mathcal{U}%
\rightarrow \mathbb{R}$ and $L:\mathcal{U}\rightarrow \mathbb{R}$ are given
by%
\begin{equation*}
\mathbf{D}^{\alpha }L_{n}(f)(f_{1},\ldots ,f_{\alpha })=(-1)^{\alpha
-1}(\alpha -1)!\mathbb{P}_{n}(f^{-\alpha }f_{1}\cdots f_{\alpha }),
\end{equation*}%
\begin{equation*}
\mathbf{D}^{\alpha }L(f)(f_{1},\ldots ,f_{\alpha })=(-1)^{\alpha -1}(\alpha
-1)!\mathbb{P}(f^{-\alpha }f_{1}\cdots f_{\alpha }).
\end{equation*}

(b) The $\alpha $-th partial Fr\'{e}chet derivatives of $L_{k}:\Theta \times 
\mathcal{U}\rightarrow \mathbb{R}$ and $L:\Theta \times \mathcal{U}%
\rightarrow \mathbb{R}$ with respect to the second variable are, for $\theta
\in \Theta $, given by%
\begin{equation*}
\mathbf{D}^{\alpha }L_{k}(\theta ,f)(f_{1},\ldots ,f_{\alpha })=(-1)^{\alpha
-1}(\alpha -1)!\mu _{k}(f^{-\alpha }(\rho (\cdot ,\theta ))f_{1}(\rho (\cdot
,\theta ))\cdots f_{\alpha }(\rho (\cdot ,\theta ))),
\end{equation*}%
\begin{eqnarray*}
\mathbf{D}^{\alpha }L(\theta ,f)(f_{1},\ldots ,f_{\alpha }) &=&(-1)^{\alpha
-1}(\alpha -1)!\mu (f^{-\alpha }(\rho (\cdot ,\theta ))f_{1}(\rho (\cdot
,\theta ))\cdots f_{\alpha }(\rho (\cdot ,\theta ))) \\
&=&(-1)^{\alpha -1}(\alpha -1)!\int_{\Omega }f^{-\alpha }f_{1}\cdots
f_{\alpha }\,p_{\theta }d\lambda .
\end{eqnarray*}
\end{prop}

The next result is a uniform version of Lemma~2 in Nickl (2007). It provides
rates of convergence for all derivatives of the auxiliary log-likelihood
function that hold uniformly in $\theta $ and $p$.

\begin{prop}
\label{Closeness of the derivatives of the scaled likelihood} Let $\alpha $
be a positive integer, and let $\mathcal{H}_{1},\ldots ,\mathcal{H}_{\alpha
} $ be bounded subsets of some Sobolev space $\QTR{up}{\mathsf{W}}%
_{2}^{s}(\Omega )$ of order $s>1/2$. If Assumption \ref{rho: Hoelderity of
rho} and $\zeta >0$ are satisfied, then 
\begin{equation}
\sup_{\Theta \times \mathcal{P}(t,\zeta ,D)}\left\Vert \mathbf{D}^{\alpha
}L_{k}(\theta ,p)-\mathbf{D}^{\alpha }L(\theta ,p)\right\Vert _{\mathcal{H}%
_{1}\times \cdots \times \mathcal{H}_{\alpha }}=O_{\mu }(k^{-1/2})\quad 
\text{as }k\rightarrow \infty .  \label{unif_diff_deriv}
\end{equation}
\end{prop}

\begin{proof}
Note that 
\begin{equation*}
\sup_{\Theta \times \mathcal{P}(t,\zeta ,D)}\left\Vert \mathbf{D}^{\alpha
}L_{k}(\theta ,p)-\mathbf{D}^{\alpha }L(\theta ,p)\right\Vert _{\mathcal{H}%
_{1}\times \cdots \times \mathcal{H}_{\alpha }}=(\alpha -1)!\,\Vert \mu
_{k}-\mu \Vert _{\mathcal{H}^{\ast }}
\end{equation*}%
by Proposition \ref{prop:DerivativesofthescaledlikelihoodNickl}, where $%
\mathcal{H}^{\ast }=\left\{ h(\rho (\cdot ,\theta )):h\in \mathcal{H},\theta
\in \Theta \right\} $ and%
\begin{equation*}
\mathcal{H}=\left\{ p^{-\alpha }h_{1}\cdot \ldots \cdot h_{\alpha }:\,p\in 
\mathcal{P}(t,\zeta ,D),\,h_{1}\in \mathcal{H}_{1},\ldots ,h_{\alpha }\in 
\mathcal{H}_{\alpha }\right\} .
\end{equation*}%
Since $\zeta >0$, the class $\mathcal{H}$ is a bounded subset of the
Sobolev-space $\QTR{up}{\mathsf{W}}_{2}^{r}(\Omega )$ with $r=\min (t,s)>1/2$
by Proposition \ref{prop:SobolevembedsinHoelder}. Measurability of the
supremum on the l.h.s.~of (\ref{unif_diff_deriv}) now follows immediately
from Proposition \ref{Borel_2} in Appendix \ref{App D2}. The class $\mathcal{%
H}^{\ast }$ is $\mu $-Donsker by an application of Proposition~\ref%
{Bracketing entropy of F*}(a), hence $\Vert \mu _{k}-\mu \Vert _{\mathcal{H}%
^{\ast }}$ is bounded in probability at rate $k^{-1/2}$ by Prohorov's
theorem.
\end{proof}

\bigskip

The following lemma is a special case of Berge's (1963) maximum theorem.

\begin{lem}
\label{lem: Continuous selections} Let $X$ be a metrizable space and $Y$ a
compact metrizable space. Let $u:X\times Y\rightarrow \lbrack -\infty
,\infty )$ be a continuous function that has a unique maximizer, say $v(x)$,
on the fiber $\{(x,y):y\in Y\}$ for every $x\in X$. Then the mapping $%
v:X\rightarrow Y$ is continuous.
\end{lem}

\section{Appendix: Proofs for Section \protect\ref{Rates}\label{App D}}

The following lemma is a consequence of Birman and Solomyak (1967),
cf.~Lorentz, v.Golitschek, and Makovoz (1996), p. 506. It can also be
obtained from Theorem 1 in Nickl and P\"{o}tscher~(2007) via a retraction
argument; see Gach (2010).

\begin{lem}
\label{Lemma: bounded subsets are uniformly Donsker} Let $\mathcal{F}$ be a
bounded subset of the Sobolev space $\mathsf{W}_{2}^{s}(\Omega )$ of order $%
s>1/2$. Then the sup-norm metric entropy of $\mathcal{F}$ satisfies%
\begin{equation*}
H(\varepsilon ,\mathcal{F},\mathsf{W}_{2}^{s}(\Omega ),\Vert \cdot \Vert
_{\Omega })\lesssim \varepsilon ^{-1/s}.
\end{equation*}
\end{lem}

\bigskip

\textbf{Proof of Proposition \ref{Bracketing entropy of F*}: }(a) Choose a
real number $r\leq s$ satisfying $1/2<r<3/2$ and $2r-1\leq a$, where $a$ is
as in Assumption \ref{rho: Hoelderity of rho}. Then $\mathcal{F}$ can also
be viewed as a bounded subset of $\mathsf{W}_{2}^{r}(\Omega )$, and hence of 
$\mathsf{C}^{r-1/2}(\Omega )$, in view of Proposition~\ref%
{prop:SobolevembedsinHoelder}(b),(c). We use this to obtain%
\begin{equation*}
\sup_{f\in \mathcal{F}}|f(\rho (v,\theta ^{\prime }))-f(\rho (v,\theta
))|\leq L_{r}|\rho (v,\theta ^{\prime })-\rho (v,\theta )|^{r-1/2}\leq L_{r}%
\left[ R(v)\Vert \theta ^{\prime }-\theta \Vert ^{\gamma }\right] ^{r-1/2}
\end{equation*}%
for some finite constant $L_{r}>0$ and all $v\in V$, all $\theta ,\theta
^{\prime }\in \Theta $, where we have made use of Assumption \ref{rho:
Hoelderity of rho}. A cover of $\mathcal{F}^{\ast }$ is obtained from
suitable covers of $\Theta $ and $\mathcal{F}$ as follows: Fix $\varepsilon
>0$ and set $\delta (\varepsilon )=(\varepsilon /L_{r})^{1/\nu }$, where $%
\nu :=\gamma (r-1/2)$. To cover $\Theta $, note that it is contained in an $%
m $-cube of edge length $l$ and thus in the union of at most $\lceil l\sqrt{m%
}/\delta (\varepsilon )\rceil ^{m}$-many closed Euclidean balls $B(\theta
_{i},\delta (\varepsilon ))$ with centers $\theta _{i}\in \Theta $ and
radius $\delta (\varepsilon )$, where $\lceil x\rceil $ denotes the smallest
integer not less than $x$. To cover $\mathcal{F}$, we take $N(\varepsilon ,%
\mathcal{F},\mathsf{W}_{2}^{s}(\Omega ),\Vert \cdot \Vert _{\Omega })$-many
sup-norm closed balls $[f_{j}-2\varepsilon ,f_{j}+2\varepsilon ]$ of radius $%
2\varepsilon $ whose centers $f_{j}$ already belong to $\mathcal{F}$. {[}%
Note that this can always be achieved.{]} We claim that the brackets%
\begin{equation}
\lbrack f_{j}(\rho (\cdot ,\theta _{i}))-R^{\nu /\gamma }(\cdot )\varepsilon
-2\varepsilon ,f_{j}(\rho (\cdot ,\theta _{i}))+R^{\nu /\gamma }(\cdot
)\varepsilon +2\varepsilon ]  \label{Cover of F*}
\end{equation}%
with $i=1,\ldots ,\lceil l\sqrt{m}/\delta (\varepsilon )\rceil ^{m}$ and $%
j=1,\ldots ,N(\varepsilon ,\mathcal{F},\mathsf{W}_{2}^{s}(\Omega ),\Vert
\cdot \Vert _{\Omega })$ provide a cover of $\mathcal{F}^{\ast }$. To see
this, let $h\in \mathcal{F}^{\ast }$, that is, $h=f(\rho (\cdot ,\theta ))$
for some $\theta \in \Theta $ and $f\in \mathcal{F}$, implying that there
are indices $i,j$ such that $\theta \in B(\theta _{i},\delta (\varepsilon ))$
and $f\in \lbrack f_{j}-2\varepsilon ,f_{j}+2\varepsilon ]$. Consequently, 
\begin{equation*}
h\in \lbrack f_{j}(\rho (\cdot ,\theta ))-2\varepsilon ,f_{j}(\rho (\cdot
,\theta ))+2\varepsilon ].
\end{equation*}%
Now, 
\begin{eqnarray*}
h(v)\leq f_{j}(\rho (v,\theta ))+2\varepsilon &\leq &f_{j}(\rho (v,\theta
_{i}))+|f_{j}(\rho (v,\theta ))-f_{j}(\rho (v,\theta _{i}))|+2\varepsilon \\
&\leq &f_{j}(\rho (v,\theta _{i}))+R^{\nu /\gamma }(v)\varepsilon
+2\varepsilon
\end{eqnarray*}%
for all $v\in V$, where the last inequality follows from the first display
in the proof and the choice of $\delta (\varepsilon )$. Similarly, 
\begin{equation*}
f_{j}(\rho (v,\theta _{i}))-R^{\nu /\gamma }(v)\varepsilon -2\varepsilon
\leq h(v).
\end{equation*}%
By construction of $r$, we have that $\int_{V}\left( R^{\nu /\gamma }\right)
^{2}d\mu <\infty $, and hence the $\mathcal{L}^{2}(\mu )$-bracketing size of
any of the brackets in (\ref{Cover of F*}) can be bounded by $\varepsilon $
times a positive constant $c$ that only depends on $R$, $r$, and $\mu $.
Using the elementary inequality $\lceil x\rceil ^{m}\leq \max (1,(2x)^{m})$
this leads to the relationship 
\begin{equation*}
N_{[\hspace{0.75ex}]}(c\varepsilon ,\mathcal{F}^{\ast },\Vert \cdot \Vert
_{2,\mu })\leq \max (1,\left( 2l\sqrt{m}L_{r}^{1/\nu }\right)
^{m}\varepsilon ^{-m/\nu })\,N(\varepsilon ,\mathcal{F},\mathsf{W}%
_{2}^{s}(\Omega ),\Vert \cdot \Vert _{\Omega }).
\end{equation*}%
Apply Lemma~\ref{Lemma: bounded subsets are uniformly Donsker} to get 
\begin{equation*}
H_{[\hspace{0.75ex}]}(\varepsilon ,\mathcal{F}^{\ast },\Vert \cdot \Vert
_{2,\mu })\lesssim \max (0,1-\log \varepsilon )+\varepsilon ^{-1/s}\lesssim
\varepsilon ^{-1/s}
\end{equation*}%
which proves (\ref{bracketing metric entropy of F*}). The claim that $%
\mathcal{F}^{\ast }$ is $\mu $-Donsker now follows from Ossiander's central
limit theorem (see Theorem~7.2.1 in Dudley, 1999) since clearly $\mathcal{F}%
^{\ast }\subseteq \mathcal{L}^{2}(V,\mathcal{V},\mu )$ holds.

(b) For any fixed $\varepsilon >0$, we take for $\mathcal{F}^{\ast }$ the
cover given in (\ref{Cover of F*}). Since the elements of $\mathcal{F}$ are
bounded below by $\chi >0$, the sets%
\begin{equation*}
\left[ \log \max (\chi ,f_{j}(\rho (\cdot ,\theta _{i}))-R^{\nu /\gamma
}(\cdot )\varepsilon -2\varepsilon ),\log (f_{j}(\rho (\cdot ,\theta
_{i}))+R^{\nu /\gamma }(\cdot )\varepsilon +2\varepsilon )\right] ,
\end{equation*}%
for $i=1,\ldots ,\lceil l\sqrt{m}/\delta (\varepsilon )\rceil ^{m}$, $%
j=1,\ldots ,N(\varepsilon ,\mathcal{F},\mathsf{W}_{2}^{s}(\Omega ),\Vert
\cdot \Vert _{\Omega })$ are non-empty brackets and cover $\log \mathcal{F}%
^{\ast }$. Since the logarithm is Lipschitz on $[\chi ,\infty )$ with
Lipschitz constant $\chi ^{-1}$, the $\mathcal{L}^{2}(\mu )$-bracketing size
of these brackets can be bounded by $\chi ^{-1}$ times the $\mathcal{L}%
^{2}(\mu )$-bracketing size of the corresponding brackets given in (\ref%
{Cover of F*}). Arguing now as in the proof of Part (a) completes the proof. 
$\blacksquare $

\section{Appendix: Measurability Issues\label{App D2}}

\begin{lem}
\label{Borel}Suppose $t>1/2$. Then the Borel $\sigma $-fields $\mathcal{B}(%
\mathsf{W}_{2}^{t}(\Omega ),\Vert \cdot \Vert _{\Omega })$, and $\mathcal{B}(%
\mathsf{W}_{2}^{t}(\Omega ),\Vert \cdot \Vert _{s,2})$ for $0\leq s\leq t$
all coincide. In particular, the norms $\Vert \cdot \Vert _{\Omega }$ and $%
\Vert \cdot \Vert _{s,2}$ for $0\leq s\leq t$ are $\mathcal{B}(\mathsf{W}%
_{2}^{t}(\Omega ),\Vert \cdot \Vert _{\Omega })$-measurable.
\end{lem}

\begin{proof}
Since the $\Vert \cdot \Vert _{\Omega }$-topology on $\mathsf{W}%
_{2}^{t}(\Omega )$ is coarser than the $\Vert \cdot \Vert _{s,2}$-topology
on $\mathsf{W}_{2}^{t}(\Omega )$, which in turn is coarser than the $\Vert
\cdot \Vert _{t,2}$-topology on $\mathsf{W}_{2}^{t}(\Omega )$
(cf.~Proposition \ref{prop:SobolevembedsinHoelder}), it suffices to show
that $\mathcal{B}(\mathsf{W}_{2}^{t}(\Omega ),\Vert \cdot \Vert
_{t,2})\subseteq \mathcal{B}(\mathsf{W}_{2}^{t}(\Omega ),\Vert \cdot \Vert
_{\Omega })$. The former $\sigma $-field is generated by the collection of
all closed $\Vert \cdot \Vert _{t,2}$-balls since $(\mathsf{W}%
_{2}^{t}(\Omega ),\Vert \cdot \Vert _{t,2})$ is separable. As shown in the
proof of Lemma 3 in Nickl (2007), these balls are $\Vert \cdot \Vert
_{\Omega }$-compact and hence belong to $\mathcal{B}(\mathsf{W}%
_{2}^{t}(\Omega ),\Vert \cdot \Vert _{\Omega })$.
\end{proof}

\begin{prop}
\label{Borel_1}(a) The quantities $\Vert \hat{p}_{n}-p_{\blacktriangle
}\Vert _{\Omega }$, $\Vert \hat{p}_{n}-p_{\blacktriangle }\Vert _{s,2}$ for $%
0\leq s\leq t$, $\Vert \tilde{p}_{k}(\theta )-p_{\theta }\Vert _{\Omega }$,
and $\Vert \tilde{p}_{k}(\theta )-p_{\theta }\Vert _{s,2}$ for $0\leq s\leq
t $ are random variables.

(b) Suppose Assumptions \ref{mod:Inclusion} and \ref{rho: Continuity of rho}
are satisfied. Then $\sup_{\theta \in \Theta }\Vert \tilde{p}_{k}(\theta
)-p_{\theta }\Vert _{\Omega }$ and \linebreak $\sup_{\theta \in \Theta
}\Vert \tilde{p}_{k}(\theta )-p_{\theta }\Vert _{s,2}$ for $0\leq s<t$ are
random variables.
\end{prop}

\begin{proof}
(a) Follows immediately from Theorem \ref{Theorem: existence of
AML-estimators} and Lemma \ref{Borel}. (b) By Assumption \ref{rho:
Continuity of rho} and Proposition \ref{Interrelation} in Appendix \ref{App
A} the parameterization $\theta \mapsto p_{\theta }(x)$ is continuous, and
hence is continuous in the $\Vert \cdot \Vert _{\Omega }$- and $\Vert \cdot
\Vert _{s,2}$-norms ($0\leq s<t$) in view of Assumption \ref{mod:Inclusion}
and Proposition \ref{Proposition: equivalence of pointwise and sup-norm
convergence} in Appendix \ref{App A}. By Theorem \ref{Theorem: existence of
AML-estimators}(b) and again Proposition \ref{Proposition: equivalence of
pointwise and sup-norm convergence} $\theta \mapsto \tilde{p}_{k}(\theta
)-p_{\theta }$ is then continuous in the same norms. Since $\Theta $ is
separable, (b) follows from Part (a).
\end{proof}

\begin{prop}
\label{Borel_2}Suppose $s>1/2$.

(a) Then 
\begin{equation*}
\mathfrak{X}_{n}(\breve{x},f)=\sqrt{n}\left( \int_{\Omega }\hat{p}_{n}(\cdot
;x_{1},\ldots ,x_{n})f(\cdot )d\lambda -\mathbb{P}(f)\right)
\end{equation*}%
and%
\begin{equation*}
\mathfrak{Y}_{n}(\breve{x},f)=n^{-1/2}\sum_{i=1}^{n}\left( f(x_{i})-\mathbb{P%
}(f)\right)
\end{equation*}%
are Borel measurable on $\Omega ^{n}$ for every $f\in \mathsf{W}%
_{2}^{s}(\Omega )$, where $\breve{x}$ denotes $(x_{1},\ldots ,x_{n})\in
\Omega ^{n}$. Furthermore, if $\mathcal{F}$ is a non-empty bounded subset of 
$\mathsf{W}_{2}^{s}(\Omega )$, then $\sup_{f\in \mathcal{F}}\left\vert 
\mathfrak{Z}_{n}(\breve{x},f)\right\vert $ is Borel measurable on $\Omega
^{n}$, where $\mathfrak{Z}_{n}$ stands for any of $\mathfrak{X}_{n}$, $%
\mathfrak{Y}_{n}$, and $\mathfrak{X}_{n}-\mathfrak{Y}_{n}$.

(b) Then 
\begin{equation*}
\mathfrak{U}_{k}(\breve{v},\theta ,f)=\sqrt{k}\int_{\Omega }\left( \tilde{p}%
_{k}(\theta )(\cdot ;v_{1},\ldots ,v_{k})-p_{\theta }(\cdot )\right) f(\cdot
)d\lambda
\end{equation*}%
and%
\begin{equation*}
\mathfrak{V}_{k}(\breve{v},\theta ,f)=k^{-1/2}\sum_{i=1}^{k}\left( f(\rho
(v_{i},\theta ))-\mu (f(\rho (\cdot ,\theta )))\right)
\end{equation*}%
are Borel measurable on $V^{k}$ for every $\theta \in \Theta $ and every $%
f\in \mathsf{W}_{2}^{s}(\Omega )$, where $\breve{v}$ denotes $(v_{1},\ldots
,v_{k})\in V^{k}$. Furthermore, if Assumption \ref{rho: Continuity of rho}
is satisfied and $\mathcal{F}$ is a non-empty bounded subset of $\mathsf{W}%
_{2}^{s}(\Omega )$, then $\sup_{\theta \in \Theta }\sup_{f\in \mathcal{F}%
}\left\vert \mathfrak{V}_{k}(\breve{v},\theta ,f)\right\vert $ is Borel
measurable on $V^{k}$; if, additionally, Assumption~\ref{mod:Inclusion}
holds, then $\sup_{\theta \in \Theta }\sup_{f\in \mathcal{F}}\left\vert 
\mathfrak{W}_{k}(\breve{v},\theta ,f)\right\vert $ is Borel measurable on $%
V^{k}$, where $\mathfrak{W}_{k}$ stands for any of $\mathfrak{U}_{k}$ and $%
\mathfrak{U}_{k}-\mathfrak{V}_{k}$.

(c) Then 
\begin{equation*}
\mathfrak{T}_{k}(\breve{v},\theta ,f)=k^{-1}\sum_{i=1}^{k}\tilde{p}%
_{k}^{-1}(\theta )(\rho (v_{i},\theta );v_{1},\ldots ,v_{k})f(\rho
(v_{i},\theta ))
\end{equation*}%
is Borel measurable on $V^{k}$ for every $\theta \in \Theta $ and every $%
f\in \mathsf{W}_{2}^{s}(\Omega )$. Furthermore, if Assumption \ref{rho:
Continuity of rho} is satisfied, $\mathcal{F}$ is a non-empty bounded subset
of $\mathsf{W}_{2}^{s}(\Omega )$, and $\zeta >0$ holds, then $\sup_{\theta
\in \Theta }\sup_{f\in \mathcal{F}}\left\vert \mathfrak{T}_{k}(\breve{v}%
,\theta ,f)\right\vert $ is Borel measurable on $V^{k}$.
\end{prop}

\begin{proof}
(a) Since $(x_{1},\ldots ,x_{n})\mapsto \hat{p}_{n}(\cdot ;x_{1},\ldots
,x_{n})$ is a measurable map from $\Omega ^{n}$ into $(\mathcal{P}(t,\zeta
,D),\Vert \cdot \Vert _{\Omega })$ by Theorem \ref{Theorem: existence of
AML-estimators}, since the map $p\mapsto \sqrt{n}\left( \int pfd\lambda -%
\mathbb{P}(f)\right) $ is $\Vert \cdot \Vert _{\Omega }$-continuous on $%
\mathcal{P}(t,\zeta ,D)$ for every $f\in \mathsf{W}_{2}^{s}(\Omega )$, and
since every $f$ is clearly Borel measurable, we see that $\mathfrak{X}_{n}(%
\breve{x},f)$ as well as $\mathfrak{Y}_{n}(\breve{x},f)$ are Borel
measurable on $\Omega ^{n}$ for every $f\in \mathsf{W}_{2}^{s}(\Omega )$.
Furthermore, it is easy to see that $\mathfrak{X}_{n}(\breve{x},f)$ and $%
\mathfrak{Y}_{n}(\breve{x},f)$, and thus also $\mathfrak{X}_{n}(\breve{x},f)-%
\mathfrak{Y}_{n}(\breve{x},f)$, are continuous on $(\mathcal{F},\Vert \cdot
\Vert _{\Omega })$ for given $\breve{x}$. Since $(\mathcal{F},\Vert \cdot
\Vert _{\Omega })$ is clearly separable, Borel measurability of the suprema
in Part (a) follows.

(b) The first claim is proved completely analogous, making also use of the
fact that $\rho $ is measurable in its first argument. The second claim is
also proved analogously by showing that now $\mathfrak{U}_{k}(\breve{v}%
,\theta ,f)$ and $\mathfrak{V}_{k}(\breve{v},\theta ,f)$ are continuous on
the separable space $(\Theta \times \mathcal{F},\left\Vert \cdot \right\Vert
+\Vert \cdot \Vert _{\Omega })$ for given $\breve{v}$: for $\mathfrak{V}_{k}$
use that $\theta \mapsto \rho (v,\theta )$ is continuous on $\Theta $ by
Assumption~\ref{rho: Continuity of rho} and that $\mathcal{F}$ is a sup-norm
bounded set of continuous functions. For $\mathfrak{U}_{k}$ use the fact
that $\theta \mapsto \tilde{p}_{k}(\theta )$ as a mapping from $\Theta $
into the space $(\mathcal{P}(t,\zeta ,D),\Vert \cdot \Vert _{\Omega })$ is
continuous by Theorem \ref{Theorem: existence of AML-estimators}, and that
the same is true for $p_{\theta }$ in view of Assumption \ref{mod:Inclusion}%
, Proposition \ref{Interrelation} in Appendix \ref{App A}, and Remark \ref%
{Remark: Equivalence of pointwise and sup-norm convergence of the
parametrization}.

(c) Measurability of $\mathfrak{T}_{k}(\cdot ,\theta ,f)$ for $\theta \in
\Theta $ and $f\in \mathsf{W}_{2}^{s}(\Omega )$ follows from measurability
of $f$ and $\rho (\cdot ,\theta )$ and Remark \ref{Remark: meas_pos}(i).
Continuity of $\mathfrak{T}_{k}(\breve{v},\cdot ,\cdot )$ on the separable
space $(\Theta \times \mathcal{F},\left\Vert \cdot \right\Vert +\Vert \cdot
\Vert _{\Omega })$ follows from continuity of $\tilde{p}_{k}(\theta )(\cdot
;v_{1},\ldots ,v_{k})$ and $f(\cdot )$, Assumption \ref{rho: Continuity of
rho}, and $\zeta >0$.
\end{proof}

\section{Appendix: Uniform Rates of Convergence and Entropy Bounds for
Empirical Processes \label{App E}}

The subsequent theorem is a uniform version of Theorem~3.2.5 in van der
Vaart and Wellner (1996).

\begin{thm}
\label{Uniform rates of convergence: parameter set} Let $(\Lambda ,\mathcal{A%
},P)$ be a probability space, $S$ and $T$ non-empty sets, and let $d$ be a
non-negative real-valued function on $T\times T$. Consider a sequence of
real-valued stochastic processes $(H_{k}(\sigma ,\tau ):\sigma \in S,\,\tau
\in T)$ defined on $(\Lambda ,\mathcal{A})$ and a function $H:S\times
T\rightarrow \mathbb{R}$ with the property that for every $\sigma \in S$
there exists a $\tau (\sigma )\in T$ such that for all $\tau \in T$%
\begin{equation}
H(\sigma ,\tau )-H(\sigma ,\tau (\sigma ))\leq -Cd^{\alpha }(\tau ,\tau
(\sigma ))  \label{Concavity}
\end{equation}%
holds, where $C,\alpha >0$ are constants neither depending on $\sigma $ nor $%
\tau $. Suppose, for all $\delta >0$, 
\begin{equation}
\func{E}^{\ast }\sup_{\sigma \in S}\sup_{\tau \in T,d(\tau ,\tau (\sigma
))\leq \delta }\sqrt{k}\left\vert (H_{k}-H)(\sigma ,\tau )-(H_{k}-H)(\sigma
,\tau (\sigma ))\right\vert \leq \varphi _{k}(\delta )
\label{Modulus of continuity}
\end{equation}%
is satisfied for real-valued functions $\varphi _{k}$ such that for some $%
\beta <\alpha $ the functions $\delta \mapsto \delta ^{-\beta }\varphi
_{k}(\delta )$ are all non-increasing in $\delta $. Assume further that, for
every $\sigma \in S$, $\hat{\tau}_{k}(\sigma ):\Lambda \rightarrow T$
satisfies%
\begin{equation}
H_{k}(\sigma ,\hat{\tau}_{k}(\sigma ))\geq H_{k}(\sigma ,\tau )\quad \text{%
for all }\tau \in T,  \label{Maximizer condition}
\end{equation}%
and let $r_{k}$ be a sequence of positive reals such that 
\begin{equation}
\sup_{k\in \mathbb{N}}\frac{r_{k}^{\alpha }\varphi _{k}(r_{k}^{-1})}{\sqrt{k}%
}<\infty .  \label{Rate of convergence condition}
\end{equation}%
Then, for every $\sigma \in S$, $\tau (\sigma )$ is a maximizer of $H(\sigma
,\cdot )$, and%
\begin{equation*}
\sup_{\sigma \in S}d(\hat{\tau}_{k}(\sigma ),\tau (\sigma ))=O_{P}^{\ast
}(r_{k}^{-1})\quad \text{as $k\rightarrow \infty $.}
\end{equation*}
\end{thm}

\begin{proof}
We have to show that for every $N\in \mathbb{N}$%
\begin{equation*}
\lim_{N\rightarrow \infty }\limsup_{k\rightarrow \infty }P^{\ast }\left(
r_{k}\sup_{\sigma \in S}d(\hat{\tau}_{k}(\sigma ),\tau (\sigma
))>2^{N}\right) =0.
\end{equation*}%
For $k,j\in \mathbb{N}$, set $V_{k,j}=\{(\sigma ,\tau ):2^{j-1}<r_{k}d(\tau
,\tau (\sigma ))\leq 2^{j}\}$. Then 
\begin{equation*}
r_{k}\sup_{\sigma \in S}d(\hat{\tau}_{k}(\sigma ),\tau (\sigma ))>2^{N}
\end{equation*}%
implies that there is some $\sigma _{0}\in S$ such that $r_{k}d(\hat{\tau}%
_{k}(\sigma _{0}),\tau (\sigma _{0}))>2^{N}$, which in turn gives $(\sigma
_{0},\hat{\tau}_{k}(\sigma _{0}))\in V_{k,j_{0}}$ for some $j_{0}>N$.
Combine this with (\ref{Concavity}) and (\ref{Maximizer condition}) to get 
\begin{equation*}
(H_{k}-H)(\sigma _{0},\hat{\tau}_{k}(\sigma _{0}))-(H_{k}-H)(\sigma
_{0},\tau (\sigma _{0}))\geq Cd^{\alpha }(\hat{\tau}_{k}(\sigma _{0}),\tau
(\sigma _{0}))>Cr_{k}^{-\alpha }2^{\alpha j_{0}-\alpha }.
\end{equation*}%
This implies 
\begin{eqnarray*}
\lefteqn{P^{\ast }\left( r_{k}\sup_{\sigma \in S}d(\hat{\tau}_{k}(\sigma
),\tau (\sigma ))>2^{N}\right) } \\
&\leq &\sum_{j>N}P^{\ast }\left( \sup_{(\sigma ,\tau )\in V_{k,j}}\left\vert 
\sqrt{k}(H_{k}-H)(\sigma ,\tau )-\sqrt{k}(H_{k}-H)(\sigma ,\tau (\sigma
))\right\vert \geq C\sqrt{k}r_{k}^{-\alpha }2^{\alpha j-\alpha }\right) \!.
\end{eqnarray*}%
Via Markov's inequality (for outer probability) and (\ref{Modulus of
continuity}), the r.h.s.~in the previous display can be bounded by 
\begin{equation*}
\sum_{j>N}\frac{\varphi _{k}(2^{j}r_{k}^{-1})r_{k}^{\alpha }}{C\sqrt{k}%
2^{\alpha j-\alpha }}\leq \sum_{j>N}\frac{2^{\beta j}\varphi
_{k}(r_{k}^{-1})r_{k}^{\alpha }}{C\sqrt{k}2^{\alpha j-\alpha }}\leq \frac{%
2^{\alpha }}{C}\sup_{k\in \mathbb{N}}\frac{r_{k}^{\alpha }\varphi
_{k}(r_{k}^{-1})}{\sqrt{k}}\sum_{j>N}2^{(\beta -\alpha )j},
\end{equation*}%
where the first inequality follows from $\varphi _{k}(c\delta )\leq c^{\beta
}\varphi _{k}(\delta )$ for $c\geq 1$. Note that the upper bound is finite
by (\ref{Rate of convergence condition}) and does not depend on $k$; since $%
\sum_{j>N}2^{(\beta -\alpha )j}$ converges to $0$ as $N\rightarrow \infty $
as $\beta <\alpha $ holds, the proof is complete.
\end{proof}

\bigskip 

We next present an upper bound for $\func{E}^{\ast }\Vert \sqrt{n}%
(P_{n}-P)\Vert _{\mathcal{F}}$ for sup-norm bounded classes of functions $%
\mathcal{F}$. This result is essentially well-known, see Lemma~3.4.2 in van
der Vaart and Wellner (1996), but we provide \emph{explicit} constants. A
proof, under the additional assumption that $Y_{1},\ldots ,Y_{n}$ are the
coordinate projections on a product space, can be found in Gach (2010);
inspection of the proof reveals that this assumption is unnecessary.

\begin{thm}
\label{Van der Vaart lemma} Suppose $(\Lambda ,\mathcal{A},P)$ is a
probability space, $Y_{1},\ldots ,Y_{n}$ are i.i.d.~with law $P$, and $P_{n}$
denotes the empirical measure associated with $Y_{1},\ldots ,Y_{n}$. Let $%
\mathcal{F}$ be a non-empty class of $\mathcal{A}$-measurable functions on $%
\Lambda $, which are bounded by $B$, $0<B<\infty $, in the sup-norm and by $%
\eta $, $0<\eta <\infty $, with respect to $\Vert \cdot \Vert _{2,P}$. Then 
\begin{equation*}
\func{E}^{\ast }\Vert \sqrt{n}(P_{n}-P)\Vert _{\mathcal{F}}\leq (1696+64%
\sqrt{2})\,I_{[\hspace{0.75ex}]}(\eta ,\mathcal{F},\Vert \cdot \Vert _{2,P})%
\left[ 1+\frac{B}{\eta ^{2}\sqrt{n}}\,I_{[\hspace{0.75ex}]}(\eta ,\mathcal{F}%
,\Vert \cdot \Vert _{2,P})\right] \!.
\end{equation*}
\end{thm}

\section{Appendix: Auxiliary Results for SMD-Estimation\label{App G}}

\begin{lem}
\label{Lemma: Generic continuity of Q} Suppose $\mathcal{P}_{\Theta
}\subseteq \mathcal{L}^{2}(\Omega )$ and $\theta \mapsto p_{\theta }$ is a
continuous mapping from $\Theta $ into $(\mathcal{L}^{2}(\Omega ),\Vert
\cdot \Vert _{2})$. Let $f:\Omega \rightarrow \mathbb{R}$ be an integrable
function satisfying $\inf_{x\in \Omega }f(x)>0$. Then%
\begin{equation*}
H(\theta ):=\int_{\Omega }(f-p_{\theta })^{2}f^{-1}d\lambda
\end{equation*}%
is a continuous real-valued function on $\Theta $.
\end{lem}

\begin{proof}
Rewrite the integrand as $f-2p_{\theta }+p_{\theta }^{2}/f$, and note that
each term is integrable by the hypotheses. Hence, $H$ is real-valued. For
continuity, let $\theta _{l},\theta \in \Theta $ be such that $\Vert \theta
_{l}-\theta \Vert $ converges to $0$. Letting $c=\inf_{x\in \Omega }f(x)$,%
\begin{eqnarray*}
|H(\theta _{l})-H(\theta )| &=&\left\vert \int_{\Omega }p_{\theta
_{l}}^{2}f^{-1}d\lambda -\int_{\Omega }p_{\theta }^{2}f^{-1}d\lambda
\right\vert \leq c^{-1}\int_{\Omega }\left\vert p_{\theta
_{l}}^{2}-p_{\theta }^{2}\right\vert d\lambda \\
&\leq &c^{-1}\Vert p_{\theta _{l}}-p_{\theta }\Vert _{2}(\Vert p_{\theta
_{l}}-p_{\theta }\Vert _{2}+2\Vert p_{\theta }\Vert _{2})\rightarrow 0\text{
\ \ for }l\rightarrow \infty .
\end{eqnarray*}
\end{proof}

\begin{prop}
\label{Proposition: properties of Q_n and Q_n,k} (a) Suppose $\mathcal{P}%
_{\Theta }\subseteq \mathcal{L}^{2}(\Omega )$ and $\theta \mapsto p_{\theta
} $ is a continuous map from $\Theta $ into $(\mathcal{L}^{2}(\Omega ),\Vert
\cdot \Vert _{2})$. Then, on the event where $\inf_{x\in \Omega }\hat{p}%
_{n}(x)>0$, 
\begin{equation*}
\mathbb{Q}_{n}(\theta )=\int_{\Omega }(\hat{p}_{n}-p_{\theta })^{2}\hat{p}%
_{n}^{-1}d\lambda
\end{equation*}%
holds and $\mathbb{Q}_{n}$ is a continuous real-valued function on $\Theta $%
. [In particular, in case $\zeta >0$ holds, the above event is the entire
sample space $\Omega ^{n}$.]

(b) Let Assumption~\ref{rho: Continuity of rho} be satisfied. Then, on the
event where $\inf_{x\in \Omega }\hat{p}_{n}(x)>0$,%
\begin{equation*}
\mathbb{Q}_{n,k}(\theta )=\int_{\Omega }(\hat{p}_{n}-\tilde{p}_{k}(\theta
))^{2}\hat{p}_{n}^{-1}d\lambda
\end{equation*}%
holds and $\mathbb{Q}_{n,k}$ is a continuous real-valued function on $\Theta 
$. [In particular, in case $\zeta >0$ holds, the above event is the entire
sample space $\Omega ^{n}\times V^{k}$.]

(c) Suppose $\mathcal{P}_{\Theta }\subseteq \mathcal{L}^{2}(\Omega )$ and $%
\theta \mapsto p_{\theta }$ is a continuous map from $\Theta $ into $(%
\mathcal{L}^{2}(\Omega ),\Vert \cdot \Vert _{2})$. If Assumption \ref%
{dens:StrictInequality} holds, then $Q$ is a continuous real-valued function
on $\Theta $.
\end{prop}

\begin{proof}
Parts (a) and (c) are immediate consequences of Lemma~\ref{Lemma: Generic
continuity of Q}. We next prove Part (b): Since $\hat{p}_{n}$ and $\tilde{p}%
_{k}(\theta )$ belong to $\mathcal{P}(t,\zeta ,D)$ by construction, these
densities are sup-norm bounded by $C_{t}D$. Hence, $\mathbb{Q}_{n,k}$ is
real-valued whenever $\inf_{x\in \Omega }\hat{p}_{n}(x)>0$. Since the map $%
\theta \mapsto \tilde{p}_{k}(\theta )$ is continuous by Theorem~\ref%
{Theorem: existence of AML-estimators}(b), continuity of $\mathbb{Q}_{n,k}$
then follows from the theorem of dominated convergence.
\end{proof}

\begin{prop}
\label{Proposition: Uniform convergence properties of II-objective functions}
(a) Suppose $\mathcal{P}_{\Theta }\subseteq \mathcal{L}^{2}(\Omega )$ and $%
\theta \mapsto p_{\theta }$ is a continuous map from $\Theta $ into $(%
\mathcal{L}^{2}(\Omega ),\Vert \cdot \Vert _{2})$. Let further Assumptions %
\ref{dens:Element} and \ref{dens:StrictInequality} be satisfied. Then%
\begin{equation*}
\sup_{\theta \in \Theta }\left\vert \mathbb{Q}_{n}(\theta )-Q(\theta
)\right\vert =o_{\mathbb{P}}^{\ast }(1)\quad \text{as }n\rightarrow \infty .
\end{equation*}

(b) Let Assumptions \ref{dens:Element}, \ref{dens:StrictInequality}, \ref%
{mod:Inclusion}, \ref{mod: Strict inequality}, and \ref{rho: Continuity of
rho} be satisfied. Then%
\begin{equation*}
\sup_{\theta \in \Theta }\left\vert \mathbb{Q}_{n,k}(\theta )-Q(\theta
)\right\vert =o_{\func{Pr}}^{\ast }(1)\quad \text{as }\min (n,k)\rightarrow
\infty .
\end{equation*}

(c) Suppose $\zeta >0$ holds and Assumptions~\ref{mod:Inclusion} and \ref%
{rho: Hoelderity of rho} are satisfied. Then 
\begin{equation*}
\sup_{n\in \mathbb{N}}\,\sup_{\Omega ^{n}}\,\sup_{\theta \in \Theta }|%
\mathbb{Q}_{n,k}(\theta )-\mathbb{Q}_{n}(\theta )|=O_{\mu }^{\ast
}(k^{-t/(2t+1)})\quad \text{as $k\rightarrow \infty $.}
\end{equation*}%
If Assumption~\ref{mod:Inclusion} is strengthened to \ref%
{mod:InclusionWithStrictInequalities}, then%
\begin{equation*}
\sup_{n\in \mathbb{N}}\,\sup_{\Omega ^{n}}\,\sup_{\theta \in \Theta }|%
\mathbb{Q}_{n,k}(\theta )-\mathbb{Q}_{n}(\theta )|=O_{\func{Pr}}^{\ast
}(k^{-1/2})\quad \text{as $k\rightarrow \infty $.}
\end{equation*}
\end{prop}

\begin{proof}
(a) Set $\chi =2^{-1}\inf_{x\in \Omega }p_{\blacktriangle }(x)$ and observe
that $\chi >0$ by Assumption~\ref{dens:StrictInequality}. In view of Remark~%
\ref{Remark 12}(i) there is a sequence of events $A_{n}$ that have
probability converging to $1$ as $n\rightarrow \infty $ such that $%
\inf_{x\in \Omega }\hat{p}_{n}(x)>\chi $. On these events we then have%
\begin{equation*}
\sup_{\theta \in \Theta }|\mathbb{Q}_{n}(\theta )-Q(\theta )|=\sup_{\theta
\in \Theta }\left\vert \int_{\Omega }\frac{p_{\theta }^{2}}{\hat{p}_{n}}%
d\lambda -\int_{\Omega }\frac{p_{\theta }^{2}}{p_{\blacktriangle }}d\lambda
\right\vert \leq \chi ^{-2}\sup_{\theta \in \Theta }\Vert p_{\theta }\Vert
_{2}^{2}\,\Vert \hat{p}_{n}-p_{\blacktriangle }\Vert _{\Omega }.
\end{equation*}%
Since $\Theta $ is compact, the assumptions on $\mathcal{P}_{\Theta }$ imply
that $\sup_{\theta \in \Theta }\Vert p_{\theta }\Vert _{2}<\infty $. Part
(a) of Theorem~\ref{Theorem: consistency of AML-estimators} now completes
the proof.

(b) Let $\chi $ and $A_{n}$ be as in the proof of Part (a). On $A_{n}$ we
have%
\begin{eqnarray*}
\sup_{\theta \in \Theta }|\mathbb{Q}_{n,k}(\theta )-Q(\theta )| &\leq
&\sup_{\theta \in \Theta }\left\vert \int_{\Omega }\frac{\tilde{p}%
_{k}(\theta )^{2}}{\hat{p}_{n}}d\lambda -\int_{\Omega }\frac{p_{\theta }^{2}%
}{p_{\blacktriangle }}d\lambda \right\vert \\
&\leq &\sup_{\theta \in \Theta }\left\vert \int_{\Omega }(\tilde{p}%
_{k}(\theta )-p_{\theta })\frac{\tilde{p}_{k}(\theta )+p_{\theta }}{\hat{p}%
_{n}}d\lambda +\int_{\Omega }p_{\theta }^{2}\left( \frac{1}{\hat{p}_{n}}-%
\frac{1}{p_{\blacktriangle }}\right) d\lambda \right\vert \\
&\leq &2\chi ^{-1}\sup_{\theta \in \Theta }\Vert \tilde{p}_{k}(\theta
)-p_{\theta }\Vert _{\Omega }+\chi ^{-2}D^{2}\Vert \hat{p}%
_{n}-p_{\blacktriangle }\Vert _{\Omega }.
\end{eqnarray*}%
The result then follows from Parts (a) and (c) of Theorem~\ref{Theorem:
consistency of AML-estimators}.

(c) Note that $\tilde{p}_{k}(\theta )\in \mathcal{P}(t,\zeta ,D)$ by
construction and $p_{\theta }\in \mathcal{P}(t,\zeta ,D)$ by Assumption~\ref%
{mod:Inclusion}. Hence, these densities are sup-norm bounded uniformly in $%
\theta $ (and $v_{1},\ldots ,v_{k}\in V$ in case of $\tilde{p}_{k}(\theta )$%
). Observe now that%
\begin{equation*}
\mathbb{Q}_{n,k}(\theta )-\mathbb{Q}_{n}(\theta )=\int_{\Omega }(\tilde{p}%
_{k}(\theta )-p_{\theta })\,\frac{\tilde{p}_{k}(\theta )+p_{\theta }}{\hat{p}%
_{n}}d\lambda .
\end{equation*}%
Using $\zeta >0$, Part (d) of Proposition~\ref{prop:SobolevembedsinHoelder}
applied to $\left\{ \hat{p}_{n}:x_{1},\ldots ,x_{n}\in \Omega ,\,n\in 
\mathbb{N}\right\} $ shows that $\left\{ 1/\hat{p}_{n}:x_{1},\ldots
,x_{n}\in \Omega ,\,n\in \mathbb{N}\right\} $ is bounded in $\mathsf{W}%
_{2}^{t}(\Omega )$. By Assumption~\ref{mod:Inclusion} and the construction
of $\tilde{p}_{k}(\theta )$, it follows from Part (a) of Proposition~\ref%
{prop:SobolevembedsinHoelder} that%
\begin{equation}
\left\{ \frac{\tilde{p}_{k}(\theta )+p_{\theta }}{\hat{p}_{n}}:\theta \in
\Theta ,\,x_{1},\ldots ,x_{n}\in \Omega ,\,v_{1},\ldots ,v_{k}\in V,\,n,k\in 
\mathbb{N}\right\}  \label{Formel: Bounded set}
\end{equation}%
is contained in a Sobolev ball $\mathcal{U}_{t,B}$ for some $B$ satisfying $%
0<B<\infty $. The first claim then follows from Theorem~\ref{Uniform rate of
AML-estimators} with $s=0$ (note that under $\zeta >0$ Assumption \ref%
{mod:Inclusion} implies Assumption \ref{mod: Strict inequality}), where we
have made use of the inequality $\int_{\Omega }|f|d\lambda \leq \lambda
(\Omega )^{1/2}\Vert f\Vert _{2}$ and the fact that the set in (\ref{Formel:
Bounded set}) is bounded in the sup-norm. If Assumption~\ref{mod:Inclusion}
is strengthened to \ref{mod:InclusionWithStrictInequalities}, we may apply
Part (c) of Theorem~\ref{Theorem: Uniform Donsker-type theorem} with $%
\mathcal{F}$ equal to the set given in (\ref{Formel: Bounded set}) to obtain
the second claim.
\end{proof}

\begin{remark}
\normalfont If $\zeta >0$ holds, then the events $A_{n}$ in Parts (a) and
(b) of the above proof are the entire sample space and $\mathbb{Q}_{n}-Q$,
respectively $\mathbb{Q}_{n,k}-Q$, is continuous on $\Theta $. By
separability of $\Theta $, the measurability of the respective suprema then
follows.
\end{remark}

\begin{lem}
\label{DerivativesOfQ_nAndQ}\hspace*{0ex}(a) Let Assumptions~\ref%
{mod:Inclusion} and \ref{mod: domination conditions} be satisfied. Then, on
the event $\inf_{x\in \Omega }\hat{p}_{n}(x)>0$, the objective function $%
\mathbb{Q}_{n}$ is twice continuously partially differentiable on $\Theta
^{\circ }$ with%
\begin{eqnarray*}
\frac{\partial \mathbb{Q}_{n}}{\partial \theta _{i}}(\theta )
&=&-2\int_{\Omega }(\hat{p}_{n}-p_{\theta })\frac{\partial p}{\partial
\theta _{i}}(\cdot ,\theta )\hat{p}_{n}^{-1}d\lambda , \\
\frac{\partial ^{2}\mathbb{Q}_{n}}{\partial \theta _{i}\partial \theta _{j}}%
(\theta ) &=&2\int_{\Omega }\left( \frac{\partial p}{\partial \theta _{i}}%
(\cdot ,\theta )\frac{\partial p}{\partial \theta _{j}}(\cdot ,\theta )+%
\frac{\partial ^{2}p}{\partial \theta _{i}\partial \theta _{j}}(\cdot
,\theta )p_{\theta }\right) \hat{p}_{n}^{-1}d\lambda ,
\end{eqnarray*}%
for $i,j=1,\ldots ,m$.

(b) Let Assumptions~\ref{dens:StrictInequality}, \ref{mod:Inclusion}, and %
\ref{mod: domination conditions} be satisfied. Then $Q$ is twice
continuously partially differentiable on $\Theta ^{\circ }$ with%
\begin{eqnarray*}
\frac{\partial Q}{\partial \theta _{i}}(\theta ) &=&-2\int_{\Omega
}(p_{\blacktriangle }-p_{\theta })\frac{\partial p}{\partial \theta _{i}}%
(\cdot ,\theta )p_{\blacktriangle }^{-1}d\lambda , \\
\frac{\partial ^{2}Q}{\partial \theta _{i}\partial \theta _{j}}(\theta )
&=&2\int_{\Omega }\left( \frac{\partial p}{\partial \theta _{i}}(\cdot
,\theta )\frac{\partial p}{\partial \theta _{j}}(\cdot ,\theta )+\frac{%
\partial ^{2}p}{\partial \theta _{i}\partial \theta _{j}}(\cdot ,\theta
)p_{\theta }\right) p_{\blacktriangle }^{-1}d\lambda ,
\end{eqnarray*}%
for $i,j=1,\ldots ,m$.
\end{lem}

\begin{proof}
Note that the densities involved are all uniformly bounded by Assumption \ref%
{mod:Inclusion}. Under the respective assumptions, differentiation and
integration can be interchanged, leading to the above formulae upon noting
that the integral of $\partial ^{2}p/(\partial \theta _{i}\partial \theta
_{j})(\cdot ,\theta )$ is zero. Continuity of the partial derivatives
follows from the theorem of dominated convergence.
\end{proof}

\begin{prop}
\label{Lemma: convergence of D2Q_n} Let Assumptions \ref{dens:Element}, \ref%
{mod:Inclusion}, and \ref{mod: domination conditions} be satisfied and
suppose $\zeta >0$. Then, for all $i,j=1,\ldots ,m$, 
\begin{equation}
\sup_{\theta \in \Theta ^{\circ }}\left\vert \frac{\partial ^{2}\mathbb{Q}%
_{n}}{\partial \theta _{i}\partial \theta _{j}}(\theta )-\frac{\partial ^{2}Q%
}{\partial \theta _{i}\partial \theta _{j}}(\theta )\right\vert =o_{\mathbb{P%
}}(1)\quad \text{as }n\rightarrow \infty .  \label{diff_second_der}
\end{equation}
\end{prop}

\begin{proof}
Let $b<\infty $ be a bound for all the integrals appearing in Assumption \ref%
{mod: domination conditions}. By Lemma~\ref{DerivativesOfQ_nAndQ} the
l.h.s.~of (\ref{diff_second_der}) is not larger than $2\zeta
^{-2}b(1+C_{t}D)\Vert \hat{p}_{n}-p_{\blacktriangle }\Vert _{\Omega }$,
which converges to $0$ in probability by Theorem~\ref{Theorem: consistency
of AML-estimators}(a). Measurability of the supremum in (\ref%
{diff_second_der}) follows from continuity of the second derivatives (Lemma %
\ref{DerivativesOfQ_nAndQ}) and separability of $\Theta ^{\circ }$.
\end{proof}

\begin{remark}
\normalfont If $\zeta =0$ the assertion of the preceding proposition still
holds true in outer probability under Assumptions \ref{dens:Element}, \ref%
{dens:StrictInequality}, \ref{mod:Inclusion}, and \ref{mod: domination
conditions}, if $\partial ^{2}\mathbb{Q}_{n}(\theta )/\partial \theta
\partial \theta ^{\prime }$ is interpreted as the zero matrix on the event
where $\inf_{x\in \Omega }\hat{p}_{n}(x)=0$.
\end{remark}

\end{document}